\documentclass[10pt,reqno]{amsart}
 \RequirePackage{amsthm,amsmath}
 \RequirePackage[numbers]{natbib}
 \RequirePackage[colorlinks,linkcolor=blue,citecolor=blue,urlcolor=blue]{hyperref}
 \usepackage{amssymb}
 \usepackage{anysize}
 \usepackage{hyperref}
 \usepackage{extarrows}
 \usepackage{multirow}
 \usepackage{stmaryrd}
\usepackage{newlfont}
 \usepackage{color}
 \usepackage{amsmath}
 \usepackage[french]{babel}
 \usepackage{enumitem}
 \usepackage{mathtools} 
\usepackage{dsfont} 

\allowdisplaybreaks
\usepackage{soul}

\usepackage{authblk}
\usepackage{changes}
\usepackage{lipsum}
 
\numberwithin{equation}{section}

\numberwithin{equation}{section}
\newtheorem{theorem}{Theorem}[section]
\newtheorem{lemma}[theorem]{Lemma}
\newtheorem{proposition}[theorem]{Proposition}
\newtheorem{definition}[theorem]{Definition}

\newtheorem{assumption}[theorem]{Assumption}

\renewcommand{\Re}{\mathrm{Re}\,}
\renewcommand{\Im}{\mathrm{Im}\,}

\newcommand{\im}{\mathrm{Im}\,}

\newcommand{\E}{{\mathbb E }}
\newcommand{\R}{{\mathbb R }}
\newcommand{\N}{{\mathbb N}}

\renewcommand{\P}{{\mathbb P}}
\newcommand{\C}{{\mathbb C}}
\newcommand{\mr}{\mathfrak m}
\newcommand{\ii}{\mathrm{i}}
\newcommand{\deq}{\mathrel{\mathop:}=}

\newcommand{\dd}{\mathrm{d}}
\newcommand{\ie}{\emph{i.e., }}
\newcommand{\eg}{\emph{e.g., }}

\newcommand{\g}{\mathfrak{g}}
\newcommand{\bs}{\boldsymbol}

\def\ea{e_{0}(\lambda)}
\def\Tr{\mathrm{Tr}}
\def\i{\mathrm{i}}

\def\pzab{\frac{\partial^2}{\partial z_1 \partial z_2}}
\def\pza{\frac{\partial}{\partial z_1}}
\def\pzb{\frac{\partial}{\partial z_2}}
\def\pzaa{\frac{\partial}{\partial \overline{z_1}}}
\def\pzbb{\frac{\partial}{\partial \overline{z_2}}}
\def\pz{\frac{\partial}{\partial z}}
\def\pzp{\frac{\partial}{\partial z'}}
\def\pzz{\frac{\partial}{\partial \overline{z}}}
\def\pzzp{\frac{\partial}{\partial \overline{z'}}}
\def\px{\frac{\partial}{\partial x}}
\def\py{\frac{\partial}{\partial y}}

\def\tf{\tilde{f}}
\def\tg{\tilde{g}}

\renewcommand{\mathbf}[1]{\bs{#1}}
 
\marginsize{1in}{1in}{1in}{1in}

\begin{document}

\begin{center}

 \begin{minipage}{0.85\textwidth}
 \vspace{2.5cm}
 
\begin{center}
\large\bf
Central limit theorem for mesoscopic eigenvalue statistics of deformed Wigner matrices and sample covariance matrices
\end{center}
\end{minipage}
\end{center}

\renewcommand{\thefootnote}{\fnsymbol{footnote}}	
\vspace{1cm}
\begin{center} \begin{minipage}{1.1\textwidth}

 \begin{minipage}{0.33\textwidth}
\begin{center}
Yiting Li  \\
		\footnotesize 
		{KTH Royal Institute of Technology}\\
{\it yitingl@kth.se}
\end{center}
\end{minipage}
\begin{minipage}{0.33\textwidth}
	\begin{center}
		Kevin Schnelli\footnotemark[3]\\
		\footnotesize 
		{KTH Royal Institute of Technology}\\
		{\it schnelli@kth.se}
	\end{center}
\end{minipage}
\begin{minipage}{0.33\textwidth}
\begin{center}
Yuanyuan Xu\footnotemark[2]\\
\footnotesize 
{KTH Royal Institute of Technology}\\
{\it yuax@kth.se}
\end{center}
\end{minipage}
\end{minipage}
\end{center}

\bigskip
\footnotetext[2]{Supported by  the G\"oran Gustafsson Foundation and the Swedish Research Council Grant VR-2017-05195.}
\footnotetext[3]{Supported in parts by the Swedish Research Council Grant VR-2017-05195.}

\renewcommand{\thefootnote}{\fnsymbol{footnote}}	

\vspace{1cm}

\begin{center}
 \begin{minipage}{0.83\textwidth}\footnotesize{
 {\bf Abstract.}
 We consider $N$ by $N$ deformed Wigner random matrices of the form $X_N=H_N+A_N$, where $H_N$ is a real symmetric or complex Hermitian Wigner matrix and $A_N$ is a deterministic real bounded diagonal matrix. We prove a universal Central Limit Theorem for the linear eigenvalue statistics of~$X_N$ for all mesoscopic scales both in the spectral bulk and at regular edges where the global eigenvalue density vanishes as a square root. The method relies on  studying the characteristic function of the linear statistics~\cite{character} by using the cumulant expansion method, along with local laws for the Green function of $X_N$~\cite{isotropic3, isotropic, locallaw} and analytic subordination properties of the free additive convolution~\cite{global2,global}. We also prove the analogous results for high-dimensional sample covariance matrices.
 }
\end{minipage}
\end{center}

\vspace{1em}
 
\begin{center}
 \begin{minipage}{0.83\textwidth}\footnotesize{
 {\bf R\'esum\'e.}
 Nous  consid\'erons des matrices al\'eatoires de Wigner d\'eform\'ees de taille $N$ de la forme $X_N = H_N + A_N$,
o\`u $H_N$ est une matrice hermitienne de Wigner sym\'etrique ou complexe r\'eelle, et $A_N$ est une 
matrice diagonale d\'eterministe avec des entr\'ees r\'eelles et born\'ees. Nous prouvons un th\'eor\`eme de limite centrale universel pour les statistiques lin\'eaires des valeurs propres de $X_N$ pour toutes les \'echelles m\'esoscopiques \`a la fois dans le centre de spectre et aux bords r\'eguliers o\`u
la densit\'e globale des valeurs propres disparait sous forme de racine carr\'ee. La m\'ethode repose sur l'\'etude de la fonction caract\'eristique des statistiques lin\'eaires~\cite{character} en utilisant la m\'ethode des cumulants, ainsi que les lois locales
pour la fonction de Green de $X_N$~\cite{isotropic3, isotropic, locallaw} et les propri\'et\'es de subordination analytique de la convolution libre additive~\cite{global2,global}. Nous prouvons \'egalement les r\'esultats analogues pour des matrices de corr\'elation empirique de haute dimension.
 }
\end{minipage}
\end{center}

 \vspace{5mm}
 
 {\small
\footnotesize{\noindent\textit{Date}: \today}
 
 \vspace{2mm}

 }

\thispagestyle{headings}

\section{Introduction}

\subsection{Linear eigenvalue statistics of Wigner matrices}

A Wigner matrix $H_N$ is an $N \times N$ real symmetric or complex Hermitian random matrix with independent entries up to the constraint $H_N=H_N^*$. In the case the entries are Gaussian random variables, these matrices belong to the Gaussian Orthogonal Ensemble (GOE), Gaussian Unitary Ensemble (GUE), respectively. Wigner \cite{wigner_annals} proved the semicircle law stating that the empirical eigenvalue distribution of $H_N$ converges to the semicircle distribution with density $\rho_{sc}(x)=\frac{1}{2 \pi} \sqrt{4-x^2} \mathbf{1}_{[-2,2]}$. That is, for any test function $f \in C_c(\R)$,
\begin{align*}
\frac{1}{N}\sum\limits_{i=1}^Nf(\lambda_i) \to \int_{\R}f(x) \rho_{sc}(x) \dd x\quad\mbox{as }N\to \infty, 
\end{align*}
in probability, which can be understood as a Law of Large Numbers.

Johansson \cite{Johansson} obtained the corresponding Central Limit Theorem (CLT) for such linear eigenvalue statistics of the GUE, i.e.,
\begin{align}\label{linear_statistics_for_Wigner}
\sum\limits_{i=1}^Nf(\lambda_i)- N \int_{\R} f(x)\rho_{sc}(x)\dd x
\end{align}
converges in distribution to a centered Gaussian random variable. Strikingly different from the classical CLT, the linear statistics need not be normalized by $N^{-\frac{1}{2}}$, which is a manifestation of the strong eigenvalue correlations.  Bai and Yao \cite{baiwigner} used a martingale method to prove such CLTs for Wigner matrices and analytic test functions. Lytova and Pastur \cite{lytova+pastur}, and Shcherbina \cite{M.Shcherbina} improved these results by weakening the regularity conditions on the test functions. More recently, Sosoe and Wong~\cite{SosoeWong} obtained CLTs for Wigner matrices with $H^{1+\epsilon}$ test functions using Littlewood--Paley decompositions.

Boutet de Monvel and Khorunzhy initiated the study of mesoscopic linear eigenvalue statistics, i.e., the derivation of Gaussian fluctuations for the random variable 
\begin{align}\label{eq1}
\sum\limits_{i=1}^Nf\Big(\frac{\lambda_i-E_0}{\eta_0}\Big)-\mathbb E\Big[\sum\limits_{i=1}^Nf\Big(\frac{\lambda_i-E_0}{\eta_0}\Big)\Big],
\end{align}
with fixed $E_0 \in (-2,2)$ on mesoscopic scales $N^{-1} \ll \eta_0 \ll 1$. In \cite{meso1,meso2}, they obtained CLTs for the test function $(x-\ii)^{-1}$ on all mesoscopic scales for the GOE, and  $N^{-\frac{1}{8}} \ll \eta_0 \ll 1$ for symmetric Wigner matrices, respectively. Lodhia and Simm \cite{mesowigner} extended the CLT for complex Wigner matrices and general test functions on scales $N^{-1/3}\ll\eta_0\ll1$. He and Knowles \cite{moment} used moment estimates for Green functions to prove the CLTs for all arbitrary Wigner matrices on the optimal scales $N^{-1}\ll\eta_0\ll1$. More recently, Landon and Sosoe \cite{character} obtained similar CLT by means of the characteristic function.

Mesoscopic central limit theorems are important tools in the theory of homogenization of Dyson Brownian motion (DBM) introduced by Bourgade, Erd\H{o}s, Yau and Yin~\cite{BEYY} to prove fixed energy universality of local eigenvalue statistics of Wigner matrices.  Landon, Sosoe and Yau \cite{character2} subsequently derived a mesoscopic CLT to show fixed energy universality of the DBM. The dynamical approach using Dyson Brownian motion to prove the universality of the eigenvalue statistics on microscopic scale for all symmetry classes was initiated by Erd\H{o}s, Schlein and Yau in~\cite{ESY}; we refer to the surveys~\cite{bull_american_math_society,book} for further details.  Mesoscopic central limit theorems combined with DBM were used by Landon and Sosoe~\cite{character} and by Bourgade and Mody~\cite{logarithm} to derive Gaussian fluctuations of single eigenvalues, and in~\cite{logarithm, Bourgadegap} to show Gaussian fluctuations of the determinant of Wigner matrices.

Mesoscopic CLTs can also be studied at the spectral edges, where the mesoscopic scales are $N^{-\frac{2}{3}} \ll \eta_0 \ll 1$. For the GUE, Basor and Widom \cite{basor+widom} used asymptotics of the Airy kernel to prove mesoscopic CLTs at the edges. Min and Chen \cite{Min+Chen} subsequently considered edge CLTs for the GOE. Recently, Adhikari and Huang \cite{huang} obtained the mesoscopic CLT at the edges down to the optimal scale $\eta_0 \gg N^{-\frac{2}{3}}$ for the DBM.

\subsection{Deformed Wigner matrices}

In the present paper we are interested in deformed Wigner matrices. A deformed Wigner matrix is an $N \times N$ random matrix of the form
\begin{align}\label{deformed_Wigner_matrix}
X_N=H_N+A_N,
\end{align}
where $H_N$ is a real symmetric or complex Hermitian Wigner matrix and $A_N$ is a real deterministic diagonal matrix. It is also known as the Rosenzweig-Porter model in the physics literature. Suppose the empirical eigenvalue distribution of $A_N$ has a deterministic limiting measure, denoted by $\mu_{\alpha}$. It was shown by Pastur~\cite{solution} that the empirical eigenvalue distribution of $X_N$ converges weakly in probability to the free additive convolution of $\mu_{sc}$ and $\mu_\alpha$, denoted by $\mu_{fc}=\mu_{sc} \boxplus \mu_{\alpha}$; see also~\cite{Voi86}.

A CLT for the linear eigenvalue statistics with test functions in $C^2_c(\R)$ was obtained by Ji and Lee~\cite{global} under a one-cut assumption on $\mu_{fc}$. They also computed the expectation and variance in terms of~$\mu_{\alpha}$. Dallaporta and Fevrier \cite{global2} obtained the CLT for general~$\mu_{fc}$. Their results are summarized in Theorem~\ref{global} below.

In the present paper, we study the fluctuations of the linear eigenvalue statistics~\eqref{eq1} in the mesoscopic regime. We choose $\mu_{\alpha}$ properly such that the free additive convolution $\mu_{fc}$ is supported on a single interval and vanishes as a square root at the end-points. This edge behavior of the limiting eigenvalue distribution is quite common in random matrix theory, and sometimes referred to as regular edge. Denoting $\kappa_0=\kappa_0(E_{0})$ the distance from $E_0$ to the closest edge of the free convolution measure, we derive a CLT at energy~$E_0$ on scales $\eta_0$ with $N^{-1} \ll\eta_0\sqrt{\eta_0+\kappa_0}\leq 1$; see Theorem~\ref{thm:weak_convergence}. This range of $\eta_0$ covers the global scale as well as all mesoscopic scales up to the spectral edges. For energies $E_0$ in the bulk and at the edges respectively, we compute the variances and biases explicitly on the mesoscopic scales, where we recover the formulas for the Gaussian ensembles. This shows the expected universality of the linear eigenvalue fluctuations on mesoscopic scales. The universality of the eigenvalue statistics on the microscopic scale was derived for the deformed GUE in the bulk~\cite{edgesolution} and at the edge~\cite{edgesolution2}. For deformed Wigner matrices, the local bulk universality was obtained in~\cite{O;rourke_Vu} for a special class of $A_N$ using moment matching, under a one cut assumption in~\cite{bulk} and in~\cite{SchnelliErdoes,LandonYau} for the general case using the DBM methods. The edge universality was derived in~\cite{edge} using a Green function comparison method, and in~\cite{bulk, LandonYau_edge} using DBM. More recently, quantum unique ergodicity for deformed Wigner matrices was derived in \cite{quantum}.

In the proof of the main results, we follow the idea of~\cite{character} to compute the characteristic function of~\eqref{eq1} in combination with the Helffer-Sj\"ostrand formula and cumulant expansions; see \eqref{cumulant} below. Cumulant expansions were used in e.g.~\cite{meso2, moment, character} to study the linear eigenvalue statistics of random matrices. We also rely on local laws for Green functions \cite{isotropic3, isotropic, locallaw} and analytic subordination for the free convolution measure, as used in~\cite{global2,global,locallaw}.

On the global scale, the method used to derive the CLT for deformed Wigner matrices ~\cite{global2}  is insensitive to the behavior of the free convolution measure $\mu_{fc}$. An interesting aspect of the free additive convolution and deformed Wigner matrices is that the densities may show other edge behaviors than square roots. For such setups, one expects mesoscopic CLTs at the edges with different scalings, variances and biases. This is a main motivation for us to study linear eigenvalue statistics at spectral edges. The local eigenvalue statistics at such critical edges are only partly understood, see~e.g.~ \cite{gumbel, eigenvector} for some results. At cusp points in the interior of the bulk spectrum the universality of the local eigenvalue fluctuations was recently proved in~\cite{Cusp2, Cusp1}.

\subsection{Sample covariance matrix}
Sample covariance matrices form another class of archetypal random matrix models, with applications in multivariate statistical analysis. We consider the separable sample covariance matrices of the form $H=\Sigma^{1/2} X X^* \Sigma^{1/2}$, where $X$ is a $M \times N$ matrix with independent random variables, and $\Sigma^{1/2}$ is the square root of the $M \times M$ diagonal and positive definite matrix $\Sigma$. The dimensionality $M$ is chosen to be proportional to the sample size $N$. Assuming that the eigenvalue distribution of $\Sigma$ has a deterministic limit $\mu_\sigma$, it was proved by Marchenko and Pastur \cite{M+Pastur} that the spectral measure of $H$ approaches a deterministic probability measure. In the null case $\Sigma=I$, the limiting measure is called the Marchenko-Pastur distribution, $\mu_{MP}$. For the non-null case $\Sigma \neq I$, the limiting measure is given by the free multiplicative convolution of $\mu_{MP}$ and $\mu_\sigma$, denoted by $\mu_{MP} \boxtimes \mu_\sigma$, see \cite{freetimes2, freetimes, baibook}. A CLT for the fluctuations of the linear eigenvalue statistics was first studied by Jonsson \cite{Jonsson} for Wishart matrices where $X$ has Gaussian entries. CLTs for linear eigenvalue statistics with analytic test functions for general sample covariance matrices were then studied by Bai and Silverstein in \cite{baisample}. The regularity condition on the test functions was weakened by \cite{baisample2, lytova+pastur, M.Shcherbina} for the null case and \cite{Najim+Yao} for the non-null case. In the second part of this paper, we extend the techniques to derive corresponding CLTs for the mesoscopic eigenvalue statistics of sample covariance matrices; see Theorem \ref{meso_sample}.

\subsection{Related models}
Deformed Wigner matrices are closely related to Dyson Brownian motion, for which mesoscopic CLTs were obtained inside the bulk \cite{Duits+Johansson,character2,Huang+Landon} and at the regular edges \cite{huang}. The mesoscopic linear statistics were also studied for random band matrices \cite{erdos+knowles,erdos+knowles2}, sparse Wigner matrices~\cite{sparse}, mesoscopic eigenvalue density correlations for Wigner matrices~\cite{HeKnowlesDensityCorrelation}, invariant $\beta$-ensembles \cite{Bekerman+lodhia} and orthogonal polynomial ensembles \cite{Breuer+Duits}. The global fluctuations of the deformed GOE/GUE can also be studied using the framework of second order freeness \cite{secondfreeness}.

\subsection{Structure of this paper}

Section \ref{sec:model_and_results} contains the precise definitions, assumptions and the main results. The proof the main theorem is carried out in Section \ref{sec:strategy}-\ref{sec:proof_of_main_lemma}. In Section \ref{sec:proof_of_main_thm} and \ref{sec:expectation}, we compute the variances and the biases in the bulk and at the edges. In Section \ref{sec:sample_covariance}, we consider sample covariance matrices and obtain the corresponding results. Some auxiliary results are proved in the Appendices.

\subsection{Notation}  We denote the upper half-plane by $\C^+\deq\{z\in\C\,:\,\im z>0\}$ and the positive real line by $\R^+\deq\{x\in\R\,:\,x>0\}$. For any vector $v \in \C^{N}$, we use $\|v\|_2$ to denote the Euclidean norm. For a matrix $A \in \C^{N \times N}$, we denote by $\|A\|_{\mathrm{op}}$ its operator norm induced by the Euclidean vector norm. We use~$c,k$ and~$C,K$ to denote strictly positive constants that are independent of $N$. Their values may change from line to line. We use standard big O and small o notations. For $X, Y \in \R$, we write $X \sim Y$  if there exist constants $c, C>0$ such that $c |Y| \leq |X| \leq C |Y|$. We write $X \ll Y$ if there exists a small $\tau>0$ such that $|X| \leq N^{-\tau} |Y|$ for large $N$. We will use the following definition on high-probability estimates from~\cite{EKY}. 
\begin{definition}\label{definition of stochastic domination}
Let $\mathcal{X}\equiv \mathcal{X}^{(N)}$ and $\mathcal{Y}\equiv \mathcal{Y}^{(N)}$ be two sequences of nonnegative random variables. We say~$\mathcal{Y}$ stochastically dominates~$\mathcal{X}$ if, for all (small) $\epsilon>0$ and (large)~$D>0$,
\begin{align}\label{le proba est}
\P\big(\mathcal{X}^{(N)}>N^{\epsilon} \mathcal{Y}^{(N)}\big)\le N^{-D},
\end{align}
for sufficiently large $N\ge N_0(\epsilon,D)$, and we write $\mathcal{X} \prec \mathcal{Y}$ or $\mathcal{X}=O_\prec(\mathcal{Y})$.
\end{definition}
We often use the notation $\prec$ also for deterministic quantities, then~\eqref{le proba est} holds with probability one. Stochastic domination has the following properties.
\begin{lemma}\label{dominant}(Proposition 6.5 in \cite{book})
	\begin{enumerate}
		\item $X \prec Y$ and $Y \prec Z$ imply $X \prec Z$;
		\item If $X_1 \prec Y_1$ and $X_2 \prec Y_2$, then $X_1+X_2 \prec Y_1+Y_2$ and $X_1X_2 \prec Y_1Y_2;$
		\item  If $X \prec Y$, $\E Y \geq N^{-c_1}$ and $|X| \leq N^{c_2}$ almost surely with fixed constants $c_1$ and $c_2$, then we have $\E X \prec \E Y$.
	\end{enumerate}
\end{lemma}

\section{Model and main results}\label{sec:model_and_results}

\subsection{Model and assumptions}

Let $H \equiv H_N$ be an $N\times N$ real or complex Wigner matrix satisfying the following assumption.
\begin{assumption}\label{assumption_1_deformed}
	For a real $(\beta=1)$ symmetric Wigner matrix $H$ we assume that:
	\begin{enumerate}
		\item $\{H_{ij}|i\le j\}$ are independent real-valued centered random variables with $H_{ij}=H_{ji}$.
		\item  For $i \neq j$, $\E [(\sqrt{N} H_{ij})^2]=1$; $\E [(\sqrt{N} H_{ii})^2]=m_2$ for some constant $m_2>0$. In addition, $\E [(\sqrt{N} H_{ij})^4]=W_4$ for some constant $W_4>0$.
		\item All entries have uniform sub-exponential decay, that is, there exist $C_0>0$ and $\theta>1$ such that
		\begin{align}\label{moment condition}\P\Big( |\sqrt{N} H_{ij}| \geq x \Big) \leq C_0 e^{-x^{\frac{1}{\theta}}},\quad\forall i,j\,.\end{align}
		In particular, we have 
		\begin{align}\label{moment condition finite}
		\E[|\sqrt{N} H_{ij}|^p] \leq C (\theta p)^{\theta p}~ (p \geq 3)\,.\end{align}
	\end{enumerate}
	For complex $(\beta=2)$ Hermitian Wigner matrix we assume that:
	\begin{enumerate}
		\item $\{{\Re}H_{ij},{\Im}H_{ij}|i\le j\}$ are independent centered real-valued random variables with $H_{ij}=\overline{H_{ji}}$.
		\item For $i \neq j$, $\E [H_{ij}^2]=0$ and $\E [(\sqrt{N} |H_{ij}|)^2]=1$; $\E [(\sqrt{N} |H_{ii}|)^2]=m_2$ for some constant $m_2>0$. In addition, $\E [(\sqrt{N} |H_{ij}|)^4]=W_4$ for some constant $W_4>0$.
		\item The sub-exponential tail assumption in~\eqref{moment condition} holds.
	\end{enumerate}
\end{assumption}

Let $\{A_N\}=\mathrm{Diag}(a_i)$ be a sequence of real deterministic diagonal $N \times N$ matrices with $\|A\|_{\mathrm{op}}$ uniformly bounded in $N$. The empirical spectral measure of $A_N$ is defined by $\mu_A\deq\frac{1}{N} \sum_{i=1}^{N} \delta_{a_i}$.

For a probability measure $\nu$ on $\R$ denote by $m_\nu$ its Stieltjes transform, i.e.
\begin{align}
 m_\nu(z)\deq\int_\R\frac{\dd\nu(x)}{x-z}\,,\qquad z\in\C^+\,.
\end{align}
Note that $m_{\nu}\,:\C^+\rightarrow\C^+$ is analytic and can be analytically continued to the real line outside the support of $\nu$. Moreover, $m_{\nu}$ satisfies $\lim_{\eta\nearrow\infty}\ii\eta {m_{\mu}}(\ii \eta)=-1$. Conversely, if $m\,:\,\C^+\rightarrow\C^+$ is an analytic function with $\lim_{\eta\nearrow\infty}\ii\eta m(\ii\eta)=-1$, then $m$ is the Stieltjes transform of a probability
measure $\nu$, \ie $m(z) = m_{\nu}(z)$, for all $z\in\C^+$; see \eg~\cite{Ak}.

The following assumption ensures the existence of the weak limiting measure of $\mu_A$.

\begin{assumption}\label{assumption_2_deformed}
	There exists a deterministic and compactly supported probability measure denoted as~$\mu_{\alpha}$, such that $\mu_A$ converges weakly to $\mu_{\alpha}$. In addition, there exists $\alpha_0>0$ such that for any fixed compact set $D_0 \subset \C^+\cup\R$ with $D_0 \cap \mathrm{supp} (\mu_{\alpha}) =\emptyset$, 
	\begin{equation}\label{levy}
	\max_{z \in D_0} \Big| \int_\R \frac{\dd \mu_{A}(x)}{x-z} -\int_\R \frac{\dd \mu_{\alpha}(x)}{x-z} \Big| =O(N^{-\alpha_0}),
	\end{equation}
	for sufficiently large $N$.
	\end{assumption}

Define the deformed Wigner matrix as
$$X_N\deq A_N+H_N.$$
The eigenvalues of $X_N$ are denoted as $\lambda_i \in \R$. The empirical spectral measure of $X_N$ is defined by~$\mu_N(x)=\frac{1}{N} \sum_{i=1}^N \delta_{\lambda_i}$. For $z\in\C^+$, we introduce the Green function, $G(z)$, and its normalized trace as
$$G(z)\deq (X_N-zI)^{-1}\qquad m_N(z)\deq\frac{1}{N} \Tr\, G(z)=\int_{\R} \frac{\dd \mu_N(\lambda)}{\lambda-z}\,,$$
i.e., $m(z)\equiv m_N(z)$ is the Stieltjes transform of $\mu_N$.

The empirical spectral distribution $\mu_N$ converges as $N$ tends to infinity to the free additive convolution of $\mu_\alpha$ and the standard semicircle law, denoted by $\tilde \mu_{fc}:=\mu_\alpha \boxplus \mu_{sc}$. The free convolution measure can be described by analytic subordination~\cite{Bia98, Voi93}: Its Stieltjes transform, $\tilde{m}_{fc}$, is the unique solution to the Pastur equation
	\begin{equation}\label{self6}
	\tilde{m}_{fc}(z)=\int_\R \frac{1}{a-z-\tilde{m}_{fc}(z)} \dd \mu_{\alpha}(a)\,,
	\end{equation}
subject to the constraint $\im \tilde m_{fc}(z)>0$, $z\in\C^+$. 

Since the convergence speed in~\eqref{levy} can be very slow, we work with a finite $N$ version of the free convolution measure. Let $\mu_{fc}:=\mu_{A} \boxplus \mu_{sc}$. The Stieltjes transform of $\mu_{fc}$, denoted by $m_{fc}$, is hence the unique solution to
 \begin{equation}\label{self}
	m_{fc}(z)=\frac{1}{N} \sum_{i=1}^N \frac{1}{a_i-z-m_{fc}(z)}\,,
	\end{equation}
such that $\Im m_{fc}(z)>0$, $z\in\C^+$. Note that $\mu_{fc}$ depends on $N$, but is deterministic.

 Biane~\cite{free2} proved that $\widetilde\mu_{fc}$ and $\mu_{fc}$ are absolutely continuous probability measures whose densities, are analytic wherever positive. We denote the density functions by $\widetilde\rho_{fc}$ and $\rho_{fc}$. In general the measures $\rho_{fc}$ and $\widetilde\rho_{fc}$ are supported on several disjoint intervals and may have irregular edges where the densities do not vanish as a square root or have cusp points inside the support. The following assumption will rule out such scenarios.

\begin{assumption}\label{assumption2.3}
	Let $\mathcal{I}$ be the smallest interval that contains the support of $\mu_{\alpha}$, and assume that
	\begin{equation*}
	\inf_{x \in\mathcal{I}} \int_{\R} \frac{\dd \mu_{\alpha}(a)}{(a-x)^2} \geq 1+ w,
	\end{equation*}
	for some constant $w>0$ (the left side may be infinite). Similarly, let $\hat{\mathcal{I}}$ be the smallest interval that contains the support of $\mu_{A}$, and assume that
	$$\inf_{x \in \hat{\mathcal{I}}} \int_{\R} \frac{\dd \mu_{A}(a)}{(a-x)^2} \geq 1+ w,$$
	for sufficiently large $N$.
\end{assumption}
The above assumption ensures that the density functions $\rho_{fc}$ and $\tilde{\rho}_{fc}$  are supported on a single interval (for $N$ sufficiently large) and vanish as square roots at the endpoints of the support.

\begin{lemma}(Lemma 2.4, 3.2 and 3.5 in \cite{bulk})
	Under Assumption~\ref{assumption2.3}, there exists $\tilde{L}_-$ and $\tilde L_+ \in \R$, such that $\mathrm{supp}\,\tilde \rho_{fc} =[\tilde L_-,\tilde L_+]$, and $\tilde \rho_{fc}$ is strictly positive in $(\tilde L_-,\tilde L_+)$. Moreover, there exists $C>1$ such that
	$$C^{-1} \sqrt{\tilde \kappa} \leq \tilde{\rho}_{fc}(E) \leq C \sqrt{\tilde \kappa}\,,\qquad\qquad E\in[\tilde L_-,\tilde L_+]\,,$$
	where $\tilde \kappa\deq\min \{ |E-\tilde{L}_-|, |E-\tilde L_+|\}.$ The endpoints $\tilde{L}_\pm$ are the two real solutions to the equation 
	\begin{align}\label{L}
\int_\R \frac{\dd\mu_\alpha(a)}{(a-\widetilde L_\pm-m_{fc}(\widetilde L_\pm))^2} =1\,. 
	\end{align}
	The  same holds true, for sufficiently large $N$, if we replace $\mu_\alpha$, $\tilde\rho_{fc}$, $\tilde L_{\pm}$ and $\tilde\kappa$ by $\mu_A$, $\rho_{fc}$, $  L_{\pm}$ and~$\kappa$, respectively. Here $[L_-,L_+]$ is the support of $\rho_{fc}$ and $\kappa\deq\min \{ |E-{L}_-|, |E- L_+|  \}.$ 
	\end{lemma}

\subsection{Local law for the deformed Wigner matrices}

We introduce the spectral domain, 
\begin{equation}\label{ddd}
D':=\Big\{z=E+\ii \eta:  |E| \leq M, N^{-1+c} \leq \eta \leq  3 \Big\},
\end{equation}
where $M>1+\max\{ |\tilde{L}_-|, |\tilde L_+|\}$ and $c>0$ is small. Define the deterministic control parameters

\begin{equation}\label{control}
\Psi(z):=\sqrt{ \frac{\im m_{fc}(z)}{N |\eta|}} +\frac{1}{N |\eta|}\,,\qquad\Theta(z):=\frac{1}{N |\eta|}\,, \qquad z=E+\ii \eta \in \C \setminus \R.
\end{equation}
Using (\ref{1}), (\ref{12}) in Lemma \ref{previous} below, we have 
$$C N^{-\frac{1}{2}} \leq \Psi(z) \ll 1\,,\qquad\quad  z \in D'\,.$$
The following local law for the Green function was proved in \cite{locallaw}.
\begin{theorem}\label{locallaw}(Local law for the deformed Wigner matrix, Theorem 2.10 in \cite{locallaw})
	Under the Assumptions 2.1-2.3, the following holds
\begin{align}\label{G}
	\max_{ij}\Big| G_{ij}(z)- \delta_{ij}\frac{1}{a_i-z-m_{fc}(z)} \Big| \prec \Psi(z), \quad \Big| N^{-1} \Tr G(z)-m_{fc}(z) \Big| \prec \Theta(z),
\end{align}
uniformly for $z\in D'$.
\end{theorem}

 The local law gives  strong rigidity estimates for the eigenvalues of $X_N$. It also gives an upper bound, up to factors of $N^\epsilon$, on the size of the fluctuations $\Tr G(z) -\E \Tr G(z) $. It is hence natural to study the fluctuations of $\Tr G(z) -\E \Tr G(z)$. The CLT for the linear eigenvalue statistics for general test functions is proved in \cite{global2} and \cite{global} on global scale when $\im z$ is order one. Via the Helffer-Sj\"ostrand functional calculus, a CLT for the resolvent can be translated to a CLT for the linear statistics.

\begin{theorem}[Theorem 2.15 of \cite{global}]\label{global}	
Under the Assumptions \ref{assumption_1_deformed}-\ref{assumption2.3}, for any $\varphi \in C_c(\R)$ which is analytic on a neighborhood of $[\tilde{L}_-, \tilde{L}_+]$, the random variable $\sum_{i=1}^N \varphi(\lambda_i)-N \int_{\R} \varphi(x) \rho_{fc}(x) \dd x$ converges in distribution to the Gaussian random variable with mean $M(\varphi)=-\frac{1}{2 \pi \ii} \int_{\Gamma} \varphi(z) \tilde b(z) \dd z$, and variance
 $V(\varphi)=\frac{1}{(2 \pi \ii)^2} \int_{\Gamma} \int_{\Gamma} \varphi(z_1) \varphi(z_2) \tilde K(z_1,z_2) \dd z_1 \dd z_2$, where 
 $$\tilde b(z):=\frac{\tilde{m}''_{fc}(z)}{2(1+\tilde{m}'_{fc}(z))^2} \Big( (m_2-1)+\tilde{m}'_{fc}(z)+(W_4-3)\frac{\tilde{m}'_{fc}(z)}{1+\tilde{m}'_{fc}(z)}  \Big),$$ 
 and $\tilde K(z_1,z_2):=(m_2-2) \frac{\partial^2 \tilde{I} }{\partial z_1 \partial z_2}+(W_4-3) \Big( \tilde{I} \frac{\partial^2 \tilde{I} }{\partial z_1 \partial z_2} +\frac{\partial \tilde{I} }{\partial z_1} \frac{\partial \tilde{I} }{\partial z_2} \Big)+\frac{2}{(1-\tilde{I})^2} \Big( \frac{\partial \tilde{I} }{\partial z_1} \frac{\partial \tilde{I} }{\partial z_2} +(1-\tilde{I}) \frac{\partial^2 \tilde{I} }{\partial z_1 \partial z_2}  \Big), $ with 
 $$\tilde{I}(z_1,z_2)\deq\int_\R \frac{1}{(x-z_1-\tilde{m}_{fc}(z_1))(x-z_2-\tilde{m}_{fc}(z_2))} \dd \mu_{\alpha}(x).$$
	Here $\Gamma$ is a rectangular contour with vertices $(a_{\pm}\pm\ii v_0)$ so that $\pm(a_{\pm}-\tilde{L}_{\pm})>0$ and $\Gamma$ lies within the analytic domain of $\varphi$.
\end{theorem}
Using ideas of M. Shcherbina \cite{M.Shcherbina}, the above result can be extended to $C^2_c(\R)$ test functions.  In \cite{global2}, the corresponding result was obtained for the multi-cut regime. 

\subsection{Main results}

Choose $E_0 \in [-1+\tilde L_-, 1+\tilde L_+]$ and $ N^{-1} \ll \eta_0 \ll 1$. Consider a test function $g \in C_c^2(\R)$ and set
\begin{equation}\label{fn}
f_{N}(x):=g\Big( \frac{x-E_0}{\eta_0}\Big).
\end{equation}
We will write $f_N$ as $f$ for notational simplicity. Define
\begin{equation}\label{kappa}
\kappa_0:=\mathrm{dist}(\mbox{supp}(f),\{L_+,L_-\}).
\end{equation}

Following \cite{lytova+pastur, character}, we study the characteristic function 
\begin{equation}\label{mme}
\phi(\lambda):=\E[e(\lambda)], \quad \mbox{where }e(\lambda):=\exp \Big\{ \i \lambda(\Tr f(X_N)-\E \Tr f(X_N)) \Big\}, \qquad \lambda \in \R.
\end{equation}
Let $\tau>0$ be an arbitrary small constant and define
\begin{equation}\label{domain}
\Omega_{0} := \Big\{ x + \i y \in \C\,:\, |y| \geq N^{-\tau} \eta_0  \Big\}\,.
\end{equation}
A key observation in \cite{character} is that working on $\Omega_0$ instead of all $\C$, effectively removes the ultra-local scales without affecting the mesoscopic linear statistics.

\begin{proposition}\label{prop}
	Let $X_N$ be a deformed Wigner matrix satisfying Assumptions \ref{assumption_1_deformed}-\ref{assumption2.3}. Let $\eta_0 \sqrt{\kappa_0+\eta_0} \geq N^{-1+c_0}$ for some $c_0>0$. Then there exists a small $0 < \tau< \frac{c_0}{16}$, such that the characteristic function~(\ref{mme}) satisfies
	\begin{align}\label{vf}
	\phi'(\lambda)=-\lambda \phi(\lambda) V(f)+\tilde{\mathcal{E}}, \qquad V(f):=\frac{1}{\pi^2} \int_{\Omega_{0}} \int_{\Omega_{0}} \pzaa \tf(z_1) \pzbb \tf(z_2) K(z_1,z_2) \dd^2z_1 \dd^2z_2,
	\end{align}
	where $\tf$ is an almost analytic extension of $f$ given in (\ref{tilde_f}) below and $\beta=1,2$ is the symmetry parameter. 
	The kernel $K$ is given by
	\begin{align}\label{kernel}
	K(z_1,z_2):= \pzab  \Big( (m_2-\frac{2}{\beta}) {I} +\frac{(W_4-1-\frac{2}{\beta})}{2} {I}^2\Big)+\frac{2}{\beta} \frac{\partial}{\partial z_1}\Big(\frac{1}{1- I}\frac{\partial  I}{\partial z_2}\Big),
	\end{align}
	with
	\begin{equation}\label{iii}
	{I}(z_1,z_2):=\int_\R \frac{1}{(x-z_1-{m}_{fc}(z_1))(x-z_2-{m}_{fc}(z_2))} \dd \mu_{A}(x),
	\end{equation}
	and the error $\tilde{\mathcal{E}}$ is bounded by
	$$|\tilde{\mathcal{E}}|=O_{\prec}\Big( |\lambda| \log N N^{- \tau} \Big)+O_{\prec}\Big( \frac{(1+|\lambda|^4)N^{3 \tau}}{N \eta_0 \sqrt{\kappa_0+ \eta_0}}\Big)+O_{\prec}\Big( \frac{(1+|\lambda|^4)N^{2 \tau}}{\sqrt{N \eta_0 \sqrt{\kappa_0+\eta_0}}}\Big),$$
	provided that $ V(f)=O(1)$. 
	\end{proposition}

Proposition \ref{prop} implies the following result.
\begin{theorem}\label{thm:weak_convergence}
	Under the same assumptions as in Proposition \ref{prop}, if we further assume that there exist $c,C>0$ such that $c \leq V(f) \leq C$ for sufficiently large $N$, then $\frac{\Tr f(X_N)-\E \Tr f(X_N)}{\sqrt{V(f)}}$ converges in distribution to a standard Gaussian random variable.
\end{theorem}

We remark that the Theorem \ref{thm:weak_convergence} applies to the global scale as well as the mesoscopic scales. The expectation of $\Tr f(X_N)$ has the following asymptotic expansion, which matches the result in \cite{global2, global} on the global scale.
\begin{proposition}\label{prop2} Under the same assumptions as in Proposition \ref{prop}, the so-called bias is given by
	\begin{equation}\label{bias_formula}
	\E \Tr f(X_N)-N \int_{\R} f(x) \rho_{fc}(x) \dd x=\frac{1}{2 \pi} \int_{\Omega_0} \pzz \tf(z) b(z)  \dd^2z +O(N^{-\tau})+O_{\prec}\Big( \frac{N^{2 \tau}}{\sqrt{N \eta_0 \sqrt{\kappa_0+\eta_0}}}\Big),
	\end{equation}
	where $\tf$ is given in (\ref{tilde_f}) below, and
	\begin{equation}\label{bz}
	b(z):=\Big(\frac{2}{\beta}-1\Big)\frac{1}{1-I_s(z)} \frac{\dd I_s(z)}{\dd z} +\Big(m_2-\frac{2}{\beta}\Big) \frac{\dd I_s(z)}{\dd z}+\Big(W_4-1-\frac{2}{\beta}\Big) I_s(z)\frac{\dd I_s(z)}{\dd z},
	\end{equation}
	with
	\begin{equation}\label{Is}
	{I_s}(z):=\int_{\R} \frac{1}{(x-z-{m}_{fc}(z))^2} \dd \mu_{A}(x).
	\end{equation}

\end{proposition}
 Note that the variance $V(f)$ in (\ref{vf}) and the bias in (\ref{bias_formula}) are N-dependent and their formulas depend explicitly on the free convolution measures. We will compute their limits on the mesoscopic scales as $N$ goes to infinity in order to obtain the following universal CLTs in the bulk and at the regular edges respectively.
\begin{theorem}\label{meso}(Mesoscopic CLT in the bulk)
	Let $X_N$ be a deformed Wigner matrix satisfying Assumptions \ref{assumption_1_deformed}-\ref{assumption2.3}. Let $ N^{-1+c} \leq \eta_0 \leq N^{-c}$ with some small $c>0 $, fix $E_0 \in (\tilde{L}_-,\tilde{L}_+)$ such that $\kappa_0 > c_0$, for some $c_0>0$ and large $N$. Then, for any test function $g \in C^2_c(\R)$, the linear statistics 
	\begin{equation}\label{linear_stat}
	\sum_{i=1}^N g \Big( \frac{\lambda_i-E_0}{\eta_0} \Big)-N \int_{\R} g \Big(\frac{x-E_0}{\eta_0} \Big) \rho_{fc}(x) \dd x
	\end{equation}
	converges in distribution to a Gaussian random variable with mean zero and variance
	\begin{align}\label{bulk_variance}
	\frac{1}{2 \beta \pi^2} \int_{\R}  \int_{\R}  \frac{(g(x_1)-g(x_2))^2}{(x_1-x_2)^2} \dd x_1 \dd x_2=\frac{1}{\beta \pi} \int_{\R} |\xi| |\hat{g}(\xi)|^2 \dd \xi,
	\end{align}
	where $\hat{g}(\xi):=(2 \pi)^{-1/2} \int_{\R} g(x) e^{-\i \xi x} \dd x$. In particular, the bias vanishes in the bulk regime.
\end{theorem}

\begin{theorem}\label{mesoedge}(Mesoscopic CLT at the edge)
	Let $X_N$ be a deformed Wigner matrix satisfying Assumptions \ref{assumption_1_deformed}-\ref{assumption2.3}. Let $N^{-\frac{2}{3}+c} \leq \eta_0 \leq N^{-c}$ with some small $c>0$.  For any function $g \in C^2_c(\R)$,
	the linear statistics (\ref{linear_stat}) with $E_0=L_+$
	converges in distribution to a Gaussian random variable with mean $\Big(\frac{2}{\beta}-1\Big) \frac{g(0)}{4}$ and variance
	\begin{align}\label{edge_variance}
	\frac{1}{4 \beta \pi^2}  \int_{\R}  \int_{\R}  \Big( \frac{g(-x^2)-g(-y^2)}{x-y} \Big)^2 \dd x \dd y=\frac{1}{2 \beta \pi} \int_{\R} |\xi| |\hat{h}(\xi)|^2 \dd \xi,
	\end{align}
	where $h(x)=g(-x^2)$ and $\hat{h}(\xi):=(2 \pi)^{-1/2} \int_{\R} h(x) e^{-\i \xi x} \dd x$. At the left edge $E_0=L_-$, we obtain a similar CLT with $h(x)=g(x^2)$.
\end{theorem}

{\it Remark:} The bulk variance (\ref{bulk_variance}) agrees with the GOE/GUE. For the edges, the bias and variance in~\eqref{edge_variance} coincide with those of the GUE/GOE obtained in~\cite{basor+widom, Min+Chen} and the Dyson Brownian motion in~\cite{huang}.

{\it Remark:} We remark that our assumption that the fourth moments of the off-diagonal entries are identical can easily  be relaxed in the above theorems. The regularity condition we impose on the test function $g$ is clearly not optimal, and we expect results can be extended to $C^{1,r,s}(\R)$ functions; see~\cite{moment}. The CLTs also hold true if we consider the resolvent test function $g(x)=\frac{1}{x-\ii}$.

Finally, for test functions in $C^2_c(\R)$, we can relax the single support condition for $\mu_{fc}$ by assuming instead that the cuts of the support of $\mu_{fc}$ are separated by order one and the density $\rho_{fc}$ has square-root decay near the edges.

\section{Proof of Proposition \ref{prop}}\label{sec:strategy}
In this section, we prove Proposition \ref{prop} by reducing it to the main technical result Lemma~\ref{lemma4}. Recall the scaled test function $f$ on scale $\eta_0$ from (\ref{fn}). There are constants such that
\begin{equation}\label{assumpf}
\|f\|_1 \leq C \eta_0; \qquad \|f'\|_1 \leq C'; \qquad \|f''\|_1 \leq \frac{C''}{\eta_0}.
\end{equation}

We use the Helffer-Sj\"ostrand formula to link $f(X_N)$ to the Green function of $X_N$.
\begin{lemma}\label{helffler}(Helffer-Sj\"ostrand formula) 
	Let $f \in C_c^2(\R)$ and $\chi(y)$ be a smooth cutoff function with support in $[-2,2]$, with $\chi(y)=1$ for $|y| \leq 1$. Define its almost-analytic extension 
	\begin{equation}\label{tilde_f}
	\tilde{f}(x+\ii y):=(f(x)+\ii y f'(x)) \chi(y).
	\end{equation}
	Then we have
	\begin{equation}\label{use_x}
	f(\lambda)=\frac{1}{\pi} \int_{\C} \frac{\pzz \tilde{f}(z)}{\lambda-z} \dd^2z=\frac{1}{2 \pi} \int_{\R^2} \frac{\ii y f''(x) \chi(y)+\ii \Big( f(x)+\ii y f'(x) \Big) \chi'(y)}{\lambda-x-\ii y} \dd x \dd y,
	\end{equation}
	where $z=x+\ii y$, $\pzz=\frac{1}{2}(\px+\ii \py)$, and $\dd ^2z$ is the Lebesgue measure on $\C$.
\end{lemma}

Therefore, we write
\begin{align}\label{fw}
\Tr f(X_N) -\E \Tr f(X_N) =\frac{1}{\pi} \int_{\C} \pzz \tf(z) (\Tr(G(z))-\E \Tr G(z))\dd^2z.
\end{align}
Plugging the above equation in $e(\lambda)$ given by (\ref{mme}), we have
\begin{equation}\label{e}
e(\lambda)=\exp\Big\{ \frac{\i \lambda}{\pi}\int_{\C}  \pzz \tf(z) (\Tr(G(z))-\E \Tr G(z)) \dd^2 z\Big\}.
\end{equation}
Taking the derivative of the characteristic function given in (\ref{mme}), and applying (\ref{e}), we get
\begin{equation}\label{phi}
\phi'(\lambda)=\frac{\i}{\pi}\int_{\C} \pzz \tf(z)   \E \Big[ e(\lambda)(\Tr(G(z))-\E \Tr G(z)) \Big]  \dd^2z.
\end{equation}

Following \cite{character}, we restrict the domain of the spectral parameter to $\Omega_{0}$, as the very local scales do not contribute to $\phi(\lambda)$.  We write
\begin{equation}\label{tr_split}
\Tr f(X_N) -\E \Tr f(X_N)=\frac{1}{\pi}  \Big( \int_{\Omega_{0}} +\int_{\Omega^c_{0}} \Big) \frac{\partial\tilde{f}(z)}{\partial\bar{z}} \big(\Tr(G(z))-\E \Tr G(z)\big) \dd^2z.
\end{equation}
Recall $\tilde f$ in (\ref{tilde_f}) and the definition of $\Omega_0$ in (\ref{domain}). Since $\chi(y)=1$ for $|y| \leq 1$, we can write the second integral in (\ref{tr_split}) with $z=x+\ii y$ as
\begin{align}
\frac{N}{  \pi} \int_{\R} \int_{0}^{\frac{\eta_0}{N^{\tau} }} \ii y f''(x) (m_N(z) -\E m_N(z)  ) \dd x \dd y&=-\frac{N}{ \pi}  \int_{\R} \int_{0}^{\frac{\eta_0}{N^{\tau} }}   y f''(x) \Im (m_N(z)-\E m_N(z) ) \dd x \dd y,
\end{align}
where we used the fact that $m_N(\overline{z})=\overline{m_N(z)}$. We now choose a small $\tau>0$ such that $ N^{-1} \ll y_0:=\sqrt{\frac{\eta_0}{N^{1+\tau}}} \leq N^{-\tau} \eta_0$. In the regime $y \in [y_0, N^{-\tau} \eta_0]$, the integral can be estimated using the local law (\ref{G}), (\ref{assumpf}) and Lemma~\ref{dominant}, i.e.,
\begin{align}
\Big|\frac{N}{ \pi}  \int_\R \int_{y_0}^{N^{-\tau} \eta_0}   y f''(x) \Im \Big(m_N(z)-\E m_N(z) \Big) \dd x \dd y\Big| \prec  \Big| \int_\R f''(x)\dd x \int_{y_0}^{N^{-\tau} \eta_0}  \dd y\Big|=O_{\prec}(N^{-\tau}).
\end{align}
In the regime $y \in [0, y_0]$, the local law is not sharp but instead we use the fact that $y \rightarrow \Im m_N(x+\ii y)y$ is increasing. That is,
\begin{align}
\Big|\frac{N}{\pi}  \int_\R \int_{0}^{y_0} y f''(x) \Im (m_N(z)-\E m_N(z) ) \dd x \dd y\Big| =O_{\prec}\Big(\frac{N y_0^2}{\eta_0}\Big)=O_{\prec}(N^{-\tau}).
\end{align}
Therefore, we have from (\ref{tr_split}) that
\begin{equation}\label{fw2}
\Tr f(X_N) -\E \Tr f(X_N)=\frac{1}{\pi}  \int_{\Omega_{0}} \pzz \tilde{f}(z)  (\Tr(G(z))-\E \Tr G(z)) \dd^2z+O_{\prec}(N^{-\tau}).
\end{equation}

Using the same argument, since $|e(\lambda)|=1$, we have
\begin{equation}\label{newphi}
\phi'(\lambda)=\frac{\i}{\pi}\int_{\Omega_0} \pzz \tf(z)   \E \Big[ e(\lambda)(\Tr(G(z))-\E \Tr G(z)) \Big]  \dd^2z+O_{\prec}(N^{-\tau}).
\end{equation}
Similarly, we restrict the integration domain of $e(\lambda)$ in (\ref{e}) to $\Omega_0$. Let
\begin{equation}\label{e2}
\ea:=\exp\Big\{ \frac{\i \lambda}{ \pi}   \int_{\Omega_{0}}  \pzz \tf(z) (\Tr(G(z))-\E \Tr G(z)) \dd^2 z  \Big\}.
\end{equation}
In addition, (\ref{fw2}) implies that $|e(\lambda)-\ea| = O_{\prec}\Big( |\lambda| N^{ -\tau}  \Big).$ We also have $|\ea|=1$, using $|e(\lambda)|=1$ and $\Tr G(\overline{z})=\overline{\Tr G(z)}$. If we further replace $e(\lambda)$ by $\ea$ in (\ref{newphi}), then we get
\begin{equation}\label{phi2}
\phi'(\lambda)=\frac{\i}{ \pi}\int_{\Omega_{0}} \pzz \tf(z)  \E \Big[ (\ea (\Tr(G(z))-\E \Tr G(z)) \Big] \dd^2 z+O_{\prec}\Big( |\lambda| \log N N^{ -\tau} \Big).
\end{equation}

The last error term on the right side, and many error terms below, are estimated using the following lemma, which is a variant of Lemma 4.4 in \cite{character}. The proof is provided in Appendix B.
\begin{lemma}\label{lemma3}
		Suppose $h(z)$ is a holomorphic function on $\Omega_0$ and $|h(z)| \leq \frac{K}{|\Im z|^s}$ for some constants $s, K \geq 0$, then there exists some constant C such that
	$$\Big|\int_{\Omega_{0}}  \pzz \tf(z) h(z) \dd^2z \Big| \leq CK N^{\tau s} \eta_0^{1-s}.$$
	For $1 \leq s \leq 2$, the bound is sharpened to  $CK \log (N) \eta_0^{1-s}$.
\end{lemma}

Thus, in order to study $\phi'(\lambda)$, it is sufficient to estimate $\E \Big[ \ea (\Tr(G(z))-\E \Tr G(z)) \Big]$. The key input is the following cumulant expansion formula.
\begin{lemma}\label{cumulant}(Cumulant expansion formula)
	Let $h$ be a real-valued random variable with finite moments, and $f$ is a complex-valued smooth function on $\R$ with bounded derivatives. Let $c_k$ be the $k$-th cumulant of $h$, given by $c_{k}(h):=(-\ii)^{k} \frac{\dd}{\dd t} \log \E e^{\ii t h} \vert_{t=0}$. Then for any fixed $l \in \N$, we have
	$$\E [h f(h)]=\sum_{k=0}^l \frac{1}{k!} c_{k+1}(h)\E[ f^{(k)}(h) ]+R_{l+1},$$
	where the error term satisfies
	$$|R_{l+1}| \leq C_l \E |h|^{l+2} \sup_{|x| \leq M} |f^{(l+1)}(x)| +C_l \E \Big[ |h|^{l+2} 1_{|h|>M}\Big] \|f^{(l+1)}\|_{\infty},$$
	and $M>0$ is an arbitrary fixed cutoff.
\end{lemma}
For reference, we refer e.g. to Lemma 3.1 in \cite{moment}. We give the  proof the following lemma in Section~5. 

\begin{lemma}\label{lemma4}
	For any $z:=E+\ii \eta \in \Omega_{0} \cap D'$, see (\ref{ddd}), and $\kappa:=\min \{ |E-{L}_-|, |E- L_+|  \}.$ we have
	$$\E [\ea (\Tr G(z)- \E \Tr G(z))] =\frac{\ii \lambda}{\pi } \E [\ea] \int_{\Omega_{0}}\pzzp \tf(z') K(z,z')  \dd^2 z'+\mathcal{E}(z),$$
	where $K$ is given in (\ref{kernel})
	and $\mathcal{E}(z)$ is analytic in $\Omega_0$ and satisfies
	\begin{equation}\label{error_sample}
	\mathcal{E}(z)=O_{\prec} \Big( \frac{1+|\lambda|^4}{\sqrt{\kappa+\eta}} \Big) \Big( \frac{(\kappa+\eta)^{1/4}}{\sqrt{N \eta^3}}+\frac{1}{\sqrt{N \eta^2}}+\frac{1}{\sqrt{N \eta_0 \eta}}+ \frac{1}{N \eta_0 \eta} +\frac{1}{N\eta^2}\Big).
	\end{equation}
\end{lemma}

Admitting Lemma \ref{lemma4} and plugging in (\ref{phi2}),  we have
$$\phi'(\lambda)=- \lambda \E [\ea] V(f)+O_{\prec}\Big( |\lambda| \log N N^{- \tau} \Big)+\tilde{\mathcal{E}},$$
where 
$$V(f)=\frac{1}{\pi^2} \int_{\Omega_{0}}\int_{\Omega_{0}} \pzz \tf(z)  \pzzp  \tf(z') K(z,z') \dd^2 z \dd^2 z', \qquad \tilde{\mathcal{E}}=\frac{\i}{\pi}  \int_{\Omega_0} \pzz \tf(z)  \mathcal{E}(z) \dd^2 z.$$
By the definition of $\kappa_0$ in (\ref{kappa}), $\kappa \geq \kappa_0$. Moreover $|\eta| \geq N^{-\tau} \eta_0$, for $z\in\Omega_0$. Using Lemma \ref{lemma3}, we hence obtain the estimate 
$$\tilde{\mathcal{E}}=O_{\prec}\Big( \frac{(1+|\lambda|^4)N^{3 \tau}}{N \eta_0 \sqrt{\kappa_0+ \eta_0}}\Big)+O_{\prec}\Big( \frac{(1+|\lambda|^4)N^{2 \tau}}{\sqrt{N \eta_0 \sqrt{\kappa_0+\eta_0}}}\Big).$$ 
Assuming $V(f) \prec O(1)$, we can replace $\ea$ by $e(\lambda)$ with error $O_{\prec}( |\lambda| N^{ -\tau})$. Thus we have completed the proof of Proposition \ref{prop}.

\section{Properties of the free convolution}\label{sec:preliminary}

\subsection{Properties of $m_{fc}$ and $\tilde{m}_{fc}$}

In this subsection, we recall some properties of the Stieltjes transforms $m_{fc}$ and $\tilde{m}_{fc}$ of the free convolution measures.
Let $\kappa=\kappa(E)$ be the distance from $E$ to the closest spectral edge, i.e.,
$$\kappa:=\min \Big\{ |E-L_-|, |E-L_+|  \Big\}.$$
Similarly define $\tilde{\kappa}:=\min \Big\{ |E-\tilde{L}_-|, |E-\tilde{L}_+|  \Big\}$. 
Define the spectral domain
$$D:=\Big\{z=E+\i \eta:  |E|<M, 0< \eta \leq 3 \Big\}.$$
\begin{lemma}\label{previous}(Lemma 3.5, Lemma A.1 in \cite{bulk})
	\begin{enumerate}
		\item For all $z \in D$, there exists $C>1$ such that 
		\begin{equation}\label{1}
		C^{-1}\sqrt{\tilde{\kappa} +\eta}\le|\Im \tilde{m}_{fc}(z)| \le C \sqrt{\tilde{\kappa} +\eta},
		\end{equation}
		if $E \in [\tilde{L}_-,\tilde{L}_+]$. If $E \in [\tilde{L}_-,\tilde{L}_+]^c$, then 
		\begin{equation}\label{12}
		C^{-1}\frac{\eta}{\sqrt{\tilde{\kappa}+\eta}}\le|\Im \tilde{m}_{fc}(z)| \le C \frac{\eta}{\sqrt{\tilde{\kappa}+\eta}}.
                \end{equation}
                
		\item (Stability bound) There exists $C>1$, such that 
		\begin{equation}\label{2}
		C^{-1} \leq  |a-z-\tilde{m}_{fc}(z)| \leq C,
		\end{equation}
		uniformly for $z \in D$ and $a \in \mathrm{supp}(\mu_{\alpha})$ .
		\item For all $z \in D$, there exist $k,K>0$ such that
		\begin{equation}\label{3}
		k \sqrt{\tilde{\kappa} +\eta} \leq \Big| 1-\int_{\R} \frac{1}{(x-z-\tilde{m}_{fc}(z))^2} \dd \mu_{\alpha}(x) \Big| \leq K \sqrt{\tilde{\kappa} +\eta}.
		\end{equation}
		
		\item There exist $C>0$ and $c_0>0$ such that for all $z \in D$ satisfying $\tilde{\kappa}+\eta \leq c_0$,
		\begin{equation}\label{4}
		C^{-1} \leq \Big| \int_{\R} \frac{1}{(x-z-\tilde{m}_{fc}(z))^3} \dd \mu_{\alpha}(x) \Big| \leq C;
		\end{equation}
		moreover, there exists $C>1$ such that for all $z \in D$,
		$$\Big| \int_{\R} \frac{1}{(x-z-\tilde{m}_{fc}(z))^3} \dd \mu_{\alpha}(x) \Big| \leq C.$$
	\end{enumerate}
\end{lemma}

The following lemma implies that $m_{fc}$ behaves similarly as $\tilde{m}_{fc}$, for sufficiently large $N$.
\begin{lemma}\label{distance1} (Lemma 3.6 in \cite{bulk})
	Under Assumptions 2.2 and 2.3, for sufficiently large~$N$, statements 1-4 in Lemma \ref{previous} hold true with $\tilde{m}_{fc}$, $\tilde{\kappa}$, $\mu_\alpha$ and $\tilde L_\pm$ replaced by $m_{fc}$, $\kappa$, $\mu_A$ and $L_\pm$ respectively. Moreover, the constants in these inequalities can be chosen uniformly in $N$ for sufficiently large $N$. Furthermore, there exists $c>0$ such that
	\begin{equation}\label{differencem}
	\max_{z \in D}\Big| \tilde{m}_{fc}(z)- m_{fc}(z) \Big| \leq  N^{-\frac{c\alpha_0}{2}},\qquad |\tilde{L}_{\pm}-L_{\pm}|\leq  N^{-c\alpha_0},
	\end{equation}
	for sufficiently large $N$.
\end{lemma}

Recall the function $I(z_1,z_2)$ given in (\ref{iii}) and $I_s(z)$ in (\ref{Is}). By direct computation, one proves the following lemma.
\begin{lemma}\label{Izz}
	For $z_1 \neq z_2$, we have
	\begin{equation}\label{11}
	I(z_1,z_2)=\frac{m_{fc}(z_1)-m_{fc}(z_2)}{z_1+m_{fc}(z_1)-z_2-m_{fc}(z_2)};\qquad I_s(z) = \frac{ m'_{fc}(z)}{1+m'_{fc}(z)}.
	\end{equation}
\end{lemma}

As a result of Lemmas \ref{previous}, \ref{distance1} and \ref{Izz}, we have the following lemma.

\begin{lemma}\label{m}
	There exists $C>1$ such that
	$$ |I(z_1,z_2)| \leq C; \qquad |I_s(z)| \leq 1; \qquad C^{-1} \sqrt{\kappa+\eta} \leq |1-I_s(z)| \leq C \sqrt{\kappa+\eta};$$
	$$ |m_{fc}(z)| \leq C; \qquad |m_{fc}'(z)| \leq \frac{C}{\sqrt{ \kappa+\eta}};\qquad |m_{fc}''(z)| \leq \frac{C}{\sqrt{( \kappa+\eta)^3}},$$
	uniformly for $z, z_1, z_2 \in D$.
\end{lemma}

The proof of the above two lemmas can be found in Appendix B.

\subsection{Properties of the Green function}
As a more general version of the local law in Theorem \ref{locallaw}, we introduce the anisotropic local law. Recall the control parameters $\Psi$ and $\Theta$ from (\ref{control}).
\begin{theorem}\label{isotropic}(Theorem 12.2, 12.4 in \cite{isotropic}; Theorem 2.1, 2.2 in \cite{isotropic2}; Theorem 2.6 in \cite{isotropic3})
	For any deterministic vector $v,w \in \C^N$ and matrix $B \in \C^{N \times N}$, we have
	$$\Big| \langle v, G(z) w \rangle  - \langle v, \widehat{G}(z) w \rangle \Big| \prec \|v\|_2 \|w\|_2 \Psi(z),\qquad \Big| N^{-1} \Tr (B (G(z)-\widehat{G}(z)))  \Big| \prec \|B\|_{\mathrm{op}} \Theta(z),$$
	uniformly in $z \in \Big\{z=E+\ii \eta:  |E| \leq \rho^{-1}, N^{-1+\rho} \leq \eta \leq  \rho^{-1} \Big\}$, where $\rho$ is small so that $\rho^{-1} \geq \|A\|_{\mathrm{op}}$, and $\widehat{G}=\mathrm{Diag} \Big(\frac{1}{a_i-z-m_{fc}(z)} \Big).$
	
\end{theorem}

\section{Proof of Lemma \ref{lemma4}}\label{sec:proof_of_main_lemma}

For the simplicity of the presentation, we consider only the real symmetric case here. The complex case being similar is proved in Appendix A.  For notational simplicity, let 
\begin{equation}\label{gi}
g_i(z):=\frac{1}{a_i-z-m_{fc}(z)}, \qquad z \in \C \setminus \mathrm{supp}(\mu_{fc}).
\end{equation}

Before we proceed the proof of Lemma \ref{lemma4}, we state a useful lemma. 

\begin{lemma}\label{lemma2}

For any $i,j$, we have	
\begin{align}
	\frac{\partial \ea}{\partial H_{ij}}&=-\frac{\ii (2-\delta_{ij})\lambda}{\pi}  \ea  \int_{\Omega_{0}}\pzz \tf(z) \frac{\dd}{\dd z}G_{ji} \dd^2 z;\label{lemma21}\\
	\frac{\partial^2 \ea}{\partial^2 H_{ij}}&=\frac{\ii (2-\delta_{ij}) \lambda}{\pi} \ea  \int_{\Omega_{0}}\pzz \tf(z) \frac{\dd}{\dd z} (g_i(z) g_j(z)) \dd^2 z+O_{\prec} \Big( \frac{(1+|\lambda|)^2}{\sqrt{N \eta_0}}\Big).\label{lemma22}
	\end{align}
	In general, for any integer $k \in \N$, we have
	\begin{equation}\label{dke}
	\Big| \frac{\partial^k G_{ij}}{\partial H^k_{ij}} \Big| \prec O(1); \quad \Big| \frac{\partial^k  \ea}{\partial^k H_{ij}} \Big| \prec O((1+|\lambda|)^k).
	\end{equation}
\end{lemma}
\noindent The above lemma follows from the relation
	\begin{equation}\label{dH}
	\frac{\partial G_{ij}}{\partial H_{ab}}=-\frac{G_{ia} G_{bj}+G_{ib}G_{aj}}{1+\delta_{ab}}.
	\end{equation}
The details are provided in Appendix~\ref{appendix B}. Now we are ready to prove Lemma \ref{lemma4}.

\begin{proof}[Proof of Lemma \ref{lemma4}]
	By the definition of the resolvent function, we have
	$$(z-a_i) G_{ii}=(HG)_{ii}-1.$$
	Thus we obtain that
	\begin{align*}
	(z-a_i) \E [\ea(G_{ii}-\E G_{ii})]=\sum_{j=1}^N \Big( \E [ H_{ij}G_{ji} \ea ]- \E [H_{ij}G_{ji}] \E [\ea] \Big).
	\end{align*}
\noindent Using the cumulant expansion Theorem \ref{cumulant}, we obtain
\begin{align}\label{sum}
(z-a_i) \E [\ea(G_{ii}-\E G_{ii})]=I_1+I_2+I_3+O_{\prec}(N^{-\frac{3}{2}}(1+|\lambda|^{4})),
\end{align}	
where
\begin{align*}
I_1:&=\frac{1}{N}\sum_{j=1}^N c^{(2)}_{ij} \bigg( \E \Big[\frac{\partial  \ea}{\partial H_{ij}} G_{ji} \Big]+\E \Big[ \Big( \frac{\partial G_{ji} }{\partial H_{ij}} - \E \Big[\frac{\partial G_{ji} }{\partial H_{ij}}  \Big] \Big) \ea \Big] \bigg);\\
I_2:&=\frac{1}{2! N^{\frac{3}{2}}}\sum_{j=1}^N c^{(3)}_{ij} \bigg( \E \Big[\frac{\partial^2 \ea}{\partial^2 H_{ij}} G_{ji}  \Big]+2\E \Big[\frac{\partial \ea}{\partial H_{ij}} \frac{\partial G_{ji}}{\partial H_{ij}} \Big]+\E \Big[(1-\E)\Big(\frac{\partial^2 G_{ji} }{\partial^2 H_{ij}} \Big) \ea \Big] \bigg);\\
I_3:&=\frac{1}{3! N^{2}}\sum_{j=1}^N c^{(4)}_{ij} \bigg( \E \Big[\frac{\partial^3  \ea}{\partial^3 H_{ij}} G_{ji} \Big]+3 \E \Big[\frac{\partial^2  \ea}{\partial^2 H_{ij}} \frac{\partial G_{ji}}{\partial H_{ij}} \Big]+3\E \Big[\frac{\partial  \ea}{\partial H_{ij}} \frac{\partial^2 G_{ji}}{\partial^2 H_{ij}} \Big]\\
&\qquad+\E \Big[ \Big( \frac{\partial^3 G_{ji} }{\partial^3 H_{ij}}-\E \Big[\frac{\partial^3 G_{ji} }{\partial^3 H_{ij}}  \Big]\Big) \ea \Big]  \bigg).
\end{align*}

\noindent Here $c^{(k)}_{ij}$ denotes the $k$-th cumulant of $\sqrt{N} H_{ij}$. In particular, 
	$$c^{(1)}_{ij}=0; \qquad c^{(2)}_{ij}=1+(m_2-1)\delta_{ij}; \qquad c^{(4)}_{ij}=W_4-3\quad (i \neq j).$$
The last term on the right side of (\ref{sum}) is estimated by (\ref{dke}),~\eqref{moment condition finite} and Lemma \ref{dominant}. Note that for $z \in \Omega_0 \cap D'$, we have the deterministic bound $|G_{ij}| \leq \|G\|_{\mathrm{op}}\leq ({\Im}z)^{-1}=O(N^c)$. Combining with $|\ea|=1$, we can use the last statement of Lemma \ref{dominant}. We will use this argument throughout the proof. The error terms in this section are all uniform in $z \in \Omega_{0} \cap D'$. In the following, we estimate $I_1$, $I_2$ and $I_3$ respectively.

	\subsection{Estimate on $I_1$}
	Using (\ref{dH}), we have for each $i$, 
	\begin{align*}
	I_1=&-\frac{1}{N} \E [\ea ((G^2)_{ii} -\E (G^2)_{ii})] -\frac{1}{N} \E [\ea (\Tr G G_{ii} -\E \Tr G G_{ii}) ]\\
	 &-\frac{m_2-2}{N} \E [\ea (G_{ii} G_{ii}-\E G_{ii} G_{ii} )]
	+\frac{1}{N} \sum_{j=1}^N (1+(m_2-1) \delta_{ij}) \E\Big[ \frac{\partial \ea}{ \partial H_{ij}} G_{ji}\Big]\\=:&A_1(i)+A_2(i)+A_3(i)+A_4(i).
	\end{align*}
	The first term can be written as
	$$A_1(i)=-\frac{1}{N} \E \Big[ \ea  (1-\E) \frac{\dd}{\dd z}   (G(z))_{ii}\Big] \prec \frac{\Psi(z)}{{N\Im}z},$$
	with $\Psi(z)$ as in (\ref{control}). The last step follows from the local law and the Cauchy integral formula.
	Similarly using the local law, the second term $A_2(i)$ can be written as
	\begin{align*}
	A_2(i)&=-\frac{1}{N} \E [\ea (\Tr G (G_{ii} - \E G_{ii}) + \E G_{ii}(\Tr G-\E \Tr G) +\E \Tr G \E G_{ii}-\E \Tr G G_{ii})  ] \\
	&=- m_{fc}(z) \E [\ea (G_{ii}-\E G_{ii})]-\frac{1}{N}g_i(z) \E [\ea (\Tr G- \E \Tr G)]+O_{\prec}( \Theta(z) \Psi(z)),
	\end{align*}
	with $\Theta(z)$ as in (\ref{control}) and $g_i(z)$ in (\ref{gi}). Here the first term of $A_2$ will be moved to the left side of the equation (\ref{sum}).
	In addition, the local law also implies that $A_3(i)=O_{\prec}\Big(\frac{\Psi(z)}{N}\Big).$
	
	Note that $A_4$ is a leading term of $I_1$. Using the local law, (\ref{lemma21}) and Lemma \ref{dke}, we write
	\begin{align*}
	A_4(i)=&A_{41}(i)+A_{42}(i)+O_{\prec} \Big((1+|\lambda|)N^{-1} \Psi(z) \Big),
	\end{align*}
	where
	\begin{align}\label{a4}
A_{41}(i)=\frac{1}{N} \sum_{j=1}^N  \E\Big[ \frac{\partial \ea}{ \partial H_{ij}} (1+ \delta_{ij})G_{ji}\Big],\quad \mbox{and}\quad A_{42}(i)=\frac{m_2-2}{N} \E\Big[ \frac{\partial \ea}{ \partial H_{ii}} g_i(z)\Big].
	\end{align}
	
	We compute these two terms below in the Section 5.4.

	\subsection{Estimate on $I_2$}
	In this subsection, we will show that $I_2$ is negligible, which can be written as
	\begin{align*}
	&\quad \frac{1}{2 N^{\frac{3}{2}}} \sum_{j=1}^N c^{(3)}_{ij} \Big(  \E \Big[\frac{\partial^2  \ea}{\partial^2 H_{ij}} G_{ji} \Big] +2\E \Big[\frac{\partial  \ea}{\partial H_{ij}} \frac{\partial  G_{ji}}{\partial H_{ij}}  \Big]  +\E\Big[ \ea \Big( \frac{\partial^2  G_{ji}}{\partial^2 H_{ij}} -\E \frac{\partial^2  G_{ji}}{\partial^2 H_{ij}} \Big) \Big] \Big)\\
	&=:B_1(i)+B_2(i)+B_3(i).
	\end{align*}

	First, we study the last term $B_3(i)$. Using (\ref{dH}) and the local law, we have
	\begin{align*}
	B_3(i)&=-\frac{1}{2 N^{\frac{3}{2}}} \sum_{j=1}^N c_{ij}^{(3)} \E \Big[ \ea \Big(  6 G_{ii}G_{jj} G_{ij} +2 (G_{ij})^{3}-6 \E [G_{ii}G_{jj} G_{ij}] -2 \E(G_{ij})^{3}   \Big) \Big]\\
	&=-\frac{3}{N^{\frac{3}{2}}} \sum_{j=1}^N \E \Big[ \ea c_{ij}^{(3)} g_i(z) g_j(z) ( G_{ij} -\E G_{ij}) \Big] +O_{\prec} \Big( N^{-\frac{1}{2}} \Psi^2(z) \Big).
	\end{align*}
	Next, we estimate $\frac{1}{\sqrt{N}} \sum_{j=1}^N c^{(3)}_{ij} g_j(z) G_{ij}$, using the anisotropic local law Theorem \ref{isotropic}. Let $v_j=\delta_{ij}$ and $w_j=\frac{1}{\sqrt{N}}c^{(3)}_{ij} g_j(z).$ And $\|w\|_2$ is bounded because of the stability bound (\ref{2}) and the moment condition~\eqref{moment condition finite}. 	
	
	Note that Theorem \ref{isotropic} holds for vector entries $w_j$ and $v_j$ that are deterministic constants. As in our setting $w_j$ depend on $z$, we use a continuity argument to show that 
		\begin{equation}\label{estimate}
		\Big| \frac{1}{\sqrt{N}} \sum_{j=1}^N c^{(3)}_{ij} g_j(z) G_{ij}(z) -\frac{1}{\sqrt{N}} c^{(3)}_{ii} (g_i(z))^2 \Big| \prec \Psi(z)\,,
		\end{equation}
uniformly in $z\in D'$. Indeed, choose a lattice $\Delta$ of the domain $D'$ in (\ref{ddd}), with $|\Delta|=N^{100}$. Then for any $z \in D'$, there exists some point $p \in \Delta$, such that $|z-p| \leq N^{-10}$. The anisotropic local law (\ref{isotropic}) combined with a union bound implies 
\begin{equation}\label{cont_argument}
\mathbb{P}\Big( \exists p \in \Delta: \Big| \frac{1}{\sqrt{N}} \sum_{j=1}^N c^{(3)}_{ij} g_j(p) G_{ij}(p) -\frac{1}{\sqrt{N}} c^{(3)}_{ii} (g_i(p))^2 \Big| \geq N^{\epsilon} \Psi(p)\Big) \leq N^{-D+100}.
\end{equation}
Recall that $g_j(z)=\frac{1}{a_j-z-m_{fc}(z)}$. Using (\ref{2}), Lemma \ref{m} and the fact that $|G_{ij}(z)| \leq \frac{1}{\eta}$, the function $\frac{1}{\sqrt{N}} \sum_{j=1}^N c^{(3)}_{ij} g_j(z) G_{ij}(z)-\frac{1}{\sqrt{N}} c^{(3)}_{ii} (g_i(z))^2$ as well as $\Psi(z)$ are Lipschitz continuous on $D'$ with Lipschitz constant at most $N^3$. Thus we obtain from (\ref{cont_argument}) that
\begin{equation}\label{cont_argument_2}
\mathbb{P}\Big( \exists z\in D': \Big| \frac{1}{\sqrt{N}} \sum_{j=1}^N c^{(3)}_{ij} g_j(z) G_{ij}(z) -\frac{1}{\sqrt{N}} c^{(3)}_{ii} (g_i(z))^2 \Big| \geq 2N^{\epsilon} \Psi(z)\Big) \leq N^{-D+100}\,,
\end{equation}		
		which implies~\eqref{estimate}.

	Next, using (\ref{2}) and $\Psi(z) \geq C N^{-\frac{1}{2}}$, we have $\Big| \frac{1}{\sqrt{N}} \sum_{j=1}^N c^{(3)}_{ij} g_j(z) G_{ij}(z) \Big| \prec \Psi(z).$ Therefore, we obtain the upper bound
	$$B_3(i)=O_{\prec}\Big( N^{-1} \Psi(z) \Big)+O_{\prec} \Big(N^{-\frac{1}{2}} \Psi^2(z) \Big)=O_{\prec} \Big(N^{-\frac{1}{2}} \Psi^2(z) \Big).$$
	
	For the second term, by (\ref{dH}), (\ref{lemma21}), and the local law we have
	\begin{align*}
	B_2(i)=&-\frac{1}{N^{\frac{3}{2}}} \sum_{j=1}^N c^{(3)}_{ij}  \E \Big[\frac{\partial  \ea}{\partial H_{ij}} (G_{ji}G_{ji} +G_{ii}G_{jj}) \Big]\\
	=&-\frac{1}{N^{\frac{3}{2}}} \sum_{j=1}^N c^{(3)}_{ij} \E \Big[\frac{\partial  \ea}{\partial H_{ij}} g_i(z) g_j(z)\Big]   +O_{\prec}\Big( \frac{(1+|\lambda|) \Psi(z)}{N \sqrt{\eta_0}} \Big)\\
	=&\frac{2 \ii \lambda}{ \pi N^{\frac{3}{2}}}   \E \Big[ \ea  \sum_{j=1}^N c^{(3)}_{ij} \int_{\Omega_{0}}\pzzp \tf(z') \frac{\dd}{\dd z'}(G(z'))_{ji} \dd^2 z'  g_i(z) g_j(z)  \Big]+O_{\prec}\Big( \frac{(1+|\lambda|) \Psi(z)}{N \sqrt{\eta_0}} \Big).
	\end{align*}
	By the same argument as in (\ref{estimate}) and the Cauchy integral formula, we have
	$$\Big| \frac{\dd}{\dd z'} \frac{1}{\sqrt{N}} \Big( \sum_{j=1}^N c^{(3)}_{ij} g_j(z) (G(z'))_{ji} \Big) \Big|=O_{\prec}\Big( \frac{\Psi(z')}{|{\Im}z'|}\Big). $$
	Using the stability bound (\ref{2}) and Lemma \ref{lemma3}, we have
	$$|B_2(i)| \prec  \frac{1+|\lambda|}{N \sqrt{N \eta_0}}+\frac{(1+|\lambda|) \Psi(z)}{N \sqrt{\eta_0}}=O_{\prec}\Big(\frac{(1+|\lambda|) \Psi(z)}{N \sqrt{\eta_0}} \Big).$$
	
	Similarly, by plugging (\ref{lemma22}) in the expression of $B_1$, we have
	\begin{align*}
	B_1(i)=&\frac{ \ii \lambda}{\pi N^{\frac{3}{2}}} \sum_{j=1}^N c^{(3)}_{ij}  \E \Big[  \ea \Big( \int_{\Omega_{0}}\pzzp \tf(z') \frac{\dd}{\dd z'}  ( g_i(z') g_j(z'))  \dd^2 z'\Big) G_{ji} \Big]
	+O_{\prec} \Big(  \frac{(1+|\lambda|^2) \Psi(z)}{N \sqrt{ \eta_0}}\Big).
	\end{align*}
	Using the anisotropic local law, we have	
	$$B_1(i) =O_{\prec} \Big(  (1+|\lambda|^2)N^{-1} \Psi(z)\Big)+O_{\prec} \Big(  \frac{(1+|\lambda|^2) \Psi(z)}{N \sqrt{\eta_0}}\Big)=O_{\prec} \Big(  \frac{(1+|\lambda|^2) \Psi(z)}{N \sqrt{\eta_0}}\Big).$$

	\subsection{Estimate on $I_3$}
	It is not hard to show that the diagonal terms for $i=j$ are negligible. Thus we can replace the fourth cumulants by $W_4-3$.  There are four terms in $I_3$ and we denote them as $D_1(i)$, $D_2(i)$, $D_3(i)$ and $D_4(i)$ respectively.
	
	First, we look at $D_1$. By the local law and (\ref{dke}), we have $|D_1(i)| \prec (1+|\lambda|^3) N^{-1} \Psi(z).$ Similarly, using~(\ref{dH}), (\ref{lemma21}) and the local law, we have $|D_3(i)| \prec \frac{(1+|\lambda|) \Psi(z) }{ N \sqrt{N \eta_0}}.$
	For the last term $D_4$, using (\ref{dH}) and the local law, we obtain that
	\begin{align*}
	D_4(i)=\frac{1}{6 N^2} \sum_{j=1}^{N} \E \Big[ \ea (1-\E) \Big( 36 G_{ii}G_{jj}(G_{ij})^2+6 (G_{ii})^2 (G_{jj})^2+6 (G_{ij})^4 \Big) \Big] \prec N^{-1} \Psi(z).
	\end{align*}
	Finally, we look at the leading term $D_2(i)$. Using the local law and (\ref{dke}), we have
	\begin{align}\label{d2}
	D_2(i)=&-\frac{W_4-3}{2 N^{2}}\sum_{j=1}^N \E \Big[\frac{\partial^2  \ea}{\partial^2 H_{ij}} \Big( (G_{ji})^2+G_{ii} G_{jj} \Big) \Big]\nonumber\\
	=&-\frac{W_4-3}{2 N^{2}}\sum_{j=1}^N \E \Big[\frac{\partial^2  \ea}{\partial^2 H_{ij}} g_i(z) g_j(z)  \Big]+O_{\prec} \Big( (1+|\lambda|^2)N^{-1} \Psi(z) \Big).
	\end{align}

	\subsection{Adding up the contributions to (\ref{sum})}\label{subsection:sum}
	Summing up the contributions from the previous subsections, we write (\ref{sum}) as
	\begin{multline*}
	(z-a_i+m_{fc}) \E [\ea(G_{ii}-\E G_{ii})]=-\frac{1}{N} g_i(z)\E [\ea (\Tr G- \E \Tr G)]+A_{41}(i)+A_{42}(i)+D_2(i)+\epsilon(i),
	\end{multline*}
	where $D_2$ is given in (\ref{d2}) and $A_{41}$, $A_{42}$ in (\ref{a4}), and $\epsilon(i)$ is the error term obtained in the previous subsections. Thanks to the stability bound (\ref{2}), we can divide  both sides by $z-a_i+m_{fc}$ to get
	\begin{align*}
	\E [\ea(G_{ii}-\E G_{ii})]=&\frac{1}{N} (g_i(z))^2 \E [\ea (\Tr G- \E \Tr G)]-g_i(z)\Big( A_{41}(i)+A_{42}(i)+D_2(i) +\epsilon(i) \Big).
	\end{align*}
	Summing over $i$ and rearranging, we find 
	\begin{equation}\label{newsum}
	(1-I_s(z)) \E [\ea (\Tr G- \E \Tr G)] =-\sum_{i=1}^N g_i(z) \Big( A_{41}(i)+A_{42}(i)+D_2(i) \Big)+\mathcal{E}_1,
	\end{equation}
	where $\mathcal{E}_1$ is the linear statistics of $\epsilon(i)$. By the argument in Section 5.1-5.3, we get 
	$$\mathcal{E}_1=O_{\prec}\Big( (1+|\lambda|^4) N \Psi(z) \Theta(z) \Big)+O_{\prec}\Big( (1+|\lambda|^4) \sqrt{N} \Psi^2(z) \Big)+O_{\prec}\Big( \frac{(1+|\lambda|^4) \Psi(z)}{\sqrt{\eta_0}} \Big).$$ 
	
	Next, we study the leading terms of the right side of (\ref{newsum}). Plugging (\ref{lemma21}) in (\ref{a4}), we have
	$$\sum_{i=1}^N \frac{A_{41}(i)}{z-a_i+m_{fc}}=\frac{ 2 \ii \lambda}{\pi N} \sum_{i=1}^N g_i(z) \E \Big[ \ea  \int_{\Omega_{0}}\pzzp \tf(z') \pzp (G(z')G(z))_{ii}  \dd^2 z' \Big].$$
	By the resolvent identity,
	\begin{equation}\label{resolvent_identity}
	G(z) G(z')=\frac{G(z)-G(z')}{z-z'}, \quad z \neq z',
	\end{equation}
	 we can write
        $$F(z,z'):=\frac{1}{N}\sum_{i=1}^N g_i(z) (G(z')G(z))_{ii}=\frac{1}{N}\sum_{i=1}^N \frac{1}{z' -{z}}g_i(z) (G_{ii}(z')-G_{ii}(z)).$$
 We separate into two cases:\\
	\textbf{Case 1:} If $z$ and $z'$ belong to different half-planes, then we have $\frac{1}{|z-z'|} \leq \frac{1}{| {\Im}z |}.$
	Thus by the anisotropic local law, we have
	\begin{align*}
	\Big| F(z,z')- \frac{1}{z' -{z}} \frac{1}{N}\sum_{i=1}^N g_i(z) (g_i(z')-g_i(z)) \Big|& =\frac{1}{|z'-z|} \Big|\frac{1}{N}\sum_{i=1}^N g_i(z) (G_{ii}(z')-g_i(z')) \Big|  \\
	+\frac{1}{|z'-z|} & \Big|\frac{1}{N}\sum_{i=1}^N g_i(z) (G_{ii}(z)-g_i(z)) \Big| =O_{\prec}\Big( \frac{\Theta(z)+\Theta(z')}{|{{\Im}z}|} \Big).
	\end{align*}
	
	\noindent \textbf{Case 2:} If $z$ and $z'$ are in the same half-plane, without loss of generality, we can assume they both belong to the upper half plane. If $|{\Im}z-{\Im}z'| \geq \frac{1}{2} {\Im}z$, then we can use the same argument as in Case~1. Thus it is sufficient to study when $|{\Im}z-{\Im}z'| \leq \frac{1}{2}{\Im}z$, which means $ \frac{2}{3} {\Im}z' \leq {\Im}z \leq 2 {\Im}z'$. Note that
	\begin{align*}
	\Big| F(z,z')-\frac{1}{z' -{z}} \frac{1}{N}\sum_{i=1}^N  g_i(z) (g_i(z')-g_i(z)) \Big|  \leq \Big|\frac{\frac{1}{N}\sum_{i=1}^N (g_i(z)-g_i(z')) (G_{ii}(z')-g_i(z'))}{z-z'}  \Big| \\
	+\Big| \frac{\frac{1}{N}\sum_{i=1}^N g_i(z') (G_{ii}(z')-g_i(z'))-\frac{1}{N}\sum_{i=1}^N g_i(z) (G_{ii}(z)-g_i(z)) }{z-z'} \Big|.
	\end{align*}
	Next, we look at the first term on the right side. By direct computation, we get
	$$\Big| \frac{g_i(z)-g_i(z')}{z-z'} \Big| \leq |g_i(z) || g_{i}(z') | \Big( 1+\Big| \frac{m_{fc}(z)-m_{fc}(z')}{z-z'} \Big| \Big).$$
	When $z,z'$ are in the same half plane, $m_{fc}$ is analytic in the neighborhood of the segment connecting $z$ and $z'$, denoted as $L(z,z')$. Thus
	$$\Big| \frac{m_{fc}(z)- m_{fc}(z')}{z-z'} \Big| \leq \sup_{\omega \in L(z,z')} \Big|m'_{fc}(\omega)\Big| \leq \frac{C}{{\Im}z}.$$
	Combining with (\ref{2}), we have
	$$\Big| \frac{g_i(z)-g_i(z')}{z-z'} \Big| \leq \frac{C'}{{\Im}z}.$$
	Using the second statement of the anisotropic local law by letting $B=\mbox{Diag}\Big(\frac{g_i(z)-g_i(z')}{z-z'}\Big)$ and the continuity argument as in (\ref{estimate}), we obtain that the first term is bounded as $O_{\prec}\Big( \frac{\Theta(z')}{{{\Im}z}} \Big)$. 
	
	For the second term, we write it as $\frac{h(z)-h(z')}{z-z'}$, where $$h(z):=\frac{1}{N}\sum_{i=1}^N g_i(z)(G_{ii}(z)-g_i(z)).$$ Since $h$ is analytic in the neighborhood of $L(z,z')$, we have
	$$\Big| \frac{h(z)- h(z')}{z-z'} \Big| \leq \sup_{\omega \in L(z,z')} \Big|\frac{\dd}{\dd \omega} h(\omega)\Big|.$$
	The anisotropic local law implies that $\sup_{w \in L(z,z')}|h(w)| \prec \Theta(z)$. Using the Cauchy integral formula, the second term is $O_{\prec}\Big( \frac{\Theta(z)}{{{\Im}z}} \Big)$. Then we obtain the same upper bound as in Case~1.
	
	Therefore, in both cases, we have
	$$F(z,z')=\frac{1}{z'-z} \Big(\frac{1}{N}\sum_{i=1}^N g_i(z) g_i(z')-\frac{1}{N}\sum_{i=1}^N g_i^2(z) \Big) +O_{\prec}\Big( \frac{\Theta(z)}{{{\Im}z}} \Big)+O_{\prec}\Big( \frac{\Theta(z')}{{{\Im}z}} \Big).$$
	
	\noindent Taking the derivative and using the Cauchy integral formula, we have
	\begin{align*}
	\pzp F(z,z')
	=\pzp\Big( \frac{1}{1-I(z,z')} \frac{\partial I(z,z')}{\partial z}\Big) (1-I_s(z))+O_{\prec}\Big( \frac{\Theta(z)}{{{\Im}z|{\Im}z'}|} \Big)+O_{\prec}\Big( \frac{\Theta(z')}{{{\Im}z|{\Im}z'}|} \Big).
	\end{align*}
	Then by using Lemma \ref{lemma3}, we have
	\begin{align*}
	\sum_{i=1}^N \frac{A_{41}(i)}{z-a_i+m_{fc}}=&\frac{ 2 \ii  \lambda}{\pi } \E [\ea] \int_{\Omega_{0}}\pzzp \tf(z') \pzp \Big( \frac{1}{1-I(z,z')} \frac{\partial I(z,z')}{\partial z} (1-I_s(z)) \Big) \dd^2 z' \\
	&+O_{\prec}\Big( \frac{\Theta(z)}{{\Im}z} \Big)+O_{\prec}\Big( \frac{1}{{N \eta_0} {\Im}z} \Big).
	\end{align*}
	
	Similarly, plugging (\ref{lemma21}) in (\ref{a4}), we have
	\begin{align*}
	\sum_{i=1}^N \frac{A_{42}(i)}{z-a_i+m_{fc}}=\frac{ (m_2-2) \ii \lambda }{\pi } \E [\ea] \int_{\Omega_{0}}\pzzp \tf(z') \pzp \Big( \frac{\partial I(z,z')}{\partial z} (1-I_s(z)) \Big) \dd^2 z' +O_{\prec}\Big( \frac{1}{\sqrt{N \eta_0}}\Big).
	\end{align*}
	
	Finally, plugging (\ref{lemma22}) in the leading term of (\ref{d2}) we have
	\begin{align*}
       \sum_{i=1}^N \frac{D_{2}(i)}{z-a_i+m_{fc}}\frac{(W_4-3)  \ii \lambda}{ \pi} \E[\ea] \int_{\Omega_{0}} \pzzp  \tf(z') \pzp \Big(  \frac{\partial I(z,z')}{\partial z} (1-I_s(z))  I(z,z') \Big)\dd^2 z'+O_{\prec}\Big( \frac{1+|\lambda|^2}{\sqrt{N \eta_0}}\Big).
	\end{align*}

	\noindent Therefore, we have
	$$(1-I_s(z)) \E [\ea (\Tr G- \E \Tr G)] =(1-I_s(z)) \frac{\i  \lambda}{\pi } \E [\ea] \cdot  \int_{\Omega_{0}}\pzzp \tf(z')  K(z,z') \dd^2 z' +\mathcal{E}_2,$$
	where $K(z,z')$ is given by (\ref{kernel}) and 
	\begin{equation}\label{erroruselater}
	\mathcal{E}_2=(1+|\lambda|^4) \Big[ O_{\prec}\Big(  N \Psi(z)\Theta(z) \Big)+O_{\prec}\Big( ( \sqrt{N} \Psi^2(z) \Big)+O_{\prec}\Big( \frac{\Psi(z)}{\sqrt{\eta_0}} \Big)+O_{\prec}\Big( \frac{ \Theta(z)}{\eta_0} \Big) +O_{\prec}\Big( \frac{ \Theta(z)}{\eta} \Big)\Big]. 
	\end{equation} 
Dividing both sides by $1-I_s(z)$, recalling from Lemma \ref{m} that $\Big|\frac{1}{1-I_s(z)} \Big| \sim \frac{1}{\sqrt{\kappa+\eta}}$, and using (\ref{1}), (\ref{12}), we have completed the proof of Lemma \ref{lemma4}.	\end{proof}

\section{Proof of Theorem \ref{meso} and Theorem \ref{mesoedge}}\label{sec:proof_of_main_thm}

In this section, we compute the variances of the mesoscopic CLT in the bulk and at the edges.

\subsection{In the bulk}
We compute the variance $V(f)$ defined in (\ref{vf}) with $f$ given in (\ref{fn}). 
\begin{lemma}\label{vfbulk}
Under the assumptions and notations of Theorem \ref{thm:weak_convergence},  we have
	$$\lim_{N \rightarrow \infty} V(f)=\frac{1}{2 \beta \pi^2} \int_{\R} \int_{\R} \frac{(g(x_1)-g(x_2))^2}{(x_1-x_2)^2} \dd x_1 \dd x_2.$$
\end{lemma}

Assuming that we have proved the above lemma, $V(f)$ converges to some positive constant since $g \in C^2_c(\R)$. Theorem \ref{meso} is a direct result of Proposition \ref{prop} after integrating $\phi'(\lambda)$ and using the Arzel\'a-Ascoli theorem and the L\'evy continuity theorem.

\begin{proof}[Proof of Lemma \ref{vfbulk}]
	Recall that
	$$V(f)=\frac{1}{\pi^2} \int_{\Omega_{0}} \int_{\Omega_{0}} \pzaa \tf(z_1) \pzbb \tf(z_2) (K_1+K_2+K_3)\dd^2z_1 \dd^2z_2,$$
	where 
	\begin{align}
	K_1&=\Big(m_2-\frac{2}{\beta}\Big) \pzab {I};\quad K_2=\Big(W_4-1-\frac{2}{\beta}\Big)  \Big( {I} \pzab {I} +\pza {I} \pzb {I} \Big);\label{k1}\\
	K_3&=\frac{2}{\beta}\frac{\partial}{\partial z_1}\Big(\frac{1}{1- I}\frac{\partial  I}{\partial z_2}\Big)=\frac{2}{\beta} \Big( \frac{1}{1-{I}(z_1,z_2)}\pzab {I}+ \frac{1}{(1-{I}(z_1,z_2))^2} \pza {I} \pzb {I} \Big).\label{k3}
	\end{align}
	Using Lemma \ref{m} and the stability bound (\ref{2}), we have
	\begin{align}
	\pza I(z_1,z_2)&=\frac{1}{N} \sum_{i=1}^N \frac{1+m'_{fc}(z_1)}{(a_i-z_1-m_{fc}(z_1))^2(a_i-z_2-m_{fc}(z_2)) }=O\Big( \frac{1}{\sqrt{\kappa_1+\eta_1}}\Big);\label{derivative1}\\
	\pzb I(z_1,z_2)&=\frac{1}{N} \sum_{i=1}^N \frac{1+m'_{fc}(z_2)}{(a_i-z_1-m_{fc}(z_1))(a_i-z_2-m_{fc}(z_2))^2 }=O \Big( \frac{1}{\sqrt{\kappa_2+\eta_2}}\Big);\label{derivative2}\\
	\pzab I(z_1,z_2)&=\frac{1}{N} \sum_{i=1}^N \frac{(1+m'_{fc}(z_1))(1+m'_{fc}(z_2))}{(a_i-z_1-m_{fc}(z_1))^2(a_i-z_2-m_{fc}(z_2))^2 }=O \Big( \frac{1}{\sqrt{(\kappa_1+\eta_1)(\kappa_2+\eta_2)}} \Big).\label{derivative3}
	\end{align}
	In addition, recalling (\ref{11}), for $z_1 \neq z_2$, we have
	\begin{equation}\label{derivative4}
	\frac{1}{1-I(z_1,z_2)}  = 1+\frac{m_{fc}(z_1)-m_{fc}(z_2)}{z_1-z_2}.
	\end{equation}
	If $z_1$ and $z_2$ are in the same half plane, $m_{fc}$ is analytic in a neighborhood of the segment connecting $z_1$ and $z_2$, denoted as $L(z_1,z_2)$.  By Lemma \ref{m}, then we have
	\begin{equation}\label{same}
	\Big| \frac{1}{1-I(z_1,z_2)} \Big| \leq 1+\Big| \frac{m_{fc}(z_1)-m_{fc}(z_2)}{z_1-z_2} \Big| \leq \sup_{z \in L(z_1,z_2)} |m'_{fc}(z)| \leq C \sup_{z \in L(z_1,z_2)}\Big(\frac{1}{\sqrt{\kappa+\eta}}\Big).
	\end{equation}
	If $z_1$, $z_2$ belong to different half planes, using Lemma \ref{m}, then we have
	\begin{equation}\label{notsame}
	\Big| \frac{1}{1-I(z_1,z_2)} \Big| \leq 1+\Big| \frac{m_{fc}(z_1)-m_{fc}(z_2)}{z_1-z_2} \Big|  \leq  \frac{C}{|z_1-z_2|} \leq \frac{C}{|\eta_1|+|\eta_2|}.
	\end{equation}
	Now, we are ready to compute $V(f)$. Since $\pzz K_i(z,z')= \pzzp K_i(z,z')=0$, $(i=1,2,3)$, and by Stokes' formula, we have
	\begin{align*}
	V(f)=-\frac{1}{4 \pi^2} \int_{\Gamma_1} \int_{\Gamma_2}  \tilde{f}(z_1) \tilde{f}(z_2) (K_1+K_2+K_3) \dd z_1 \dd z_2:=V_1+V_2+V_3,
	\end{align*}
	where $\Gamma_1=\{x_1+\i y_1: |y_1|= N^{-\tau} \eta_0 \}$ and $\Gamma_2=\{x_2+\i y_2: |y_2|= \frac{1}{2} N^{-\tau} \eta_0 \}$. We choose the orientation of both contours to be counterclockwise. The parts on the upper half plane are denoted as $\Gamma_1^+, \Gamma_2^+$, while the parts on the lower half plane are $\Gamma_1^-, \Gamma_2^-$.
	
	Using (\ref{derivative1})-(\ref{derivative3}), since $\kappa \geq \kappa_0 \geq c_0$ for some positive constant $c_0>0$, we have $|K_1+K_2| = O(1).$
	Combining with (\ref{assumpf}), by direct computation, we have $|V_1+V_2| =O(\eta_0^2).$ It then suffices to estimate $V_3$. We consider two cases.
	
	\textbf{Case 1}: If $z_1, z_2$ are in the same half plane, by (\ref{same}) and (\ref{derivative1})-(\ref{derivative3}), we have  $|K_3| = O(1).$ Therefore, 
	$$\Big( \int_{\Gamma_1^+} \int_{\Gamma_2^+} +\int_{\Gamma_1^-} \int_{\Gamma_2^-} \Big) \tf(z_1) \tf(z_2) K_3(z_1,z_2) \dd z_1 \dd z_2=O(\eta^2_0).$$
	
	\textbf{Case 2}: Consider $z_1$, $z_2$ are in different half planes. For notational simplicity, we define $m_1={m}_{fc}(z_1)$ and $m_2={m}_{fc}(z_2)$.
	Differentiating $I$ given in (\ref{11}), we have 
	\begin{equation}\label{temp}
	\pza {I}=  \frac{(z_1-z_2) m'_1-m_1+m_2}{(z_1+m_1-z_2-m_2)^2}; \qquad \pzb {I}=\frac{(z_2-z_1) m'_2+m_1-m_2}{(z_1+m_1-z_2-m_2)^2}.
	\end{equation}
	Using (\ref{notsame}) (\ref{derivative4}), (\ref{derivative1})-(\ref{derivative3}) and Lemma \ref{m}, we have
	$$K_3 =\frac{2}{\beta}  \frac{1}{(z_1-z_2)^2}  \frac{((z_1-z_2) m'_1-m_1+m_2)((z_2-z_1) m'_2+m_1-m_2)}{(z_1+m_1-z_2-m_2)^2}+O( \eta^{-1}_0N^\tau).$$
	Note that if $z \in \C^+$ and in the bulk,  then there exists $k, K >0$ such that $k \leq {\Im} m_{fc}(z) \leq K$. 
	If $z_1, z_2$ are in different half planes, there exists some constant $c>0$ such that $|z_1+m_1-z_2-m_2|>c$. Combining with Lemma \ref{m}, we have
	$$K_3=-\frac{2}{\beta} \frac{(m_1-m_2)^2}{(z_1-z_2)^2(z_1+m_1-z_2-m_2)^2} +O_{\prec}( \eta^{-1}_0N^{\tau})+O(1)=-\frac{2}{\beta} \frac{1}{(z_1-z_2)^2} +O( \eta^{-1}_0N^{\tau}).$$
	
Therefore, recalling the definition of $\tilde{f}$ in (\ref{tilde_f}), by symmetry and (\ref{assumpf}), we have
	\begin{equation}\label{mid1}
	V_3=\frac{1}{ \beta \pi^2} \int_{\Gamma_1^+}  \int_{\Gamma_2^-} \frac{\tilde{f}(z_1) \tilde{f}(z_2) }{(z_1-z_2)^2} \dd z_1 \dd z_2+O(\eta_0N^{\tau}),
	\end{equation}
	with opposite integral directions on the contours. Since $\Gamma_1^+$ and $\Gamma_2^-$ are disjoint and $\tf$ has compact support, we obtain from Cauchy's integral theorem that
	$$\frac{1}{ \beta \pi^2} \int_{\Gamma_1^+}  \int_{\Gamma_2^-} \frac{\tilde{f}(z_2)^2}{(z_1-z_2)^2} \dd z_1 \dd z_2= \frac{1}{ \beta \pi^2}  \int_{\Gamma_2^-}\tilde{f}(z_2)^2  \Big( \int_{\Gamma_1^+} \frac{1}{(z_1-z_2)^2} \dd z_1\Big) \dd z_2=0.$$
	The integral of $\frac{\tilde{f}(z_1)^2}{(z_1-z_2)^2}$ vanishes similarly. Thus, we have from (\ref{mid1}) and (\ref{tilde_f}) that
	\begin{align}\label{mid2}
	V_3=&-\frac{1}{2 \beta \pi^2} \int_{\Gamma_1^+}  \int_{\Gamma_2^-} \frac{(\tilde{f}(z_1)- \tilde{f}(z_2)^2 }{(z_1-z_2)^2} \dd z_1 \dd z_2+O(\eta_0N^{\tau})\nonumber\\
	=& \frac{1}{2 \beta \pi^2} \int_{\R}  \int_{\R} \frac{(f(x_1) -f(x_2)+\i N^{-\tau} \eta_0 (f'(x_1)- f'(x_2)))^2}{(x_1-x_2+\frac{3 \ii }{2} N^{-\tau} \eta_0)^2} \dd x_1 \dd x_2+O(\eta_0N^{\tau}).
	\end{align}
	Changing the variable 
	\begin{equation}\label{change}
	\tilde x_1=\frac{x_1-E_0}{\eta_0}; \qquad \tilde x_2=\frac{x_2-E_0}{\eta_0},
	\end{equation}
	we hence obtain from (\ref{mid2}) that
	\begin{align*}
	V_3&=\frac{1}{2 \beta \pi^2} \int_{\R}  \int_{\R}  \frac{(g(\tilde x_1)- g(\tilde x_2) +\i N^{-\tau} (g'(\tilde x_1)-g'(\tilde x_2)))^2}{(\tilde x_1-\tilde x_2+\frac{ 3 \i}{2} N^{-\tau})^2 } \dd \tilde x_1  \dd \tilde x_2 +O(\eta_0N^{\tau}).
	\end{align*}
Note that the integrand can be bounded uniformly by
\begin{equation}
\Big|  \frac{(g(\tilde x_1)- g(\tilde x_2) +\i N^{-\tau} (g'(\tilde x_1)-g'(\tilde x_2)))^2}{(\tilde x_1-\tilde x_2+\frac{3 \i}{2} N^{-\tau})^2 } \Big| \leq \frac{(g(\tilde x_1)- g(\tilde x_2) )^2}{(\tilde x_1-\tilde x_2)^2}+\frac{(g'(\tilde x_1)- g'(\tilde x_2) )^2}{(\tilde x_1-\tilde x_2)^2}.
\end{equation}
	Since $g \in C_c^2(\R)$, we then conclude the proof using dominated convergence.
\end{proof}

\subsection{Near the edge}
Theorem \ref{mesoedge} is a result of the following lemma and Proposition \ref{prop}.
\begin{lemma}\label{bfedge}
Under the assumptions and notations of Theorem \ref{thm:weak_convergence},  we have
	$$\lim_{N \rightarrow \infty}V(f)=\frac{1}{ 2 \beta \pi^2} \int_{\R}\int_{\R} \Big(\frac{g(-x^2)-g(-y^2)}{x-y} \Big)^2 \dd x \dd y.$$
\end{lemma}

\begin{proof}[Proof of Lemma \ref{bfedge}]
	Similarly as in the bulk, using Stokes' formula, we have
	$$V(f)=-\frac{1}{4 \pi^2} \int_{\Gamma_1} \int_{\Gamma_2} \tf(z_1) \tf(z_2) (K_1+K_2+K_3) \dd z_1 \dd z_2:=V_1+V_2+V_3,$$
	where $K_1,K_2,K_3$ are given by (\ref{k1}) and (\ref{k3}), and we use the same notations and definitions as in the previous subsection. Using (\ref{derivative1})-(\ref{derivative3}) and Lemma \ref{m}, we have 
	$$|K_1+K_2| =O\Big( \frac{1}{\sqrt{|{\Im}z_1 {\Im}z_2}|}\Big)=O(N^{ \tau} \eta_0^{-1}).$$
	Using (\ref{assumpf}), one can show $|V_1+V_2|=O(\eta_0N^{\tau})$. Thus it is sufficient to study the integral involving $K_3$. Using (\ref{temp}), (\ref{derivative4}), and 
	$$\frac{\partial^2}{\partial z \partial z'} I(z,z')=\frac{\partial^2}{\partial z' \partial z} I(z,z')=\frac{(m_1'+m_2'+2m_1'm_2')(z_1-z_2)-(m_1'+m_2'+2)(m_1-m_2)}{(z_1+m_1-z_2-m_2)^3},$$
	by direct computation, we have
	$$K_3=\frac{2}{\beta} \Big( \frac{(1+m_1')(1+m_2')}{(z_1+m_1-z_2-m_2)^2}  -\frac{1}{(z_1-z_2)^2}\Big).$$
	
	For the second integrand, using the similar arguments as in the previous subsection, we have
	\begin{align}
	\lim_{N \rightarrow \infty} \int_{\Gamma_1^\pm} \int_{\Gamma_2^\pm} \frac{\tf(z_1) \tf(z_2)}{(z_1-z_2)^2} \dd z_1 \dd z_2=&-\frac{1}{2} \lim_{N \rightarrow \infty}  \int_{\Gamma_1^\pm} \int_{\Gamma_2^\pm} \frac{(\tf(z_1) -\tf(z_2))^2}{(z_1-z_2)^2} \dd z_1 \dd z_2\nonumber\\
	=&-\frac{1}{2} \int_{\R}\int_{\R} \frac{(g(x_1)-g(x_2))^2}{(x_1-x_2)^2} \dd x_1 \dd x_2.
	\end{align}
	Due to the opposite integral directions, we have
	$$\lim_{N \rightarrow \infty}  \int_{\Gamma_1^\pm} \int_{\Gamma_2^\mp} \frac{\tf(z_1) \tf(z_2)}{(z_1-z_2)^2} \dd z_1 \dd z_2=\frac{1}{2} \int_{\R}\int_{\R} \frac{(g(x_1)-g(x_2))^2}{(x_1-x_2)^2} \dd x_1 \dd x_2.$$
	The whole integral with respect to the second term of $K_3$ will hence vanish when $N \rightarrow \infty$. Thus it is sufficient to study the integral of the first term $\frac{(1+m_1')(1+m_2')}{(z_1+m_1-z_2-m_2)^2}$, that is,
	\begin{align*}
	V_3(f)&=-\frac{1}{2 \beta \pi^2} \Big( \int_{\Gamma_1^+} \int_{\Gamma_2^+}+\int_{\Gamma_1^-} \int_{\Gamma_2^-}+\int_{\Gamma_1^+} \int_{\Gamma_2^-}+\int_{\Gamma_1^-} \int_{\Gamma_2^+} \Big) \tf(z_1) \tf(z_2) \frac{(1+m_1')(1+m_2')}{(z_1+m_1-z_2-m_2)^2} \dd z_1 \dd z_2\\
	&:=V_3^{++}+V_3^{--}+V_3^{+-}+V_3^{-+}.
	\end{align*}

	Let $\zeta=z+m_{fc}(z)$ and $\zeta_{\pm}=L_{\pm}+m_{fc}(L_{\pm}) \in \R$. Define $F(\zeta):=\zeta-\frac{1}{N} \sum_{i=1}^N \frac{1}{a_i-\zeta}$ so that (\ref{self}) is equivalent to $z=F(\zeta)$. Assumptions 2.2 and 2.3 imply that there exists some constant $c_0$ independent of~$N$ such that
	$$\mbox{dist}( \{\zeta_{\pm},{\hat{\mathcal{I}}}\} ) \geq c_0,$$
	for all sufficiently large $N$, where ${\hat{\mathcal{I}}}$ is the smallest interval that contains the support of $\mu_{A}$; see (\ref{2}). Hence, $F(\zeta)$ is analytic in a neighborhood of $\zeta_{+}$, where we write
	$$F(\zeta)=F(\zeta_+)+F'(\zeta_+)(\zeta-\zeta_+) +\frac{F''(\zeta_+)}{2} (\zeta-\zeta_+)^2 +O(|\zeta-\zeta_+|^3).$$
	By (\ref{L}), $F'(\zeta_+)=1-\frac{1}{N}\sum_{i=1}^N \frac{1}{(a_i-\zeta_+)^2}=0$. Moreover, $F''(\zeta_+)=-\frac{2}{N}\sum_{i=1}^N \frac{1}{(a_i-\zeta_+)^3}$, and by (\ref{4}) it is bounded uniformly from below. In general, we have $|F^{(k)}(\zeta_+)|= \Big| \frac{k!}{N} \sum_{i=1}^N \frac{1}{(a_i-\zeta_+)^{(k+1)}}\Big| = O(1)$ because of (\ref{3}). Inverting $F(\zeta)=z$ in the neighborhood of $\zeta_+$, we have the expansion
	\begin{equation}\label{A+}
	\zeta=z+m_{fc}(z)=\zeta_++c_+ \sqrt{z-L_+} \Big(1+A_+(\sqrt{z-L_+}) \Big),
	\end{equation}
	where $A_+$ is an analytic function depending on $N$ with $A_+(0)=0$. This has been shown in the proof of Lemma 3.6 and  Lemma A.1 in \cite{bulk}.
	Note that $c_+= \Big(- \frac{1}{N} \sum_{i=1}^N \frac{1}{(a_i-\zeta_+)^3}\Big)^{-\frac{1}{2}}$ is some positive number depending on $N$ but is uniformly bounded. Furthermore, the coefficients of the expansion of $A_+$ are also uniformly bounded. Thus
	$$z+m_{fc}(z)=\zeta_++c_+ \sqrt{z-L_+} +O(|z-L_+|),$$
	where the square root is taken in a branch cut such that ${\Im} \sqrt{z-L_+}>0$ as ${\Im}z>0$. Similarly, we have
	$$1+m'_{fc}(z)=\frac{c_+ }{2\sqrt{z-L_+}}+d_++O(\sqrt{|z-L_+|}),$$
	where $d_+$ is some number which depends on $N$ but is uniformly bounded.
	Let $z=L_++\eta_0 x+\i N^{-\tau} \eta_0$. Then
	$$z+m_{fc}(z)=\zeta_++ c_+\sqrt{\eta_0 (x + \i N^{-\tau} ) } +O({\eta_0});~~1+m'_{fc}(z)=\frac{c_+}{2\sqrt{\eta_0 (x+ \i N^{-\tau})}}+d_++O(\sqrt{\eta_0 }).$$
Therefore, after changing the variable as in (\ref{change}), we have 
	\begin{align*}
	V_3^{++}&=-\frac{1}{8 \beta \pi^2} \int_{\R}\int_{\R}  \frac{ \tg(x_1) \tg(x_2) \Big(\frac{1}{\sqrt{x_1+\i N^{-\tau}}}+O(\sqrt{\eta_0}) \Big)\Big(\frac{1}{\sqrt{x_2+\frac{\i}{2} N^{-\tau}}}+O(\sqrt{\eta_0}) \Big)}{\Big(\sqrt{ x_1+\i N^{-\tau}}-\sqrt{ x_2+\frac{\i}{2} N^{-\tau}}+O( \sqrt{\eta_0})\Big)^2} \dd x_1 \dd x_2\\
	&=-\frac{1}{8 \beta \pi^2} \int_{\R}\int_{\R}  \frac{ \tg(x_1) \tg(x_2)}{\sqrt{x_1+\i N^{-\tau}}\sqrt{x_2+\frac{\i}{2} N^{-\tau}} \Big(\sqrt{ x_1+\i N^{-\tau}}-\sqrt{ x_2+\frac{\i}{2} N^{-\tau}}\Big)^2} \dd x_1 \dd x_2+O(\sqrt{\eta_0}N^{3\tau}),
	\end{align*}
	where $\tg(x)=g(x)+\i N^{-\tau} g'(x)$. The last step follows from the fact that $\Big| \sqrt{x_1+\i N^{-\tau}}-\sqrt{x_2+\frac{\i}{2} N^{-\tau}} \Big| \geq C N^{- \tau}$, when $x_1,x_2$ belong to some compact set.
	Let $\gamma_1^{\pm}:=\{ x_1 \pm \i N^{-\tau}: x_1 \in \R \}$ and $\gamma_2^{\pm}:=\{ x_2 \pm \frac{ \i}{2} N^{-\tau}: x_2 \in \R\}.$
	Then we obtain
	$$V_3^{++}=-\frac{1}{8 \beta \pi^2} \int_{\gamma_1^+} \int_{\gamma_2^+}  \frac{ \tg(z_1) \tg(z_2)}{\sqrt{z_1}\sqrt{z_2} (\sqrt{ z_1}-\sqrt{ z_2})^2} \dd z_1 \dd z_2+O(\sqrt{\eta_0 } N^{3\tau} ),$$
	where $\tg(x + \i y)=g(x)+\i y g'(x)\chi(y)$. Since $\gamma_1^{+}$ and $\gamma_2^{+}$ are disjoint and $\tg$ has compact support, changing the variable $w=\sqrt{z}$ and using Cauchy's integral theorem, we have
	$$\int_{\gamma_1^+} \int_{\gamma_2^+}  \frac{ \tg(z_1)^2 }{\sqrt{z_1}\sqrt{z_2} (\sqrt{ z_1}-\sqrt{ z_2})^2} \dd z_1 \dd z_2=0=\int_{\gamma_1^+} \int_{\gamma_2^+}  \frac{ \tg(z_2)^2 }{\sqrt{z_1}\sqrt{z_2} (\sqrt{ z_1}-\sqrt{ z_2})^2} \dd z_1 \dd z_2,$$
	and thus
	$$V_3^{++}=\frac{1}{16 \beta \pi^2} \int_{\gamma_1^+} \int_{\gamma_2^+}  \frac{ (\tg(z_1) -\tg(z_2))^2}{\sqrt{z_1}\sqrt{z_2} (\sqrt{ z_1}-\sqrt{ z_2})^2} \dd z_1 \dd z_2+O(\sqrt{\eta_0 } N^{3 \tau}).$$
	Therefore, we get
	$$\lim_{N \rightarrow \infty} V_3^{++} =\frac{1}{16 \beta \pi^2} \lim_{N \rightarrow \infty} \int_{\R}\int_{\R}  \frac{ (g(x_1)-g(x_2) +\i N^{-\tau} g'(x_1)- \frac{\i}{2} N^{-\tau} g'(x_2))^2 }{\sqrt{x_1+\i N^{-\tau}}\sqrt{x_2+\frac{\i}{2} N^{-\tau}} \Big(\sqrt{ x_1+\i N^{-\tau}}-\sqrt{ x_2+\frac{\i}{2} N^{-\tau}}\Big)^2} \dd x_1 \dd x_2.$$
	We denote the integrand as $h_N(x_1,x_2)$.
	Next, we interchange the limit and the integral . One shows that there exists $C>0$ such that
	$$\Big|\sqrt{ x_1+\i N^{-\tau}}-\sqrt{ x_2+\frac{\i}{2} N^{-\tau}} \Big| \geq C \Big| \sqrt{x_1}-\sqrt{x_2} \Big|.$$
	Set
	$$h(x_1,x_2):= C^{-2} \frac{(g(x_1)-g(x_2))^2+(g'(x_1)-g'(x_2))^2}{\sqrt{|x_1|} \sqrt{|x_2|} \Big|\sqrt{x_1}-\sqrt{x_2}\Big|^2},$$
	and observe that $|h_N(x_1,x_2)| \leq h(x_1,x_2)$. Next, we will show that $h(x_1,x_2)$ is integrable.
	
	Suppose $\mathrm{supp}(g)\subset[-M,M]$ for some $M>0$. 
	Then if $x_1$ and $x_2$ are both in $[-2M,2M]$ then we have the following estimation.\\
\textbf{Case 1:} If $x_1$, $x_2$ have the same sign, then
		\begin{align*}
		h(x_1,x_2)=&\frac{(g(x_1)-g(x_2))^2+(g'(x_1)-g'(x_2))^2}{\sqrt{x_1} \sqrt{x_2} ( \sqrt{x_1} -\sqrt{x_2})^2}\\
		=&\frac{1}{\sqrt{|x_1|} \sqrt{|x_2|} }\Big(\Big(\frac{g(x_1)-g(x_2)}{x_1-x_2}\Big)^2+\Big(\frac{g'(x_1)-g'(x_2)}{x_1-x_2}\Big)^2\Big)\big(\sqrt{|x_1|}+\sqrt{|x_2|}\big)^2\\
		\le&\frac{8M}{\sqrt{|x_1|} \sqrt{|x_2|} }(\|g'\|_\infty^2+\|g''\|_\infty^2).
		\end{align*}
		
\noindent \textbf{Case 2:} If $x_1$ and $x_2$ are of opposite signs, using $|x_1-x_2|=\big(\sqrt{|x_1|} -\i\sqrt{|x_2|}\big)\big(\sqrt{|x_1|} +\i\sqrt{|x_2|}\big)$, we have
		\begin{align*}
		h(x_1,x_2)=&\frac{(g(x_1)-g(x_2))^2+(g'(x_1)-g'(x_2))^2}{\sqrt{|x_1|} \sqrt{|x_2|} \big| \sqrt{|x_1|} -\i\sqrt{|x_2|}\big|^2}\\
		=&\frac{1}{\sqrt{|x_1|} \sqrt{|x_2|} }\Big(\Big(\frac{g(x_1)-g(x_2)}{x_1-x_2}\Big)^2+\Big(\frac{g'(x_1)-g'(x_2)}{x_1-x_2}\Big)^2\Big)\big|\sqrt{|x_1|}+\i\sqrt{|x_2|}\big|^2\\
		\le&\frac{8M}{\sqrt{|x_1|} \sqrt{|x_2|} }(\|g'\|_\infty^2+\|g''\|_\infty^2).
		\end{align*}
If $x_1\not\in[-2M,2M]$, then $x_2\in[-M,M]$ otherwise $h(x_1,x_2)=0$. So for $(x_1,x_2)\in[-2M,2M]^c\times[-M,M]$,
	\begin{align*}
	h(x_1,x_2) \le \frac{4\|g\|_\infty^2+4\|g'\|_\infty^2}{\sqrt{|x_1|}\sqrt{|x_2|}(\sqrt{|x_1|}-\sqrt{|x_2|})^2}\le \frac{4\|g\|_\infty^2+4\|g'\|_\infty^2}{\sqrt{|x_1|}\sqrt{|x_2|}(\sqrt{|x_1|}-\frac{1}{\sqrt2}\sqrt{|x_1|})^2}=\frac{C}{|x_1|^{3/2}|x_2|^{1/2}}.
	\end{align*}
	Therefore, $h(x_1,x_2)$ is integrable. Thus by dominated convergence, 
	\begin{align*}
	\lim_{N \rightarrow \infty} V_3^{++}&=\frac{1}{16 \beta \pi^2} \int_{\R}\int_{\R}  \frac{ (g(x_1)-g(x_2))^2 }{\sqrt{x_1+\i 0}\sqrt{x_2+\i 0} (\sqrt{ x_1+\i 0}-\sqrt{ x_2+\i 0})^2} \dd x_1 \dd x_2\\
	&=\frac{1}{4 \beta \pi^2} \int_{\psi(\R+\i 0)} \int_{\psi(\R+\i 0)} \frac{(g(w^2_1)-g(w^2_2))^2}{(w_1-w_2)^2} \dd w_1 \dd w_2,
	\end{align*}
	where we change the variable $\psi(z):=\sqrt{z}$; with branch cut such that $\psi: \C^+ \rightarrow \C^+$.
	
	Similarly, we have
	$$\lim_{N \rightarrow \infty} V_3^{--}=\frac{1}{4 \beta \pi^2} \int_{\psi(\R-\i 0)} \int_{\psi(\R-\i 0)}  \frac{(g(w^2_1)-g(w^2_2))^2}{(w_1-w_2)^2} \dd w_1 \dd w_2;$$
	$$\lim_{N \rightarrow \infty} V_3^{+-}=\frac{1}{4 \beta \pi^2} \int_{\psi(\R+\i 0)} \int_{\psi(\R-\i 0)}  \frac{(g(w^2_1)-g(w^2_2))^2}{(w_1-w_2)^2} \dd w_1 \dd w_2;$$
	$$\lim_{N \rightarrow \infty} V_3^{-+}=\frac{1}{4 \beta \pi^2} \int_{\psi(\R-\i 0)} \int_{\psi(\R+\i 0)}  \frac{(g(w^2_1)-g(w^2_2))^2}{(w_1-w_2)^2} \dd w_1 \dd w_2.$$
	
	The contours are shown in Figure \ref{figure11}. 
	\begin{figure}[h!]
		\centering
		\includegraphics[width=0.7\textwidth]{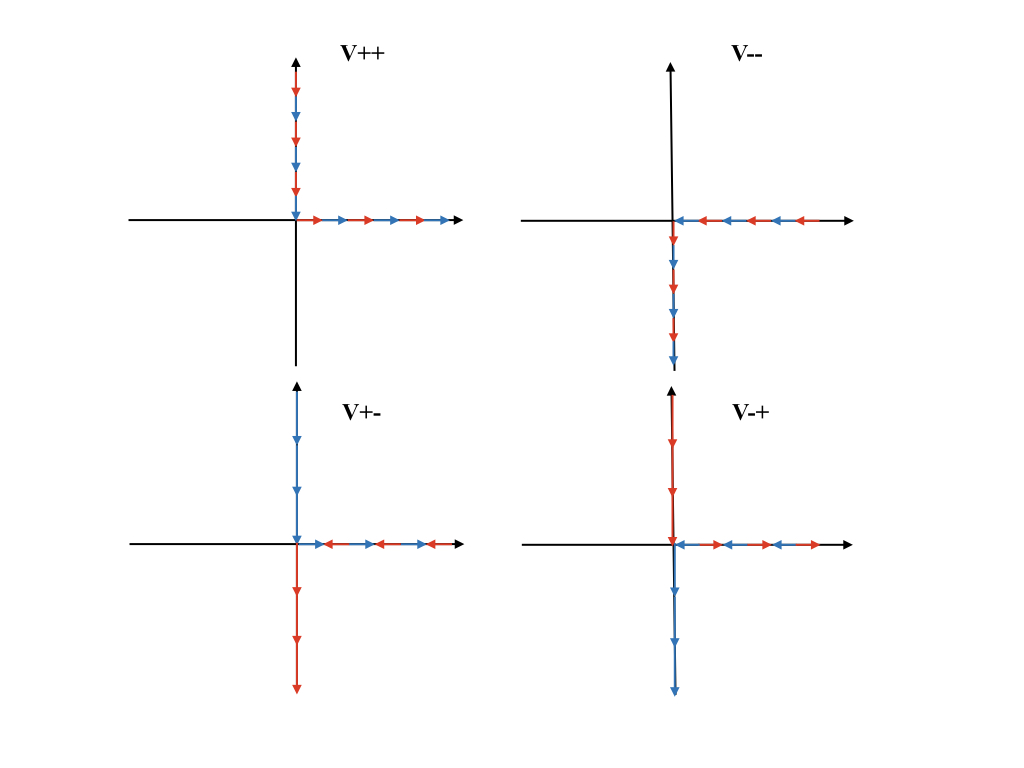}
		\caption{Integration contours in the variance term $V_3$}
		\label{figure11}
	\end{figure}
	Note that the horizontal parts of the blue and the red lines of above graph will cancel because of the opposite integral direction. To sum up, we have
	
	$$\lim_{N \rightarrow \infty} V_3=\frac{1}{4 \beta \pi^2} \int_{-\i \infty}^{\i \infty}\int_{-\i \infty}^{\i \infty} \frac{(g(w^2_1)-g(w^2_2))^2}{(w_1-w_2)^2}\dd w_1 \dd w_2=\frac{1}{4 \beta \pi^2} \int_{\R}\int_{\R} \Big(\frac{g(-x_1^2)-g(-x_2^2)}{x_1-x_2} \Big)^2 \dd x_1 \dd x_2.$$ 
	This concludes the proof of Theorem \ref{mesoedge}.
\end{proof}

\section{Proof of Proposition \ref{prop2} and computation of the bias}\label{sec:expectation}
In this section, we first prove Proposition \ref{prop2}, using the same technique as in Proposition~\ref{prop}. After this, we compute the bias on mesoscopic scales inside the bulk and at the edges.
\begin{proof}[Proof of Proposition \ref{prop2}]
We treat the expectation similarly using the cumulant expansion and (\ref{dH}):
$$(z-a_i) \E G_{ii} =\E (HG)_{ii}-1=\frac{1}{N} \E \sum_{j=1}^N c^{(2)}_{ij}  \frac{\partial G_{ji} }{\partial H_{ij}}-1 +\frac{1}{2! N^{\frac{3}{2}}}\sum_{j=1}^N c^{(3)}_{ij} \E \frac{\partial^2 G_{ji} }{\partial^2 H_{ij}} +\frac{1}{3! N^{2}}\sum_{j=1}^N c^{(4)}_{ij} \E\frac{\partial^3 G_{ji} }{\partial^3 H_{ij}} +O_{\prec}(N^{-\frac{3}{2}})$$
\begin{align*}
=&-\frac{1}{N} \sum_{j=1}^N \E G_{ii} G_{jj} -\frac{1}{N} \E (G^2)_{ii} -\frac{m_2-2}{N} \E (G_{ii})^2-1+\frac{1}{2 N^{\frac{3}{2}}}\sum_{j=1}^N c^{(3)}_{ij} \E \Big(6 G_{ii} G_{ij} G_{jj}+2 G^3_{ij}\Big)\\
&+\frac{1}{6 N^{2}}\sum_{j=1}^N (W_4-3) \E \Big(-36 G_{ii} G_{jj} G^2_{ij}-6 G^2_{ii}G^2_{jj}-6 G^4_{ij} \Big)+O_{\prec}(N^{-\frac{3}{2}}).
\end{align*}
Combining with the local law, we have
$$(z-a_i) \E G_{ii}=-\frac{1}{N}  \E G_{ii} \Tr G -\frac{1}{N} \frac{d}{dz}\frac{1}{a_i-z-m_{fc}} -\frac{m_2-2}{N} \frac{1}{(a_i-z-m_{fc})^2}-1$$
$$+\frac{3}{ N^{\frac{3}{2}}}\sum_{j=1}^N   \frac{c^{(3)}_{ij}}{(a_i-z-m_{fc})(a_j-z-m_{fc})} G_{ij}-\frac{1}{ N^{2}}\sum_{j=1}^N (W_4-3)  \frac{1}{(a_i-z-m_{fc})^2(a_j-z-m_{fc})^2}+O_{\prec}\Big( \frac{\Psi(z)}{N \eta}\Big).$$

Using the anisotropic local law and the argument as in (\ref{estimate}), one can show that the second term of the last line of above equation is $O_{\prec}( N^{-1}\Psi(z)).$
Therefore, we have
$$(z-a_i) \E G_{ii}=-\frac{1}{N}  \E  \Big( G_{ii}-\frac{1}{a_i-z-m_{fc}} \Big) \Tr G -\frac{1}{N}\frac{1}{a_i-z-m_{fc}} \E \Tr G-1-\frac{1}{N} \frac{1+m'_{fc}(z)}{(a_i-z-m_{fc})^2} $$
$$ -\frac{m_2-2}{N} \frac{1}{(a_i-z-m_{fc})^2}-\frac{1}{ N}  I_s(z) \frac{W_4-3}{(a_i-z-m_{fc})^2}+O_{\prec}\Big( \frac{\Psi(z)}{N \eta}\Big),$$
and thus
\begin{align*}
(z-a_i+m_{fc})\Big( \E G_{ii}-\frac{1}{a_i-z-m_{fc}} \Big)&=-\frac{1}{N}\frac{1}{a_i-z-m_{fc}}  \Big( \E \Tr G-N m_{fc} \Big)\\
-\frac{1}{N} \frac{1+m'_{fc}(z)}{(a_i-z-m_{fc})^2} &-\frac{m_2-2}{N} \frac{1}{(a_i-z-m_{fc})^2}-\frac{1}{ N}  I_s(z) \frac{W_4-3}{(a_i-z-m_{fc})^2}+O_{\prec}\Big( \frac{\Psi(z)}{N \eta}\Big).
\end{align*}
Dividing both sides by $a_i-z-m_{fc} \sim O(1)$ and summing over $i$, we obtain
\begin{align*}
(1-I_s(z))\E (\Tr G -N m_{fc})=&\frac{1}{N} \sum_{i=1}^N \frac{1+m'_{fc}(z)}{(a_i-z-m_{fc})^3} +\frac{m_2-2}{N} \sum_{i=1}^N \frac{1}{(a_i-z-m_{fc})^3}\\
&+\frac{W_4-3}{ N}  \sum_{i=1}^N \frac{I_s(z)}{(a_i-z-m_{fc})^3}+O_{\prec}\Big( \frac{\Psi(z)}{ \eta}\Big).
\end{align*}
Dividing both sides by $1-I_s(z)$  and using the relation $1-I_s(z)=\frac{1}{1+m'_{fc}(z)} \sim {\sqrt{\kappa+\eta}}$, we obtain
$$\E (\Tr G-N m_{fc})=\frac{1}{1-I_s(z)}\frac{1}{2} \frac{\dd I_s(z)}{\dd z} +\frac{m_2-2}{2} \frac{\dd I_s(z)}{\dd z} +\frac{W_4-3}{2} I_s(z)\frac{\dd I_s(z)}{\dd z}+O_{\prec}\Big( \frac{1}{\eta\sqrt{N \eta \sqrt{\kappa+\eta}}}\Big).$$
Plugging into (\ref{fw2}) (here we replace $\E \mu_N$ by $\mu_{fc}$), using Lemma \ref{lemma3} and Stokes' formula, we have
$$\E \Tr f(X_N)-N \int_{\R} f(x) \rho_{fc}(x) \dd x=\frac{1}{4 \pi \i} \int_{\partial \Omega_0} \tf(z) b(z) \dd z+O_{\prec}\Big( \frac{N^{2 \tau}}{\sqrt{N \eta_0 \sqrt{\kappa_0+\eta_0}}}\Big)+O_{\prec}(N^{-\tau}),$$
where $b(z)$ is given by (\ref{bz}).
Using the relation $I_s=\frac{m'_{fc}}{1+m'_{fc}}$, it coincides with the expectation that obtained in the global CLT given in Theorem \ref{global}. Thus we conclude the proof of Proposition \ref{prop2}.
\end{proof}

Next, we explicitly compute the bias in the bulk and at the edges, for the scaled test function in (\ref{fn}). 

\subsection{Bias in the mesoscopic bulk}
Note that
\begin{equation}\label{useful}
\frac{dI_s}{dz}=\frac{2}{N} \sum_{i=1}^N \frac{1+m'_{fc}}{(a_i-z-m_{fc})^3}= O\Big( \frac{1}{\sqrt{\kappa+\eta}}\Big); \qquad 1-I_s(z) \sim \frac{1}{\sqrt{\kappa+\eta}}; \qquad |I_s(z)| = O(1).
\end{equation}

If $\kappa \geq \kappa_0> c>0$, then $|b(z)|=O(1)$. In combination with (\ref{assumpf}), we have 
$$\E \Tr f(X_N)-N \int_{\R} f(x) \rho_{fc}(x) \dd x =   O_{\prec}\Big(\eta_0+\frac{N^{2\tau}}{\sqrt{N \eta_0 \sqrt{\kappa_0+\eta_0}}}+N^{-\tau} \Big),$$
hence we see that the bias vanishes as $N$ goes to infinity.

\subsection{Bias at the mesoscopic edge}
Similarly, using (\ref{useful}) and (\ref{assumpf}), the last two terms of $b(z)$ will contribute $O_{\prec}(\sqrt{N^{\tau}\eta_0})$. We have
\begin{align*}
\E \Tr f(X_N)-N \int_{\R} f(x) \rho_{fc}(x) \dd x=\frac{1}{4 \pi \i } \int_{\partial \Omega_0}  \tf(z) \frac{m_{fc}''}{1+m_{fc}'} \dd z+&O_{\prec}\Big(N^{-\tau}+ \frac{N^{2\tau}}{\sqrt{N \eta_0 \sqrt{\kappa_0+\eta_0}}}+\sqrt{N^{\tau}\eta_0}\Big).
\end{align*}
Using (\ref{A+}), we obtain the following expansions:
$$1+m'_{fc}(z)=\frac{c_+ }{2\sqrt{z-L_+}}+O(1),~~~m''_{fc}(z)=-\frac{c_+ }{4(\sqrt{z-L_+})^3}+ O \Big( \frac{1}{\sqrt{|z-L_+|}} \Big),$$
and then
$$\frac{m_{fc}''}{1+m_{fc}'} =-\frac{1}{2(z-L_+)}+O\Big( \frac{1}{\sqrt{|z-L_+|}} \Big).$$
Changing variables and using (\ref{assumpf}), we have
\begin{align*}
\E \Tr f(X_N)-&N \int_{\R} f(x) \rho_{fc}(x) \dd x=-\frac{1}{8\pi\i} \int_{\R} (g(x)+\i N^{-\tau}g'(x)) \frac{1}{x+\i N^{-\tau}} \dd x\\
&+\frac{1}{8\pi\i} \int_{\R} (g(x)-\i N^{-\tau}g'(x)) \frac{1}{x-\i N^{-\tau}} \dd x+O_{\prec}\Big(N^{-\tau}+ \frac{N^{2\tau}}{\sqrt{N \eta_0 \sqrt{\kappa_0+\eta_0}}}+\sqrt{N^{\tau}\eta_0}\Big)\\
=&-\frac{1}{8\pi\i} \int_{\R}  \frac{g(x)}{x+\i N^{-\tau}} \dd x+\frac{1}{8\pi\i} \int_{\R}  \frac{g(x)}{x-\i N^{-\tau}} \dd x+O_{\prec}\Big(N^{-\tau}+ \frac{N^{2\tau}}{\sqrt{N \eta_0 \sqrt{\kappa_0+\eta_0}}}+\sqrt{N^{\tau}\eta_0}\Big).
\end{align*}
Using the Sokhotski-Plemelj lemma, we have
\begin{align*}
\E \Tr f(X_N)-N \int_{\R} f(x) \rho_{fc}(x) \dd x=\frac{g(0)}{4}+O_{\prec}\Big(N^{-\tau}+ \frac{N^{2\tau}}{\sqrt{N \eta_0 \sqrt{\kappa_0+\eta_0}}}+\sqrt{N^{\tau}\eta_0}\Big),
\end{align*}
where we used the regularity $g \in C_c^2(\R)$. This finishes the computation of mesoscopic bias.

\section{Sample covariance matrix}\label{sec:sample_covariance}

In this section, we use the previous arguments to derive the mesoscopic eigenvalue statistics of sample covariance matrix and prove similar CLTs in the bulk and at the regular edges. We start by introducing the model in detail.

\subsection{Setup, assumptions and main results}
 Let $X_N=(X_{ij})$ be an $M \times N$ matrix satisfying the following assumption.
\begin{assumption}\label{assumption_sample_1}
\begin{enumerate}
\item $\{X_{ij}| 1 \leq i \leq M, 1 \leq j \leq N\}$ are independent real-valued centered random variables.
\item For all $i,j$, we have $\E |\sqrt{N}X_{ij}|^2=1$. In addition, $\sqrt{N}X_{ij}$ has uniformly bounded moments, that is, there exists $C_p>0$ independent of $N$ such that for all $i,j$,
\begin{equation}\label{moment_condition_sample}
\E |\sqrt{N} X_{ij}|^{p} \leq C_{p}.
\end{equation}
\item To simplify the statement,  we also assume that there exists a constant $K_4$ such that
\begin{equation}\label{k_4}
K_4:=\frac{1}{N} \sum_{j=1}^N {c^{(4)}_{ij}}, \mbox{ where $c^{(4)}_{ij}$ is the fourth cumulant of $\sqrt{N} X_{ij}$.}
\end{equation}
\end{enumerate}
\end{assumption}
Note that $M$ depends on $N$ and set 
\begin{equation}\label{ratio}
\gamma \equiv \gamma_N:=\frac{M}{N} \rightarrow \gamma_0, \qquad  0< \gamma_0 <\infty.
\end{equation}
We study the $M \times M$ sample covariance matrices 
\begin{equation}\label{H_sample}
H_M:=Y_NY^*_N,\qquad Y_N:=\Sigma^{1/2} X_N, \qquad \Sigma:=\mbox{Diag}(\sigma_i),
\end{equation}
where $\Sigma$ is an $M \times M$ positive definite, deterministic and diagonal matrix with 
\begin{equation}\label{condition_sample}
\infty > \sigma_1 \geq \sigma_2 \geq \cdots \geq \sigma_{M} > 0, \qquad \limsup_{N \rightarrow \infty} \sigma_{1} <\infty, \qquad  \liminf_{N \rightarrow \infty} \sigma_{M}>0.
\end{equation}
We denote by $\mu_{\Sigma}$ the empirical eigenvalue distribution of $\Sigma$, i.e. $\mu_{\Sigma}:=\frac{1}{M} \sum_{j=1}^M \delta_{\sigma_j}$. The following assumption ensures that the limit of $\mu_{\Sigma}$ exists.
\begin{assumption}\label{assumption_sample_2}
Together with (\ref{condition_sample}), $\mu_{\Sigma}$  converges weakly to a deterministic measure $\mu_\sigma$ as $N \rightarrow \infty$ such that $\mu_\sigma$ is compactly supported in $(0, \infty)$.
\end{assumption}
 The eigenvalues of $H \equiv H_M$ are denoted as $\lambda_i \in \R$, $1 \leq i \leq M$. The empirical spectral measure of $H_M$ is defined by $\mu_M=\frac{1}{M} \sum_{i=1}^M \delta_{\lambda_i}$. 
 The Stieltjes transform of $ \mu_M$ is then given by
\begin{equation}\label{m_sample}
m_M(z):= M^{-1} \Tr G(z),\qquad \mbox{where }G(z):=(H_M-zI)^{-1}, \quad z \in \C^+.
\end{equation}
We further introduce the $N \times N$ matrices
\begin{equation}\label{mathcalH}
\mathcal{H}_N:=Y_N^* Y_N, \qquad \mathcal{G}:=(\mathcal{H}-z)^{-1}, \qquad \mathfrak{m}_{N}(z):=N^{-1} \Tr \mathcal{G}, \quad z \in \C^+.
\end{equation}
The eigenvalues of $\mathcal{H} \equiv \mathcal{H}_N $ are denoted by $\{\mu_{i}\}_{i=1}^N$. It is straightforward that $\{\lambda_{i}\}_{i=1}^M$ differs from $\{\mu_{i}\}_{i=1}^N$ by $|N-M|$ zeros, hence we have the relation
\begin{equation}\label{gamma_N}
\mathfrak{m}_{N}(z)=\gamma m_{M}+\frac{\gamma-1}{z}.
\end{equation}

In the null case $\Sigma=I$, the Marchenko-Pastur law states that the empirical eigenvalue distribution of $H=XX^*$ converges weakly to the Marchenko-Pastur distribution with aspect ratio $\gamma_0$, whose density is given by $\dd\mu_{MP, \gamma_0}:=\frac{1}{2 \pi \gamma_0} \sqrt{\frac{[(x-\gamma_-)(\gamma_+-x)]_{+}}{x^2}} \dd x+(1-\gamma_0^{-1})_+ \delta_0$ with $\gamma_{\pm}=(1 \pm \sqrt{\gamma_0})^2$. Its Stieltjes transform $m_{MP, \gamma_0}$, or denoted by $m_{\gamma_0}$ for short, is characterized as the unique solution~of 
\begin{equation}\label{solution_sample}
1+ (z-1+\gamma_0) m(z) +\gamma_0 zm^2(z)=0, \quad \mbox{or equivalently} \quad m(z)=\frac{1}{1-\gamma_0 -\gamma_0 z m(z)-z},
\end{equation}
such that $\Im m(z)>0$, $z \in \C^+$. Because of (\ref{gamma_N}), the Stieltjes transform of the limiting spectral measure of $\mathcal{H}=X^*X$, denoted by $\mr_{\gamma_0^{-1}}$, is then given by
\begin{equation}\label{gamma}
\mathfrak{m}_{\gamma_0^{-1}}(z)=\gamma_0 m_{\gamma_0}(z)+\frac{\gamma_0-1}{z}.
\end{equation}

In the non-null case $\Sigma \neq I$, under Assumption \ref{assumption_sample_2}, the limiting spectral measure of $H=\Sigma^{1/2} X X^* \Sigma^{1/2}$ exists, henceforth referred to as the deformed Marchenko-Pastur law. Its Stieltjes transform, denoted by $m_{fc, \gamma_0}$, or $m_{fc}$ for short, is the unique solution of
\begin{equation}\label{mplawgamma}
m(z)=\int_{\R} \frac{1}{t(1-\gamma_0 -\gamma_0 z m(z))-z} \dd \mu_\sigma(t),
\end{equation}
such that $\Im m(z)>0$, $z \in \C^+$. The corresponding limiting measure, denoted by $\mu_{fc,\gamma_0}$ or $\mu_{fc}$ for short, is the free multiplicative convolution of $\mu_\sigma$ and the standard Marchenko-Pastur law with ratio $\gamma_0$, i.e., $\mu_{fc, \gamma_0}=\mu_\sigma \boxtimes \mu_{MP, \gamma_0}$; see \cite{freetimes, freetimes2}. It was proved in \cite{Silverstein+Choi} that the free multiplicative convolution measure is absolutely continuous and its density function is analytic whenever positive in $(0, \infty)$.

According to (\ref{gamma}), the Stieltjes transform of the limiting spectral measure of $\mathcal{H}=X^* \Sigma X$ is the unique solution~to
\begin{equation}\label{mplawgamma-11}
{\mr} (z)=\frac{1}{-z+\gamma_0 \int_{\R} \frac{t}{t{\mr(z)}+1} \dd \mu_\sigma(t)}, \quad \mbox{or equivalently} \quad \gamma_0-1 -z \mr(z)= \int_{\R} \frac{\gamma_0}{1+ t \mr(z)} \dd \mu_\sigma(t),
\end{equation}
such that $\Im \mr(z)>0$, $z \in \C^+$.

 For general $\Sigma$, $\mu_{fc}$ could be supported on several disjoint intervals; we refer to \cite{Hachem+Hardy+Najim_1, Silverstein+Choi, isotropic} for discussions on the support of the density. The following assumption ensures that the free multiplicative convolution measure is supported on a single interval and the edges behave like square roots. It also rules out the possibilities of outliers. 
\begin{assumption}\label{assumption_sample_3}
	Let $[\sigma_-, \sigma_+] \subset \R^+$ be the smallest interval that contains the support of $\mu_\sigma$ and set $\mathcal{I}:=[\sigma_+^{-1}, \sigma_-^{-1}]$. Assume that
	\begin{equation*}
         \inf_{x \in\mathcal{I}} \int_{\R} \Big( \frac{t x}{1-t x} \Big)^2 \dd \mu_\sigma(t)\geq \gamma_0^{-1}+w,
	\end{equation*}
	for some constant $w>0$ (the left side may be infinite). Similarly, set $\hat{\mathcal{I}}:=[ \sigma_1^{-1}, \sigma_M^{-1}]$. Assume that
	$$\inf_{x \in \hat{\mathcal{I}}} \int_{\R} \Big( \frac{t x}{1-t x} \Big)^2 \dd \mu_{\Sigma}(t)\geq \gamma^{-1}+w,$$
	for sufficiently large $N$.
\end{assumption}

Let $\xi=-\mr(z)$ so that (\ref{mplawgamma-11}) is equivalent to 
$$z=F(\xi), \qquad F(\xi):=\frac{1}{\xi} + \gamma_0 \int_{\R} \frac{t }{1-t \xi} \dd \mu_\sigma(t).$$ 
As an analogue of (\ref{L}), it was argued in \cite{Silverstein+Choi} that the edges of the support of $\mu_{fc}$, denoted as $E_{\pm}$, are given by
$E_\pm=F(\xi_\pm)$, where $\xi_{\pm} \in \R$ are the solutions to
\begin{equation}\label{solutionedge_sample}
H(\xi):=\int_{\R} \Big( \frac{t \xi}{1-t \xi} \Big)^2 \dd \mu_{\sigma}(t)=\gamma_0^{-1}.
\end{equation}
Under Assumption \ref{assumption_sample_3}, we have at most two solutions of (\ref{solutionedge_sample}), since $H(\xi)$ is monotone outside $[\sigma_+^{-1}, \sigma_-^{-1}]$. Let $\xi_+$ be the unique solution of (\ref{solutionedge_sample}) in $ (0, \sigma^{-1}_{+})$. The right boundary of the spectrum is given by $E_+=F(\xi_+)$. As for the left edge, we split into three cases. If $0<\gamma_0<1$, there is a unique solution of (\ref{solutionedge_sample}) in the interval $(\sigma_{-}^{-1}, \infty)$, denoted by $\xi_-$, and the corresponding left edge is given by $E_-=F(\xi_-)$. If $\gamma_0>1$, similarly, there is a unique solution of (\ref{solutionedge_sample}) in the interval $(-\infty,0)$. These edges are referred to as soft edges. For $\gamma_0=1$, the solution does not exist (or say $\xi_-=\infty$), corresponding to $E_-=0$. This scenario is referred to as the hard edge and the density there goes to infinity at rate of $\kappa^{-1/2}$. In this paper, we only consider the right edge for all $0<\gamma_0 < \infty$. Same discussion easily extends to the soft left edges when $\gamma_0 \neq 1$. 

In addition, Assumption \ref{assumption_sample_3} implies that  
\begin{equation}\label{gap_sampleN}
\mbox{dist}( \{ \sigma^{-1}_+, \sigma^{-1}_-\}, \xi_{\pm}) \geq c_0>0,
\end{equation}
provided that $\xi_-$ is finite ($\gamma_0 \neq 1$). The above condition is crucial for the density of the limiting measure to have the square root behavior at the soft edges. It will be proved later in Lemma \ref{sample_m}. Similar assumptions also appeared in \cite{ Bao+Pan+Zhou, edgecondition, sample_lee_kevin} for the rightmost edge and \cite{Hachem+Hardy+Najim_1, isotropic} for all edges in multi-cuts.

Since the convergence rate of $\mu_{\Sigma}$ and $\gamma$ could be very slow, from now on, we work with the finite-N version $\mu_{fc}=\mu_{MP, \gamma} \boxtimes \mu_{\Sigma}$. The corresponding Stieltjes transforms are given by (\ref{mplawgamma}) and (\ref{mplawgamma-11}) replacing the limiting measure $\mu_\sigma$ by $\mu_{\Sigma}$ and $\gamma_0$ by $\gamma$. The following notations, e.g., $m_{fc}$, $\mr$, $E_\pm$, $\xi_\pm$ and $\kappa$ are corresponding to $\mu_{MP, \gamma} \boxtimes \mu_{\Sigma}$ and are $N$-dependent.  To be consistent with the previous sections, we will use the tilde sign to denote the ones with respect to $\mu_{MP, \gamma_0} \boxtimes \mu_\sigma$. 
The second condition of Assumption \ref{assumption_sample_3} ensures the same properties for the support of $\mu_{MP, \gamma_{0}} \boxtimes \mu_{\Sigma}$. In particular, for sufficiently large $N$ we have
\begin{equation}\label{gap_sample}
\min_{i} \{|1- \xi_\pm \sigma_i|\} \geq c_0>0,
\end{equation}
if $\gamma_{0} \neq 1$. If $\gamma_0=1$, it only holds true with respect to $\xi_+$.

Next, we state the local law for the Green function of sample covariance matrix, which is an essential tool in our proof. Let $\mr \equiv \mr(z)$ be the unique solution of finite-N version of (\ref{mplawgamma-11}), i.e.,
\begin{equation}\label{mplawgamma-1}
\gamma-1 -z \mr(z)=\frac{1}{M} \sum_{i=1}^M \frac{\gamma}{1+ \sigma_i \mr(z)},
\end{equation}
such that $\Im \mr(z)>0$, $z\in \C^+$. Define the deterministic control parameters
\begin{equation}\label{control_sample}
\Psi(z):=\sqrt{ \frac{\im \mathfrak m(z)}{N |\eta|}} +\frac{1}{N |\eta|}\,,\qquad\Theta(z):=\frac{1}{N |\eta|}\,, \qquad z=E+\ii \eta \in \C \setminus \R.
\end{equation}
We also introduce the spectral domain, for some small $c>0$,
\begin{equation}\label{ddd_sample}
S':=\Big\{z=E+\ii \eta:  |E| \leq c^{-1}, N^{-1+c} \leq \eta \leq  c^{-1},|z| \geq c \Big\}.
\end{equation}
We further introduce the $N+M$ by $N+M$ matrices
$$R:=\begin{pmatrix}
-\Sigma^{-1} & X\\
X^* & -z
\end{pmatrix}^{-1}; \qquad
\Pi:=\begin{pmatrix}
-\Sigma (1+\mathfrak m\Sigma)^{-1} & 0\\
0 & \mathfrak m(z)
\end{pmatrix}; \qquad
\Sigma':=\begin{pmatrix}
\Sigma  & 0\\
0 & I
\end{pmatrix}, \quad z \in \C^+.$$
Using the Schur decomposition/Feshbach formula, we see that
$$R=\begin{pmatrix}
z \Sigma^{1/2} {G} \Sigma^{1/2} & \Sigma X \mathcal G \\
 \mathcal G X^* \Sigma  & \mathcal G
\end{pmatrix}=
\begin{pmatrix}
z \Sigma^{1/2} {G} \Sigma^{1/2} & \Sigma^{1/2} {G}Y \\
Y^* {G} \Sigma^{1/2}  & \mathcal G
\end{pmatrix}.
$$

We are ready to state the (anisotropic) local law for such random matrix.
\begin{theorem}\label{fluctuation}(Theorem 2.4 in \cite{Alex+Erdos+Knowles+Yau+Yin}, Theorem 3.6 in \cite{isotropic})
	For any deterministic unit vector $v,w \in \C^N$, we have
	$$\Big| \langle v, \Sigma'^{-1}\big( R(z)-\Pi(z) \big)\Sigma'^{-1} w \rangle   \Big| \prec \Psi(z),$$ 
	uniformly in $z  \in S'$. It also implies that
	$$|(\mathcal G(z))_{ij}-\mathfrak m(z) \delta_{ij}| \prec \Psi(z); \qquad \Big|  (G(z))_{ij}+\frac{1}{z(1+\mathfrak m \sigma_{i})} \delta_{ij} \Big| \prec \Psi(z).$$
	In addition, we have the averaged result
	$$\Big| N^{-1} \Tr \mathcal G(z)-\mathfrak m  (z)\Big| \prec  \Theta(z), \qquad \Big| M^{-1} \Tr G(z)+  \int_{\R} \frac{1}{z(1+\mathfrak mt)} \dd \mu_\Sigma(t) \Big| \prec  \Theta(z).$$
\end{theorem}

In the following, we state some properties of the Stieltjes transform $\mathfrak m$ in (\ref{mplawgamma-11}), whose proofs are given in Appendix B. Define the spectral domain, for some small $c>0$,
$$S:=\Big\{z=E+\ii \eta:  |E| \leq c^{-1}, 0 < \eta \leq  c^{-1},|z| \geq c \Big\}.$$
And set $\kappa \equiv \kappa(E):=\min \{ |E_+-E|, |E_--E| \}.$

\begin{lemma}\label{sample_m}
\begin{enumerate}
\item For $z \in S$ and sufficiently large $N$, we have 
	\begin{equation}\label{1+tm}
	|\mathfrak m(z)| \sim 1; \qquad \min_i |1+ \sigma_i \mathfrak m(z)|>c_0.
	\end{equation}	
\item	For $z \in S$ and sufficiently large $N$, we have 
	\begin{equation}\label{im_m_fc}
		|\Im \mr(z)|  \sim \begin{cases}
		\sqrt{\kappa+\eta}, & \mbox{if } E \in [E_-,E_+], \\
		\frac{\eta}{\sqrt{\kappa+\eta}}, & \mbox{if otherwise}.
		\end{cases}
		\end{equation}
\item	 For $z \in S$ with $E_--c<E< E_++c$ and $\eta \leq c$ for some small $c>0$, we have
	\begin{equation}\label{coefficient}
	1-\frac{1}{M} \sum_{i=1}^M \frac{\gamma \sigma_i}{z(1+\mathfrak{m}(z) \sigma_i)^2}=-\frac{\mr(z)} {z \mr'(z)}\sim \sqrt{\kappa+\eta}.
	\end{equation}
\item	Under the same condition as in (3), we have
         \begin{equation}\label{prime_gamma}
	 \mr'(z)=-\frac{\mr(z)}{z-\frac{1}{M} \sum_{i=1}^M \frac{\gamma \sigma_i}{(1+\mathfrak{m}(z) \sigma_i)^2}} \sim \frac{1}{\sqrt{\kappa+\eta}}.
	\end{equation}
\end{enumerate}
\end{lemma}

We are now prepared to state our main results for the sample covariance matrix.

\begin{proposition}\label{clt_sample}
Consider a sample covariance matrix satisfying Assumption \ref{assumption_sample_1}, \ref{assumption_sample_2} and \ref{assumption_sample_3} and $E_0$ is chosen to be away from zero, then Proposition \ref{prop} and \ref{prop2} hold true with 
\begin{equation}\label{kernel_sample}
K(z_1,z_2)=2 \Big( \frac{\mr_1' \mr_2'}{(\mr_1-\mr_2)^2}-\frac{1}{(z_1-z_2)^2} \Big)+K_4 \gamma  \pzab \Big( \frac{1}{M} \sum_{i=1}^M \frac{1}{(1+\mathfrak m_1 \sigma_i)(1+\mathfrak m_2 \sigma_i)}    \Big),
\end{equation}
where we use $\mr_1$ and $\mr_2$ to denote $\mr(z_1)$ and $\mr(z_2)$, and 
\begin{equation}\label{bz_sample}
b(z)=\Big( \frac{ (\mr'(z))^2}{\mr(z)} +K_4 \mr(z) \mr'(z)\Big) \frac{1}{M} \sum_{i=1}^{M} \frac{\gamma \sigma_i^2 }{(1+\mr(z) \sigma_i)^3}.
\end{equation}
\end{proposition}

Proposition \ref{clt_sample} implies that Theorem \ref{meso} and \ref{mesoedge} hold true for sample covariance matrix. More specifically, we have the following theorem.
\begin{theorem}\label{meso_sample}
	Let $H_N$ be a sample covariance matrix of the form (\ref{H_sample}) satisfying Assumptions \ref{assumption_sample_1}-\ref{assumption_sample_3}. Let $N^{-1+c_1} \leq \eta_0 \leq N^{-c_1}$ with some $c_1>0$ and fix $E_0 \in (E_-,E_+)$, such that~$\kappa_0:=\mathrm{dist}(\mathrm{supp}(f_N),\{E_\pm\}) > c_0$, for some $c_0>0$ and sufficiently large $N$. Then, for any function $g \in C^2_c(\R)$, the linear eigenvalue statistics (\ref{linear_stat})
	converges in distribution to the Gaussian random variable $\mathcal{N}\Big(0, \frac{1}{ \pi} \int_{\R} |\xi| |\hat{g}(\xi)|^2 \dd \xi \Big)$,
	where $\hat{g}(\xi):=(2 \pi)^{-1/2} \int_{\R} g(x) e^{-\i \xi x} \dd x$. 

	In addition, the linear statistics (\ref{linear_stat}) with $E_0=E_+$ and $N^{-\frac{2}{3}+c_2} \leq \eta_0 \leq N^{-c_2}$ for some $c_2>0$, converges in distribution to a Gaussian random variable $\mathcal{N} \Big( \frac{g(0)}{4}, \frac{1}{2 \pi} \int_{\R} |\xi| |\hat{h}(\xi)|^2 \dd \xi \Big)$,
	where $h(x)=g(- x^2)$ and $\hat{h}(\xi):=(2 \pi)^{-1/2} \int_{\R} h(x) e^{-\i \xi x} \dd x$. 
	Furthermore, if $\gamma_0 \neq 1$, a similar CLT for $E_0=E_-$ can be obtained with $h(x)=g(x^2)$.
\end{theorem}

{\it Remark:} We remark that (\ref{k_4}) in Assumption \ref{assumption_sample_1}  can  be removed. In addition, we can relax the single support condition for $\mu_{fc}$ by assuming instead that the cuts of the support of $\mu_{fc}$ are separated by order one and the density has square root behaviors at the edges away from zero.

\subsection{Proof of the CLT and variance computation}

	From the definition of the Green function and (\ref{H_sample}), we get
	\begin{equation}\label{uselater_sample}
	z {G}_{ii}=(HG)_{ii}-1=(\Sigma^{1/2} X X^* \Sigma^{1/2} {G})_{ii}-1=\sqrt{\sigma_i} \sum_{j=1}^N  X_{ij} (G Y)_{ij} -1.
	\end{equation}
	Similarly as (\ref{sum}), by the cumulant expansion formula, we have
\begin{align}\label{sum_sample}
z\E [\ea(G_{ii}-\E  G_{ii})]=I_1+I_2+I_3+O_{\prec}(N^{-\frac{3}{2}} (1+|\lambda|^4)),
\end{align}
where
\begin{align*}
I_1:&=\frac{\sqrt{\sigma_i}}{N}\sum_{j=1}^N  c^{(2)}_{ij} \bigg( \E \Big[\frac{\partial  \ea}{\partial X_{ij}} ( G Y)_{ij} \Big]+\E \Big[ \Big( \frac{\partial (G Y)_{ij} }{\partial X_{ij}} - \E \Big[\frac{\partial (G Y)_{ij} }{\partial X_{ij}}  \Big] \Big) \ea \Big] \bigg),\\
I_2:&=\frac{\sqrt{\sigma_i}}{2! N^{\frac{3}{2}}}\sum_{j=1}^N  c^{(3)}_{ij} \bigg( \E \Big[\frac{\partial^2 \ea}{\partial^2 X_{ij}} (G Y)_{ij}  \Big]+2\E \Big[\frac{\partial \ea}{\partial X_{ij}} \frac{\partial (G Y)_{ij}}{\partial X_{ij}} \Big]+\E \Big[\Big(\frac{\partial^2 (G Y)_{ij} }{\partial^2 X_{ij}} -\E \Big[\frac{\partial^2 (GY)_{ij} }{\partial^2 X_{ij}}  \Big] \Big) \ea \Big] \bigg),\\
I_3:&=\frac{\sqrt{\sigma_i}}{3! N^{2}}\sum_{j=1}^N c^{(4)}_{ij} \bigg( \E \Big[\frac{\partial^3  \ea}{\partial^3 X_{ij}} (G Y)_{ij} \Big]+3 \E \Big[\frac{\partial^2  \ea}{\partial^2 X_{ij}} \frac{\partial (G Y)_{ij}}{\partial X_{ij}} \Big]+3\E \Big[\frac{\partial  \ea}{\partial X_{ij}} \frac{\partial^2 (G Y)_{ij}}{\partial^2 X_{ij}} \Big]\\
&\qquad +\E \Big[(1-\E) \Big( \frac{\partial^3 (GY)_{ij} }{\partial^3 X_{ij}} \Big)\ea \Big]  \bigg).
\end{align*}
The last term on the right side of (\ref{sum_sample}) is estimated by (\ref{dke_sample}), (\ref{moment_condition_sample}) and Lemma \ref{dominant}. The argument is similar as in (\ref{sum}). The only things to check are the deterministic bounds of $(Y^*GY)_{ii}$ and $(GY)_{ij}$. Note that from (\ref{H_sample}) and (\ref{mathcalH}), $YG=\mathcal{G}Y$ and  $|G_{ij}| =O(N^c)$, for $z \in \Omega_0 \cap S'$, thus we have
\begin{align}
&(Y^*GY)_{ii}=(Y^* Y \mathcal{G} )_{ii}=(\mathcal{H} \mathcal{G} )_{ii}=(1+z \mathcal{G})_{ii}=O(N^{c_1});\label{hgh}\\
&|(GY)_{ij}| \leq \sqrt{N} (G Y Y^* G^*)_{ii}^{1/2}=\sqrt{N} (z(G G^*)_{ii}+G_{ii}^*)^{1/2}=\sqrt{N} \Big(\frac{z}{2 \mathrm{Im}z}(G_{ii}- G^*_{ii})+G_{ii}^*\Big)^{1/2}=O(N^{c_2}),\label{hg}
\end{align}
where we use Cauchy-Schwarz inequality and the resolvent identity (\ref{resolvent_identity}). 

Using the formulas,
\begin{equation}\label{dH_sample}
\frac{\partial G_{ab} }{\partial Y_{jk}}=-G_{aj} (Y^*G)_{kb}-(GY)_{ak} G_{jb}, \qquad \frac{\partial G_{ab} }{\partial X_{jk}}=\frac{\partial G_{ab} }{\partial Y_{jk}} \sqrt{\sigma_{j}},
\end{equation}
we obtain the analogue of Lemma \ref{lemma2}:

\begin{lemma}\label{lemma2_sample}
For any $i,j$, we have	
\begin{align}
	\frac{\partial \ea}{\partial X_{ij}}&=-\frac{\ii 2 \sqrt{\sigma_i} \lambda}{\pi}  \ea   \int_{\Omega_{0}}\pzz \tf(z) \frac{\dd}{\dd z} (GY)_{ij} \dd^2 z;\label{lemma21_sample}\\
	\frac{\partial^2 \ea}{\partial^2 X_{ij}}&=-\frac{\ii 2 \sigma_i \lambda}{\pi} \ea  \int_{\Omega_{0}}\pzz \tf(z) \frac{\dd}{\dd z} \Big(\frac{\mathfrak m}{1+\mathfrak m \sigma_i} \Big) \dd^2 z+O_{\prec} \Big( \frac{(1+|\lambda|)^2}{\sqrt{N \eta_0}}\Big).\label{lemma22_sample}
	\end{align}
	In general, for any integer $k \in \N$, we have
	\begin{equation}\label{dke_sample}
	\Big| \frac{\partial^k (GY)_{ij}}{\partial X^k_{ij}} \Big| \prec O(1); \quad \Big| \frac{\partial^k  \ea}{\partial^k X_{ij}} \Big| \prec O((1+|\lambda|)^k).
	\end{equation}
\end{lemma}

We first look at $I_1$. Using (\ref{dH_sample}) and (\ref{lemma21_sample}), we have
	\begin{align*}
	I_1=&\frac{{\sigma_i}}{N} \sum_{j=1}^N   \E \Big[\ea (1-\E) \Big(   G_{ii} \Big)  \Big] - \frac{\sigma_i}{N} \sum_{j=1}^N  \E \Big[\ea (1-\E) \Big( (Y^*GY)_{jj} G_{ii} \Big)  \Big]\\
	&- \frac{\sigma_i}{N} \sum_{j=1}^N  \E \Big[\ea (1-\E) \Big( (GY)_{ji}(GY)_{ji} \Big)  \Big]+\frac{\sqrt{\sigma_i}}{N} \sum_{j=1}^N  \E \Big[\frac{\partial  \ea}{\partial X_{ji}} (GY)_{ji} \Big] =: A_1+A_2+A_3+A_4.
	\end{align*}
	Note that
	$$A_1=\sigma_i \E[\ea (1-\E) G_{ii} ]=\sigma_i \E[\ea (1-\E) G_{ii} ].$$
 In addition, using the definition of the resolvent and the local law in Theorem \ref{fluctuation}, we have
 $$A_2=-\frac{\sigma_i}{N}\E [\ea (1-\E)  \big( \Tr  (G H) G_{ii} \big)]=-\gamma \sigma_i \E [\ea (1-\E)  \big( z M^{-1}\Tr G G_{ii} \big)]-\gamma \sigma_i \E [\ea (1-\E)  \big(G_{ii} \big)]$$
 $$=(-\gamma \sigma_i  z m_{fc}-\gamma \sigma_i) \E [\ea (1-\E)  \big( G_{ii} \big)] +\frac{\gamma \sigma_i}{1+\mathfrak{m} \sigma_i}\E [\ea (1-\E)  \big( M^{-1} \Tr G  \big)]+O_{\prec}\big( \Psi(z) \Theta(z) \big).$$
         Next, we use the local law to estimate the third term,
         $$A_3=- \frac{1}{N}  \E \Big[\ea (1-\E) ( Y^* G^2 Y)_{ii}  \Big]=- \frac{1}{N}  \E \Big[\ea (1-\E) \frac{\dd}{\dd z} ( Y^*GY)_{ii}  \Big]=O_{\prec}\Big( \frac{\Psi(z)}{ N \eta} \Big).$$
	Finally, we study the last term $A_4$ using (\ref{lemma21_sample}), which can be written as
\begin{align*}
A_4&=-\frac{\i 2 \sigma_i \lambda}{\pi N} \E\Big[  \ea   \int_{\Omega_{0}}\frac{\partial}{\partial \overline{z_2}} \tf(z_2) \frac{\partial}{\partial z_2} ( G(z_2) H G(z_1) )_{ii}\dd^2 z_2  \Big]\\
&=-\frac{\i 2 \sigma_i \lambda}{\pi N} \E\Big[  \ea   \int_{\Omega_{0}}\frac{\partial}{\partial \overline{z_2}} \tf(z_2) \frac{\partial}{\partial z_2} (z_2 G(z_2) G(z_1))_{ii}\dd^2 z_2  \Big].
\end{align*}
Using the resolvent identity (\ref{resolvent_identity}), the local law Theorem \ref{fluctuation} and Lemma \ref{lemma3},
we obtain that
	\begin{equation}\label{a4_sample}
	A_4 = -\frac{\i 2 \lambda}{\pi N} \E\Big[  \ea   \int_{\Omega_{0}}\frac{\partial}{\partial \overline{z_2}} \tf(z_2) \frac{\partial}{\partial z_2}  \Big( \frac{ \sigma_i z_2 (\g_i(z_1)-\g_i(z_2))}{z_1-z_2} \Big)\dd^2 z_2  \Big]+e(i),
	\end{equation}	
	where $\g_i(z):=-\frac{1}{z(1+\mathfrak{m} \sigma_i)}$ for simplicity, and $e(i)$ is the error term.
	If we consider the linear statistics of the error term $e(i)$, using the same argument as for deformed Wigner matrix in Section \ref{subsection:sum}, we have
	$$\big| \sum_{i=1}^N \g_i(z) e_i(z) \big| = O_{\prec} \Big( \frac{1}{N \eta^2_1}\Big)+O_{\prec} \Big( \frac{1}{N \eta_0\eta_1}\Big).$$

As for the second cumulant expansion term $I_2$, we apply the same argument as in (\ref{estimate}) and the anisotropic law Theorem \ref{fluctuation} to find that
 $$I_2 =O_{\prec}\Big( \frac{(1+|\lambda|^2) \Psi(z)}{N \sqrt{ \eta_0}} \Big)+O_{\prec} \Big(  \frac{\Psi^2(z)}{\sqrt{N}}\Big).$$

We compute the third cumulant expansion term $I_3$ similarly using (\ref{dH_sample}), the local law Theorem \ref{fluctuation} and (\ref{dke_sample}). The leading term comes from the second term of $I_3$, denoted by $D_2$, i.e.,
$$I_3=\frac{\sigma_i}{2 N^{2}}\sum_{j=1}^N \E \Big[\frac{\partial^2  \ea}{\partial^2 X_{ij}} c^{(4)}_{ij}  \Big( G_{ii}-(Y^*GY)_{jj} G_{ii}-(GY)_{ij} (GY)_{ij}  \Big) \Big]+O_{\prec} ((1+|\lambda|^3) N^{-1}\Psi(z)).$$
Using (\ref{hgh}) and (\ref{lemma22_sample}), $D_2$ can be written as
\begin{align}\label{d2_sample}
D_2=&\frac{\sigma_i}{2 N^{2}}\sum_{j=1}^N \E \Big[\frac{\partial^2  \ea}{\partial^2 X_{ij}} c^{(4)}_{ij} \frac{\mathfrak{m}}{1+\mathfrak m \sigma_i}  \Big]+O_{\prec} \Big( (1+|\lambda|^2)N^{-1} \Psi(z) \Big)\nonumber \\ 
=&-\frac{ \i  K_4 \sigma_i^2 \lambda}{\pi N} \E \Big[  \ea \int_{\Omega_{0}} \pzzp 
\tf(z') \pzp  \Big( \frac{\mathfrak{m}(z)}{1+\mathfrak m(z) \sigma_i}  \frac{\mathfrak{m}(z')}{1+\mathfrak m(z') \sigma_i}   \Big) \dd^2 z' \Big] \\
&+O_{\prec} \Big( \frac{1+|\lambda|^2}{ N \sqrt{N \eta_0}} \Big)+O_{\prec} \Big( (1+|\lambda|^2) N^{-1}\Psi(z) \Big),\nonumber
\end{align}
where $K_4$ is given in (\ref{k_4}) and is independent of the index $i$.

	Summing up and rearranging terms in (\ref{sum_sample}), the coefficient in front of $\E [\ea (1-\E) G_{ii}]$ is given by
	$$(z-\sigma_i+\gamma \sigma_i +\gamma \sigma_i z m_{fc} )=z(1+ \mathfrak{m}(z) \sigma_i),$$
	which is away from zero for $z \in S'$. Dividing both sides of (\ref{sum_sample}) by this coefficient, we have
	\begin{equation*}
	\E [\ea ( G_{ii} -\E  G_{ii})]=\frac{\gamma \sigma_i}{M z(1+\mathfrak{m}(z) \sigma_i)^2}\E [\ea (1-\E)  \big(  \Tr G  \big)]+\frac{ A_4(i)}{z(1+ \mathfrak{m}(z) \sigma_i)}+\frac{D_2(i)}{z(1+ \mathfrak{m}(z) \sigma_i)} +\mathcal E(i),
	\end{equation*}
	where $A_4$ and $D_2$ are given in (\ref{a4_sample}) and (\ref{d2_sample}), and the error terms $\mathcal E(i)$ is analytic in $\Omega_0$ and estimated as before. Summing over $i$ and rearranging,  we have
	\begin{equation}\label{newsum_sample}
	\Big( 1-\frac{1}{M} \sum_{i=1}^M \frac{\gamma \sigma_i}{z(1+\mathfrak{m}(z) \sigma_i)^2}  \Big) \E [\ea (1-\E ) \Tr G]=\sum_{i=1}^M \frac{ A_4(i)}{z(1+ \mathfrak{m}(z) \sigma_i)}+\sum_{i=1}^M \frac{D_2(i)}{z(1+ \mathfrak{m}(z) \sigma_i)} +\mathcal E,
	\end{equation}
	where the error term $\mathcal{E}$ has an upper bound as in (\ref{erroruselater}).	
	Dividing by the coefficient $1-\frac{1}{M} \sum_{i=1}^M \frac{\gamma \sigma_i}{z(1+\mathfrak{m}(z) \sigma_i)^2}$ and recalling (\ref{coefficient}), the first two terms on the right side of (\ref{newsum_sample}) are denoted as $A$ and $D$ respectively, and the error term is bounded by (\ref{error_sample}). Using (\ref{a4_sample}) and (\ref{coefficient}), we have
	\begin{align*}
	A&=-\frac{z_1 \mr'_1}{\mr_1} \frac{\i 2 \lambda \gamma}{\pi}  \E\Big[  \ea   \int_{\Omega_{0}}\frac{\partial}{\partial \overline{z_2}} \tf(z_2) \frac{\partial}{\partial z_2}  \Big( \frac{z_2}{z_1-z_2} \frac{1}{M}\sum_{i=1}^M \sigma_i \g_i(z_1)(\g_i(z_1)-\g_i(z_2))  \Big)\dd^2 z_2  \Big]\\
	&=\frac{\i 2 \lambda}{\pi}  \E\Big[  \ea   \int_{\Omega_{0}}\frac{\partial}{\partial \overline{z_2}} \tf(z_2)  \Big( \frac{\mr_1' \mr_2'}{(\mr_1-\mr_2)^2}-\frac{1}{(z_1-z_2)^2} \Big)\dd^2 z_2  \Big],
	\end{align*}
	where we set $\mr_1:=\mr(z_1)$ and $\mr_2:=\mr(z_2)$ for short.
	This follows from
	\begin{align*}
	z_1z_2 \frac{1}{M}\sum_{i=1}^M \sigma_i  \g_i(z_1)  \g_i(z_2) &=\frac{1}{M}\sum_{i=1}^N \frac{\sigma_i}{(1+\mr_1 \sigma_i) (1+\mr_2 \sigma_i)}=\frac{z_1\mr_1-z_2 \mr_2}{\gamma(\mr_1-\mr_2)};\\
	z_1^2 \frac{1}{M}\sum_{i=1}^M \sigma_i  \g^2_i(z_1)  &=\frac{1}{M}\sum_{i=1}^N \frac{\sigma_i}{(1+\mr_1 \sigma_i)^2}=\frac{(z_1 \mr_1 )'}{\gamma \mr_1'}.
	\end{align*}
	Similarly, recalling (\ref{d2}), we obtain that
	\begin{align*}
	D&=-\frac{z_1 \mr'_1}{\mr_1} \sum_{i=1}^N \frac{ D_2(i)}{z(1+ \mathfrak{m}(z) \sigma_i)}=\frac{ \i  K_4 \gamma  \lambda}{\pi } \E \Big[  \ea \int_{\Omega_{0}} \pzbb \tf(z_2)  \frac{1}{M} \sum_{i=1}^M \Big( \sigma_i^2 \frac{\mathfrak{m}'_1}{(1+\mathfrak m_1 \sigma_i)^2}  \pzb   \frac{\mathfrak{m}_2}{1+\mathfrak m_2 \sigma_i}  \Big) \dd^2 z_2 \Big]\\
	&=\frac{ \i  K_4 \gamma  \lambda}{\pi } \E \Big[  \ea \int_{\Omega_{0}} \pzbb \tf(z_2)  \Big( \pzab \Big[ \frac{1}{M} \sum_{i=1}^M \frac{1}{(1+\mathfrak m_1 \sigma_i)(1+\mathfrak m_2 \sigma_i)}    \Big] \Big) \dd^2 z_2 \Big].
	\end{align*}
Therefore, we obtain an analogue of Lemma \ref{lemma4}, where the kernel is given  instead by (\ref{kernel_sample}).

Next, we compute the explicit formula for the variance $V(f)$ given by (\ref{vf}) with $K$ in (\ref{kernel_sample}) and test function $f$ in (\ref{fn}). Both in the bulk and at the edge, the second term of (\ref{kernel_sample}) will only contribute $O(\eta_0N^{\tau})$, because of (\ref{prime_gamma}), (\ref{1+tm}) and (\ref{assumpf}). It is sufficient to look at the first term. In the bulk, the main contribution of $V(f)$ comes from the term $-\frac{2}{(z_1-z_2)^2}$ with $z_1$, $z_2$ in different half planes. Hence by similar arguments as in Lemma \ref{vfbulk}, we obtain that
$$\lim _{N \rightarrow \infty} V(f)=\frac{1}{2  \pi^2} \int_{\R} \int_{\R} \frac{(g(x_1)-g(x_2))^2}{(x_1-x_2)^2} \dd x_1 \dd x_2.$$
At the edge, we recall the expansion of $\mr(z)$ near $z=E_+$ from (\ref{expansion_sample}), 
$$\mr(z)=\mr(E_+)+c_+ \sqrt{z-E_+}(1+A_+(\sqrt{z-E_+}))=\xi_++c_+\sqrt{z-E_+}+O(|z-E_+|),$$
where the square root is taken in a branch cut such that $\Im \mr>0$ when $\Im z>0$, and $c_+>c_0>0$ for sufficiently large $N$. Differentiating it, we have
$$\mr'(z)=\frac{c_+}{2\sqrt{z-E_+}}+d_++O(\sqrt{|z-E_+|}).$$
Therefore, repeating the arguments in the proof of Lemma \ref{bfedge}, we get the variance at the edge
$$\lim _{N \rightarrow \infty} V(f)=\frac{1}{4 \pi^2} \int_{\R}\int_{\R} \Big(\frac{g(-x_1^2)-g(-x_2^2)}{x_1-x_2} \Big)^2 \dd x_1 \dd x_2.$$

\subsection{Expectation and bias computation}

Starting from (\ref{uselater_sample}), using the cumulant expansion and (\ref{dH_sample}), we obtain that
	\begin{align}
	z \E {G}_{ii}=&\sqrt{\sigma_i} \sum_{j=1}^N  \E X_{ij} (G Y)_{ij} -1=\frac{\sqrt{\sigma_i}}{N} \E \sum_{j=1}^N c^{(2)}_{ij}  \frac{\partial (GY)_{ij} }{\partial X_{ij}}-1 \nonumber\\
	&+\frac{\sqrt{\sigma_i}}{2! N^{\frac{3}{2}}}\sum_{j=1}^N c^{(3)}_{ij} \E \frac{\partial^2 (GY)_{ij} }{\partial^2 X_{ij}} +\frac{\sqrt{\sigma_i}}{3! N^{2}}\sum_{j=1}^N c^{(4)}_{ij} \E\frac{\partial^3 (GY)_{ij} }{\partial^3 X_{ij}} +O_{\prec}(N^{-\frac{3}{2}})\nonumber\\
=&\frac{\sigma_i}{N} \sum_{j=1}^N \E G_{ii}-\frac{\sigma_i}{N} \sum_{j=1}^N \E (Y^*GY)_{jj} G_{ii}-\frac{\sigma_i}{N} \sum_{j=1}^N \E ((GY)_{ij})^2-1\label{first_line}\\
&+\frac{\sigma_i^{3/2}}{2 N^{\frac{3}{2}}}\sum_{j=1}^N c^{(3)}_{ij} \E \Big(-6 G_{ii} (GY)_{ij}+6 G_{ii}(GY)_{ij}(Y^* G Y)_{jj}+2 ((GY)_{ij})^3\Big)\label{second_line}\\
&+\frac{\sigma_i^2}{6 N^{2}}\sum_{j=1}^N c^{(4)}_{ij} \E \Big(-6(G_{ii})^2+12 (G_{ii})^2 (Y^* G Y)_{jj}-6 (G_{ii})^2((Y^*GY)_{jj})^2 \Big)+O_{\prec}(N^{-\frac{3}{2}}).\label{last_line}
\end{align}
The first line (\ref{first_line}) can be written as
\begin{align*}
&\sigma_i \E G_{ii}-\frac{\sigma_i}{N} \E [\Tr (Y^*GY) G_{ii}]-\frac{\sigma_i}{N} \E  (GH G)_{ii}-1\\
=&\sigma_i (1-\gamma)\E G_{ii}-z{\sigma_i} \gamma \E [M^{-1}\Tr G G_{ii}]+\frac{\sigma_i}{N} \Big( \frac{1}{1+\sigma_i \mr} \Big)'-1+O_{\prec}(N^{-3/2})\\
=&\sigma_i (1-\gamma) \E G_{ii}-z{\sigma_i} \gamma m_{fc} \E \Big( G_{ii}+\frac{1}{z(1+ \mr \sigma_i)} \Big)+\frac{\sigma_i\gamma}{M(1+\mr \sigma_i)}  \E \Tr G-\frac{\sigma^2_i}{N}  \frac{\mr'}{(1+\sigma_i \mr)^2}-1+O_{\prec}\Big((N \eta)^{-\frac{3}{2}}\Big).
\end{align*}
Using the anisotropic local law and the same arguments as in (\ref{estimate}), one shows that the second line (\ref{second_line}) is $O_{\prec}( N^{-1}\Psi(z)).$ The local law implies that the last line (\ref{last_line}) becomes
$$\frac{\sigma_i^2}{ N^{2}}\sum_{j=1}^N c^{(4)}_{ij} \E \Big(-(G_{ii})^2+2 (G_{ii})^2 (Y^* G Y)_{jj}- (G_{ii})^2((Y^*GY)_{jj})^2 \Big)+O_{\prec}(N^{-\frac{3}{2}})$$
$$=-\frac{\sigma_i^2}{ N^{2}}\sum_{j=1}^N c^{(4)}_{ij}  \Big(\frac{\mr^2}{(1+\mr \sigma_i)^2} \Big)+O_{\prec}(N^{-\frac{3}{2}})=-K_4\frac{\sigma_i^2 }{ N}  \Big(\frac{\mr^2}{(1+\mr \sigma_i)^2} \Big)+O_{\prec}(N^{-\frac{3}{2}}).$$
Therefore, we have
\begin{align*}
z(1+ \mathfrak{m} \sigma_i) \Big( \E G_{ii}+\frac{1}{z(1+\mr \sigma_i)} \Big)=\frac{\sigma_i\gamma}{1+\mr \sigma_i}  \E [M^{-1}\Tr G-m_{fc}]-\frac{1}{M}  \frac{ \gamma \sigma^2_i \mr'}{(1+\sigma_i \mr)^2} \\
-K_4\frac{1 }{ M}  \frac{\gamma \sigma_i^2 \mr^2}{(1+\mr \sigma_i)^2}+O_{\prec}\Big((N \eta)^{-\frac{3}{2}}\Big).
\end{align*}
Dividing both sides by $z(1+ \mathfrak{m} \sigma_i) \sim O(1)$ from (\ref{1+tm}) and summing over $i$, we obtain
$$\Big(1-\frac{1}{M} \sum_{i=1}^M \frac{\gamma \sigma_i}{z(1+\mathfrak{m}(z) \sigma_i)^2} \Big)\E (\Tr G -M m_{fc})=  -  \frac{1}{M} \sum_{i=1}^{M} \frac{\gamma \sigma_i^2 \mr' }{z(1+\mr \sigma_i)^3}-K_4  \frac{1}{M}\sum_{i=1}^M \frac{\gamma \sigma_i^2 \mr^2}{z(1+\mr \sigma_i)^3} +O_{\prec}\Big(\frac{1}{\sqrt{N \eta^3}}\Big).$$
Dividing both sides by the coefficient of $\E (\Tr G -M m_{fc})$ and using (\ref{coefficient}), the error becomes $O_{\prec}\Big( \frac{1}{\eta\sqrt{N \eta \sqrt{\kappa+\eta}}}\Big)$ and the leading term on the right side becomes (\ref{bz_sample}).
Therefore, we obtained an analogue of Proposition~\ref{prop2} with the integral kernel $b(z)$ given instead by (\ref{bz_sample}). 

Finally, we compute the explicit formula for the bias. Using (\ref{prime_gamma}), (\ref{1+tm}) and (\ref{assumpf}), the second term of $b(z)$ will contribute $O(\sqrt{\eta_0 N^{\tau}})$ both in the bulk and at the edge. In addition, (\ref{mplawgamma-1}) implies that the first term of $b(z)$ can be written as
$$\frac{(\mr')^2}{\mr}\frac{1}{M} \sum_{i=1}^M \frac{\gamma \sigma_i^2 }{(1+\sigma_i \mr)^3}=\frac{ \mr''}{2 \mr'}-\frac{\mr'}{\mr}.$$
The second term on the right side contributes $O(\sqrt{\eta_0 N^{\tau}})$. The first term vanishes if $\kappa_0>c>0$ and thus the bias in the bulk vanishes. At the edge, we use the expansion of $\mr(z)$ around $E_+$ from (\ref{expansion_sample}), 
$$\mr'(z)=\frac{c_+}{2\sqrt{z-E_+}}+d_++O(\sqrt{|z-E_+|}), \qquad \mr''(z)=-\frac{c_+}{4(\sqrt{z-E_+})^3}+O\Big(\frac{1}{\sqrt{|z-E_+|}}\Big).$$
Hence we have
$$\frac{ \mr''}{2 \mr'}=- \frac{1}{4(z-E_+)}+O\Big(\frac{1}{\sqrt{|z-E_+|}}\Big),$$
and using the Sokhotski-Plemelj lemma, the bias at the edge becomes $\frac{g(0)}{4}$.

\appendix

\section{Complex case}\label{sec:complex}

In this appendix, we extend previous results from real symmetric to complex Hermitian matrices. We will use the complex analogue of Lemma \ref{cumulant}. 
\begin{lemma}(Complex cumulant expansion)
	Let $h$ be a complex-valued random variable with finite moments, and $f$ is a complex-valued smooth function on $\R$ with bounded derivatives. Let $c_{p,q}$ be the $(p,q)$ cumulant of $h$, which is defined as 
	$$c_{p,q}:=(-\ii)^{p+q} \Big( \frac{\partial^{p+q}}{\partial s^p \partial t^q} \log \E e^{\i s h+\i t \overline{h}} \Big) \Big|_{s,t=0}.$$ 
	Then for any fixed $l \in \N$, we have
	$$\E[ \bar{h} f(h, \bar{h})]=\sum_{p+q=0}^l \frac{1}{p!q!} c_{p,q+1}(h)\E[ f^{(p,q)}(h)] +R_{l+1},$$
	where the error term satisfies
	$$|R_{l+1}| \leq C_l \E |h|^{l+2} \max_{p+q=l+1} \Big\{ \sup_{|x| \leq M} |f^{(p,q)}(z,\bar{z})| \Big\}+C_l \E \Big[ |h|^{l+2} 1_{|h|>M}\Big] \max_{p+q=l+1} \|f^{(p,q)}(z,\bar{z})\|_{\infty},$$
	and $M>0$ is an arbitrary fixed cutoff.
\end{lemma}
Instead of (\ref{dH}), we have
\begin{equation}\label{partial}
\frac{\partial G_{ij}}{\partial H_{ab}}=-G_{ia}G_{bj},
\end{equation}
from which we obtain the analogue of Lemma \ref{lemma2}.

The assumption $\E H_{ij}^2=0$ implies that $c_{ij}^{(1,1)}=1$, $c_{ij}^{(2,2)}=W_4-2$ for $i \neq j$.  Using the anisotropic law and (\ref{partial}), one shows similarly that the expansion terms corresponding to $p+q=3$ are negligible. Using (\ref{partial}) and the analogue of Lemma \ref{lemma2}, we obtain that
         \begin{align*}
         (z-a_i) \E [\ea(G_{ii}-\E G_{ii})]=&\frac{1}{N}\sum_{j=1}^N c^{(1,1)}_{ij} \E \Big[\frac{\partial }{\partial H_{ji}} \Big( \ea (G_{ji}-\E G_{ji}) \Big) \Big] \\
&+\frac{1}{2! N^{2}}\sum_{j=1}^N c^{(2,2)}_{ij} \E \Big[\frac{\partial^3 }{\partial^2 H_{ji}\partial H_{ij}} \Big( \ea (G_{ji}-\E G_{ji}) \Big) \Big] +\cdots\\
=&\frac{1}{N} \E \Big[\ea \Big( \frac{\partial G_{ji}}{\partial H_{ji}}- \E \frac{\partial G_{ji}}{\partial H_{ji}} \Big) \Big]+\frac{1}{N} \sum_{j=1}^N  \E\Big[ \frac{\partial \ea}{ \partial H_{ji}} G_{ji}\Big]+\frac{m_2-1}{N} \E\Big[ \frac{\partial \ea}{ \partial H_{ii}} G_{ii}\Big]\\
&+\frac{1}{N^2} \sum_{j=1}^N (W_4-2) \E \Big[ \frac{\partial \ea}{\partial H_{ji} \partial H_{ij}} \frac{\partial G_{ji}}{\partial H_{ji}} \Big]+\cdots.
         \end{align*}
Thus Proposition \ref{prop} holds with modified variance, i.e. $m_2-2$ be replaced by $m_2-1$, $W_4-3$ be replaced by $W_4-2$, and the coefficient of the remaining term be 1 instead of 2. 
Similarly, as for the expectation, 
$$(z-a_i) \E G_{ii} =\E (HG)_{ii}-1=\frac{1}{N}\sum_{j=1}^N c^{(1,1)}_{ij} \E \frac{\partial G_{ji} }{\partial H_{ji}}-1 +\frac{1}{2! N^{2}}\sum_{j=1}^N c^{(2,2)}_{ij} \E  \frac{\partial^3 G_{ji} }{\partial^2 H_{ji}\partial H_{ij}} +\cdots.$$
Thus the first term of $b(z)$ given in (\ref{bz}) vanishes, $m_2-2$ is replaced by $m_2-1$ and $W_4-3$ is replaced by $W_4-2$.

\section{Proofs of Auxiliary Lemmas}\label{appendix B}

\begin{proof}[Proof of Lemma \ref{lemma3}]
         For $1 \leq s \leq 2$, the proof is given in Lemma 4.4 in \cite{character}. Since $h(z)$ is holomorphic on $\Omega_0$, $\pzz \tf(z) h(z)=\pzz(\tf(z) h(z))$. Using Stokes' formula, we have
	$$\int_{\Omega_{0}} \pzz \tf(z) h(z) \dd^2z =-\frac{\i}{2}\int_{\partial \Omega_{0}}  \tf(z) h(z) \dd z.$$
	Since $g$ is compactly support, $\tf(z)=0$ on $\partial \Omega_0$ except $$\Gamma_0:=\Big\{ x+\i y: x \in \mbox{supp}(f),  |y|=N^{-\tau} \eta_0 \Big\}.$$ Using (\ref{assumpf}) we have
	$$\Big| \int_{\Omega_{0}} \pzz \tf(z) h(z) \dd^2z \Big| \leq C K \int_{\Gamma_0} \Big(  |y |^{-s} |f(x)|+ |y|^{1-s} |f'(x)|  \Big) \dd z \leq C' K N^{\tau s} \eta^{1-s}_{0} .$$	
\end{proof}

\begin{proof}[Proof of Lemma \ref{Izz}]
	Using the self-consistent equation of $m_{fc}$ in (\ref{self}), we have
	         \begin{align*}
	         m_{fc}(z_1)-m_{fc}(z_2)&=\frac{1}{N} \sum_{i=1}^N \Big( \frac{1}{a_i-z_1-m_{fc}(z_1)}-\frac{1}{a_i-z_2-m_{fc}(z_2)} \Big)\\
	&=\frac{1}{N} \sum_{i=1}^N \Big( \frac{z_1+m_{fc}(z_1)-z_2-m_{fc}(z_2)}{(a_i-z_1-m_{fc}(z_1))(a_i-z_2-m_{fc}(z_2))} \Big).
	         \end{align*}
	Dividing $z_1+m_{fc}(z_1)-z_2-m_{fc}(z_2)$ on both sides and we get the first identity.
	Taking the derivative of (\ref{self}), we have
	\begin{equation}\label{self3}
	\frac{1}{N} \sum_{i=1}^N \frac{1+m'_{fc}(z)}{(a_i-z-m_{fc})^2}=m'_{fc}(z).
	\end{equation}
	 We treat $\tilde{I}$ and $\tilde{I}_s$ similarly. Thus we complete the proof.
\end{proof}

\begin{proof}[Proof of Lemma \ref{m}]
Note that
	$$|I_s(z)| \leq \frac{1}{N} \sum_{i=1}^N \frac{1}{|a_i-z-m_{fc}(z)|^2} =\frac{{\Im}m_{fc}(z)}{{\Im}m_{fc}(z)+\eta} < 1.$$
	
	By (\ref{self3}), we have $\frac{m'_{fc}}{1+m'_{fc}}=I_s(z)$ and thus $m'_{fc}(z)=\frac{I_s(z)}{1-I_s(z)}.$ Using  Lemma \ref{previous}, we have
	\begin{equation}\label{aaa}
	|m'_{fc}(z)| \leq \frac{1}{|1-I_s(z)|} \sim \frac{1}{\sqrt{\kappa+\eta}}.
	\end{equation}
	Differentiating (\ref{self3}) again, we obtain that
	$$\frac{m''_{fc}}{(1+m'_{fc})^3}=\frac{2}{N} \sum_{i=1}^N \frac{1}{(a_i-z-m_{fc})^3}.$$
	Combining (\ref{2}) and (\ref{aaa}),  we get the upper bound of $m''_{fc}$. The rest inequalities follow directly from Lemma \ref{previous}
\end{proof}

\begin{proof}[Proof of Lemma \ref{lemma2}]
	Using (\ref{dH}), we have
	$$\frac{\partial \ea }{\partial H_{ij}}= \frac{\i \lambda}{\pi} \ea   \int_{\Omega_{0}}\pzz \tf(z) \Big(\sum_{l=1}^N \frac{\partial G_{ll}}{\partial H_{ij}} \Big) \dd^2 z=-\frac{\i (2-\delta_{ij})\lambda}{\pi}  \ea   \int_{\Omega_{0}}\pzz \tf(z) (G^2)_{ji} \dd^2 z.$$
	Note that $(G^2)_{ji}=\frac{d}{dz}G_{ji}(z)$. Since $G_{ij}$ is analytic in $D'$, using the Cauchy integral formula and the local law, we have that for $i \neq j$,  $(G^2)_{ji} \prec \frac{\Psi(z)}{{\Im}z}$. Combining with Lemma~\ref{lemma3}, we obtain that, for $i \neq j$,
	$$\Big|\frac{\partial \ea }{\partial H_{ij}}\Big|=O_{\prec} \Big( \frac{1+|\lambda|}{\sqrt{N \eta_0}}\Big).$$
	Similarly, if $i=j$, we have 
	$$\frac{\partial \ea}{\partial H_{ii}}  = -\frac{\i \lambda}{\pi}  \ea  \int_{\Omega_{0}}\pzz \tf(z) \pz \frac{1}{a_i-z-m_{fc}(z)} \dd^2 z+O_{\prec} \Big(\frac{1+|\lambda|}{\sqrt{N \eta_0}}\Big).$$
	Furthermore, we compute that
	         \begin{align*}
	         \frac{\partial^2 \ea}{\partial^2 H_{ij}}=&-\frac{\lambda^2 (2-\delta_{ij})^2}{\pi^2} \ea \Big( \int_{\Omega_{0}}\pzz \tf(z) (G^2)_{ji} \dd^2 z  \Big)^2\\
	&+\frac{\i (2-\delta_{ij}) \lambda}{\pi} \ea  \int_{\Omega_{0}}\pzz \tf(z) \Big(2 (G^2)_{ji} G_{ij}+ (1-\delta_{ij})(G^2)_{ii} G_{jj}+(1-\delta_{ij})(G^2)_{jj} G_{ii} \Big) \dd^2 z.
	       \end{align*}
	For $i \neq j$, combining the local law and Lemma \ref{lemma3}, we have
	\begin{align*}
	\frac{\partial^2 \ea}{\partial^2 H_{ij}}&=\frac{\i 2 \lambda}{\pi} \ea  \int_{\Omega_{0}}\pzz \tf(z) \frac{\dd}{\dd z} \Big(G_{ji} G_{ij}+ G_{ii} G_{jj} \Big) \dd^2 z+O_{\prec} \Big( \frac{(1+|\lambda|)^2}{{N \eta_0}}\Big)\\
	&=\frac{\i 2 \lambda}{\pi} \ea  \int_{\Omega_{0}}\pzz \tf(z) \pz  \frac{1}{(a_i-z-m_{fc})(a_j-z-m_{fc})}  \dd^2 z+O_{\prec} \Big( \frac{(1+|\lambda|)^2}{\sqrt{N \eta_0}}\Big).
	\end{align*}
	Similarly, for $i=j$, we have
	$$\frac{\partial^2 \ea}{\partial^2 H_{ii}}=\frac{\i \lambda}{\pi} \ea  \int_{\Omega_{0}}\pzz \tf(z) \pz  \frac{1}{(a_i-z-m_{fc})^2}  \dd^2 z+O_{\prec} \Big( \frac{(1+|\lambda|)^2}{\sqrt{N \eta_0}}\Big).$$
	In general, using the local law, (\ref{dH}) and Lemma \ref{lemma3}, we complete the proof of (\ref{dke}).	
\end{proof}

\begin{proof}[Proof of Lemma \ref{sample_m}]
We start by proving the first two statements, using which we will show (\ref{coefficient}). The last statement then follows directly from (\ref{coefficient}). Note that the first equation in~(\ref{mplawgamma-11}) implies the first inequality in (\ref{1+tm}). To prove the rest, we divide the spectral domain $S$ into three regimes, corresponding to the bulk, the edge and the outside.
First, we consider $z$ near the edge $E_+$, i.e., $z \in S^{e}:=\{ z=E+\ii \eta \in S: E \in [E_+-\tau', E_+ +\tau']\}$ for some small $\tau'>0$. Let $\xi=-\mr(z)$ so that (\ref{mplawgamma-1}) is equivalent to 
$$z=F(\xi), \qquad F(\xi):=\frac{1}{\xi} + \gamma \int_{\R} \frac{t }{1-t \xi} \dd \mu_\Sigma(t).$$ 
Due to (\ref{gap_sample}), $F(\xi)$ is analytic around $\xi_+$, and we have
$$F(\xi)=F(\xi_+)+F'(\xi_+)(\xi-\xi_+)+\frac{1}{2!}F''(\xi_+)(\xi-\xi_+)^2+O(|\xi-\xi_+|^3),$$
where the linear term vanishes because of (\ref{solutionedge_sample}). It also yields
$$ F''(\xi_+)=\frac{2}{\xi_+^3}+ 2 \gamma \int \frac{t^3}{(1-t \xi_+)^3} \dd \mu_\Sigma(t)=2 \gamma \int \frac{t^2}{\xi_+(1-t \xi_+)^3} \dd \mu_\Sigma(t) \geq c>0.$$
The last step follows from the fact that $\xi_+ \geq c>0$ and $1-t\xi_+ \geq c>0$ for sufficiently large $N$. Thus we have the expansion of $\mr$ near the edge $z=E_+$, 
\begin{equation}\label{expansion_sample}
\mr(z)=\mr(E_+)+c_+ \sqrt{z-E_+}(1+A_+(\sqrt{z-E_+}))=\xi_++c_+\sqrt{z-E_+}+O(|z-E_+|),
\end{equation}
where $c_+>c_0>0$ for sufficiently large $N$, $A_+$ is an analytic function with $A_+(0)=0$, and the square root is taken with the branch cut such that $\Im \mr(z)>0$ when $\Im z>0$. Hence the corresponding density has the square root behavior at the right edge. The left edge can be treated similarly when $\gamma_0 \neq 1$. Using the definition of Stieltjes transform, one shows (\ref{im_m_fc}). Similar arguments can be found in Lemma A.5~\cite{locallaw}. 

 If $E$ is inside the bulk, i.e., $z \in S^{b}:=\{ z \in S: E \in [E_-+\tau', E_+ -\tau']\}$, then $\Im \mr \geq c>0$. Thus (\ref{1+tm}) and (\ref{im_m_fc}) follows.  Finally, for the outside spectral domain, if  $z \in S^{o}:=\{ z \in S: \mbox{dist}(E, [E_-, E_+]) \geq \tau'\}$, it follows from $\Im \mr \sim \eta$ and (\ref{gap_sample}). Similar arguments can be found in Appendix A in \cite{isotropic}. 

Next, we will prove (\ref{coefficient}). We first prove the upper bound. Taking the real and imaginary part of (\ref{mplawgamma-1}), we have
         \begin{equation}\label{useful_sample}
          E \Im \mr+ \eta \Re \mr=\frac{\gamma}{M} \sum_{i=1}^M \frac{ \sigma_i \Im \mathfrak{m}}{|1+\mathfrak{m} \sigma_i|^2}; \qquad E \Re \mr- \eta \Im \mr=-\frac{\gamma}{M} \sum_{i=1}^M \frac{ 1+\sigma_i \Re \mathfrak{m}}{|1+\mathfrak{m} \sigma_i|^2}+\gamma-1,
          \end{equation}
        with $\mr \equiv \mr(z)$. Then we have
         $$\Big| z-\frac{\gamma}{M} \sum_{i=1}^M \frac{ \sigma_i}{|1+\mathfrak{m}(z) \sigma_i|^2} \Big|=\Big| {\i \eta-\frac{\Re \mr}{\Im \mr} \eta} \Big| =\frac{|\mr| \eta}{\Im \mr}.$$
         Using $|\mr(z)| \sim 1$ and (\ref{im_m_fc}). we obtain an upper bound of the right side as $C\sqrt{\kappa+\eta}$.
        In addition,
        \begin{align*}
        &\qquad \Big| \frac{1}{M} \sum_{i=1}^M \frac{ \sigma_i}{|1+\mathfrak{m}(z) \sigma_i|^2}-\frac{1}{M} \sum_{i=1}^M \frac{ \sigma_i}{(1+\mathfrak{m}(z) \sigma_i)^2} \Big|\\
        &=2 \Big|\frac{1}{M} \sum_{i=1}^M \frac{ \sigma_i (\Im (1+\mathfrak{m}(z) \sigma_i))^2+\ii \sigma_i \Re (1+\mathfrak{m}(z) \sigma_i) \Im (1+\mathfrak{m}(z) \sigma_i)}{|1+\mathfrak{m}(z) \sigma_i|^4} \Big|  \leq  C \sqrt{\kappa+\eta},
        \end{align*}
        hence we obtain an upper bound for the left side of (\ref{coefficient}). Next, it is sufficient to show the lower bound. 
         If $z \in S^{e}$, (\ref{gap_sample}) implies that $\Re(1+\sigma_i \mr)>c$ if we choose $\tau'$ sufficiently small. We split into two cases. If $E \in [E_-, E_+]$, we have
         $$\Big|z-\frac{1}{M} \sum_{i=1}^M \frac{\gamma \sigma_i}{(1+\mathfrak{m}(z) \sigma_i)^2} \Big| \geq  \Big| \Im \Big( z- \frac{1}{M} \sum_{i=1}^M \frac{\gamma \sigma_i}{(1+\mathfrak{m}(z) \sigma_i)^2} \Big) \Big| \geq   \Big| \eta+ \frac{2\gamma}{M} \sum_{i=1}^M \frac{ \sigma^2_i \Im \mr \Re  (1+\sigma_i \mr)}{|(1+\mathfrak{m}(z) \sigma_i)|^4} \Big| $$
         \begin{equation}\label{repeat}
         \geq  \frac{2\gamma}{M} \sum_{i=1}^M \frac{\sigma_i  \Im  \mr\Re  (1+\sigma_i \mr)}{|(1+\mathfrak{m}(z) \sigma_i)|^4} \geq C \sqrt{\kappa+\eta}.
         \end{equation}
         Otherwise, $E \in [E_-, E_+]^c$, we have
         \begin{align*}
         \Big| z-\frac{\gamma}{M} \sum_{i=1}^M \frac{ \sigma_i}{(1+\mathfrak{m}(z) \sigma_i)^2} \Big| & \geq \Big| |z|-\frac{\gamma}{M} \sum_{i=1}^M \frac{ \sigma_i}{|1+\mathfrak{m}(z) \sigma_i|^2} \Big|=\Big| \sqrt{E^2+\eta^2}-E -\frac{\Re \mr}{\Im \mr} \eta \Big|\\
         & \geq \frac{|\Re \mr|}{\Im \mr} \eta -\eta \geq C \sqrt{\kappa+\eta}-\eta \geq C' \sqrt{\kappa+\eta},
         \end{align*}
         if we choose $\tau'$ sufficiently small. The last second step follows from the fact that near the edge, $|\Re \mr| \geq C>0$, because of the second equation of (\ref{useful_sample}).
Next, if $z \in S^b$, then we have $\Im \mr >c$. We also split into two cases. If $\Re \mr > 0$, we repeat (\ref{repeat}) to get (\ref{coefficient}). If $\Re \mr \leq 0$, from (\ref{useful_sample}) we have
$$\frac{\gamma}{M} \sum_{i=1}^M \frac{ \sigma_i }{|1+\mathfrak{m} \sigma_i|^2}=E + \eta \frac{\Re \mr}{\Im \mr} \leq E.$$ 
In addition, we have
$$\Re \frac{1}{M} \sum_{i=1}^M \frac{\gamma \sigma_i}{(1+\mathfrak{m}(z) \sigma_i)^2} =\frac{\gamma}{M} \sum_{i=1}^M \frac{ \sigma_i [\Re (1+\mathfrak{m}(z) \sigma_i)^2-\Im (1+\mathfrak{m}(z) \sigma_i)]^2] }{|1+\mathfrak{m}(z) \sigma_i|^4} \leq E -C . $$
Therefore, we have
          \begin{align*}
          \Big|z-\frac{1}{M} \sum_{i=1}^M \frac{\gamma \sigma_i}{(1+\mathfrak{m}(z) \sigma_i)^2} \Big| \geq   E- \Re \Big( \frac{\gamma}{M} \sum_{i=1}^M \frac{ \sigma_i}{(1+\mathfrak{m}(z) \sigma_i)^2} \Big) \geq C \geq C\sqrt{\kappa+\eta}.
	  \end{align*}
Finally, taking the derivative of (\ref{mplawgamma-1}), we obtain that
	$$(z \mr)'=\mr+z \mr'=\frac{1}{M} \sum_{i=1}^M \frac{\gamma \sigma_i \mr'}{(1+ \sigma_i \mr)^2}.$$
	Hence, we finish the proof of (\ref{coefficient}), which directly implies (\ref{prime_gamma}).
\end{proof}

\end{document}